\documentclass[11pt]{amsart}

\usepackage{amsmath}
\usepackage{fullpage}
\usepackage{xspace}
\usepackage[psamsfonts]{amssymb}
\usepackage[latin1]{inputenc}
\usepackage{graphicx,color}
\usepackage[curve]{xypic} 
\usepackage{hyperref}
\usepackage{graphicx}

\usepackage{amsmath}%
\usepackage{amsthm}%
\usepackage{amscd}
\usepackage{amsfonts}%
\usepackage{amssymb}%
\usepackage{graphicx}

\usepackage{mathrsfs}

\usepackage{tikz}
\usetikzlibrary{matrix,arrows}

\newtheorem{theorem}{Theorem}[section]
\theoremstyle{plain}

\newtheorem{conjecture}[theorem]{Conjecture}
\newtheorem{corollary}[theorem]{Corollary}

\newtheorem{lemma}[theorem]{Lemma}

\newtheorem{proposition}[theorem]{Proposition}
\theoremstyle{definition}
\newtheorem{remark}[theorem]{Remark}

\newcommand{\Mfrak}{\mathfrak{M}}
\newcommand{\Nfrak}{\mathfrak{N}}
\newcommand{\Dfrak}{\mathfrak{D}}
\newcommand{\pfrak}{\mathfrak{p}}

\newcommand{\qfrak}{\mathfrak{q}}

\newcommand{\Ifrak}{\mathfrak{I}}
\newcommand{\nfrak}{\mathfrak{n}}

\newcommand{\hfrak}{\mathfrak{h}}

\newcommand{\Sfrak}{\mathfrak{S}}
\newcommand{\afrak}{\mathfrak{a}}

\newcommand{\dfrak}{\mathfrak{d}}
\newcommand{\Xfrak}{\mathfrak{X}}
\newcommand{\ffrak}{\mathfrak{f}}

\newcommand{\Acal}{\mathscr{A}}
\newcommand{\Bcal}{\mathscr{B}}

\newcommand{\Lcal}{\mathscr{L}}
\newcommand{\Mcal}{\mathscr{M}}
\newcommand{\Ncal}{\mathscr{N}}
\newcommand{\Ucal}{\mathscr{U}}
\newcommand{\Ecal}{\mathscr{E}}
\newcommand{\Ocal}{\mathscr{O}}

\newcommand{\Xcal}{\mathscr{X}}
\newcommand{\Fcal}{\mathscr{F}}

\newcommand{\Pcal}{\mathscr{P}}

\newcommand{\Gcal}{\mathscr{G}}
\newcommand{\Scal}{\mathscr{S}}
\newcommand{\Ccal}{\mathscr{C}}
\newcommand{\Jcal}{\mathscr{J}}
\newcommand{\Ycal}{\mathscr{Y}}

\newcommand{\Z}{\mathbb{Z}}
\newcommand{\C}{\mathbb{C}}
\newcommand{\D}{\mathbb{D}}
\newcommand{\F}{\mathbb{F}}
\newcommand{\Q}{\mathbb{Q}}
\newcommand{\R}{\mathbb{R}}

\newcommand{\A}{\mathbb{A}}
\newcommand{\B}{\mathbb{B}}
\newcommand{\G}{\mathbb{G}}
\newcommand{\T}{\mathbb{T}}
\newcommand{\I}{\mathbb{I}}

\newcommand{\E}{\mathbb{E}}

\newcommand{\ord}{\mathrm{ord}}
\newcommand{\End}{\mathrm{End}}

\newcommand{\rad}{\mathrm{rad}}
\newcommand{\Spec}{\mathrm{Spec}\,}

\newcommand{\Norm}{\mathrm{N}}
\newcommand{\rn}{\mathrm{rn}}
\newcommand{\Pic}{\mathrm{Pic}}
\newcommand{\Lie}{\mathrm{Lie}}
\newcommand{\coker}{\mathrm{coker}}
\newcommand{\Disc}{\mathrm{Disc}}

\begin{document}
\title{Shimura curves and the abc conjecture}
\author{Hector Pasten}
\address{ Department of Mathematics\newline
\indent Harvard University\newline
\indent Science Center\newline
\indent 1 Oxford Street\newline
\indent Cambridge, MA 02138, USA}
\email[H. Pasten]{hpasten@gmail.com}%
\date{\today}
\thanks{This research was partially supported by a Benjamin Peirce Fellowship (at Harvard) and by a Schmidt Fellowship and  NSF Grant DMS-1128155 (at IAS)}

\begin{abstract} 
We develop a general framework to study Szpiro's conjecture and the $abc$ conjecture by means of Shimura curves and their maps to elliptic curves, introducing  new techniques that allow us to obtain several unconditional results for these conjectures. We first prove various general results about modular and Shimura curves, including bounds for the Manin constant in the case of additive reduction, a detailed study of maps from Shimura curves to elliptic curves and comparisons between their degrees, and lower bounds for the Petersson norm of integral modular forms on Shimura curves. Our main applications for Szpiro's conjecture and the $abc$ conjecture include improved effective bounds for the Faltings height of elliptic curves over $\mathbb{Q}$ in terms of the conductor, bounds for products of $p$-adic valuations of the discriminant of elliptic curves over $\mathbb{Q}$ which are polynomial on the conductor, and results that yield a modular approach to Szpiro's conjecture over totally real number fields with the expected dependence on the discriminant of the field. These applications lie beyond the scope of previous techniques in the subject. A main difficulty in the theory is the lack of $q$-expansions, which we overcome by making essential use of suitable integral models and CM points. Our proofs require a number of tools from Arakelov geometry,  analytic number theory, Galois representations, complex-analytic estimates on Shimura curves, automorphic forms, known cases of the Colmez conjecture, and results on generalized Fermat equations.
\end{abstract}

\maketitle

\setcounter{tocdepth}{1}

\tableofcontents

\section{Introduction}

This work concerns Szpiro's conjecture, the $abc$ conjecture and related questions. Our main innovation is the introduction of several techniques related to Shimura curve parametrizations of elliptic curves in order to study these problems. As we will see, this new framework leads to unconditional results that lie out of the reach of the existing approaches in the literature. For our methods to work, we need to establish a number of general results of independent interest in the theory classical modular curves and Shimura curves. These include  boundedness of the Manin constant, almost-surjectivity of the map on component groups for Shimura curve parametrizations in the case of Cerednik-Drinfeld reduction, refinements of the Ribet-Takahashi formula, and lower bounds for the Petersson norm of integral quaternionic modular forms. 

In the context of Szpiro's conjecture and the $abc$ conjecture, the main applications of the theory developed in this work concern three aspects:

(1)\quad  \emph{Effective and explicit  unconditional estimates for the Faltings height of elliptic curves over $\Q$ in terms of the conductor}. Our estimates are stronger than the existing results in the literature. Effectivity with explicit constants is a relevant aspect, since in the past even weaker bounds have been useful for explicitly solving Diophantine equations in practice. We introduce a method to prove such estimates by only considering congruences of Hecke eigenforms (as opposed to congruences of general modular forms),  which not only leads to better unconditional estimates, but it also allows us to obtain further refinements under GRH. Shimura curves help us to reduce the number of such congruences that  must be considered, leading to further improvements.

(2)\quad  \emph{Bounds for products of valuations}. For coprime positive integers $a,b,c$ with $a+b=c$, the theory of linear forms in $p$-adic logarithms allows one to bound  all the $p$-adic valuations of the product $abc$, by a power of the radical $\rad(abc)$. Our methods, instead, provide such a polynomial bound for the \emph{product} of all these $p$-adic valuations. We also obtain analogous statements for elliptic curves, by showing that the product of the exponents of the minimal discriminant is bounded polynomially on the conductor. This gives us access to prove results outside the realm of known exponential versions of the $abc$ conjecture and Szpiro's conjecture that are currently available in the literature. For instance, we prove that the product of all $p$-adic Tamagawa factors of an elliptic curve over $\Q$ (the so-called ``fudge factor'' appearing in the Birch and Swinnerton-Dyer conjecture) is bounded  by a power of the conductor. We also give an upper bound in terms of the level for the number of  primes to which level-lowereing congruences (in the sense of Ribet) can occur, showing in a precise way that they are not too numerous. Unlike classical modular approaches to the $abc$ conjecture where one tries to bound the degree of a modular parametrization $X_0(N)\to E$ for an elliptic curve $E$, our strategy focuses on comparing the degree of maps $X\to E$ for a fixed elliptic curve $E$ and several Shimura curves $X$. Also, our method is global, in the sense that we work with the product of several local contributions simultaneously, unlike linear forms in $p$-adic logarithms. 

(3)\quad \emph{A modular approach over totally real fields}. We prove that the Faltings height of modular elliptic curves over totally real fields can be bounded in terms of the degree of modular parametrizations coming from Shimura curves, with a contribution from the discriminant of the number field agreeing with well-known conjectures. Bounds of this sort over $\Q$ (without the discriminant term)  are classical, but the usual argument does not extend beyond $\Q$ since the relevant Shimura curves do not have cusps. We overcome this difficulty by means of Arakelov geometry and the theory of Heegner points. Our results have two main consequences: On the one hand, they provide evidence for Vojta's conjecture on algebraic points of bounded degree in the aspect of the dependence on the logarithmic discriminant. On the other hand, we can use effective multiplicity one results for automorphic representations on $GL_2$ over number fields along with existing modularity results in order to bound the degree of such a parametrization, obtaining unconditional estimates in this context comparable to what is known over $\Q$.

After briefly formulating the main motivating problems, the rest of this introduction is devoted to present our main results and to give an outline of the paper. First we discuss our general results on arithmetic of modular curves and Shimura curves, and then we will present in more detail the applications in the context of Szpiro's conjecture and the $abc$ conjecture mentioned in (1), (2),  and (3)  above.


\subsection{The problems} Let us briefly state the motivating problems; we take this opportunity to introduce some basic notation. Precise details will be recalled in Section \ref{SecClassical}.

For an elliptic curve $E$ over $\Q$ we write $\Delta_E$ for the absolute value of its minimal discriminant and $N_E$ for its conductor. In the early eighties, Szpiro formulated the following conjecture:

\begin{conjecture}[Szpiro's conjecture; cf. \cite{SzpiroPresentation}] \label{ConjSzpiroIntro} There is a constant $\kappa>0$ such that for all elliptic curves $E$ over $\Q$ we have $\Delta_E < N_E^\kappa$.
\end{conjecture}

The radical $\rad(n)$ of a positive integer $n$ is defined as the product of the primes dividing $n$ without repetition. Let's recall here a simple version of the $abc$ conjecture of Masser and Oesterl\'e.

\begin{conjecture}[$abc$ conjecture] \label{ConjABCIntro} There is a constant $\kappa>0$ such that for all coprime positive integers $a,b,c$ with $a+b=c$ we have $abc<\rad(abc)^\kappa$.
\end{conjecture}

Both conjectures are open. There are stronger versions in the literature (cf. \cite{Oesterle}), but we keep these simpler formulations for the sake of exposition. 

A classical construction of Frey \cite{FreyLinks} shows that Szpiro's conjecture implies the $abc$ conjecture:  To a triple of coprime positive integers $a,b,c$ with $a+b=c$ one associates the Frey-Hellegouarch elliptic curve $E_{a,b,c}$  given by the affine equation $y^2=x(x-a)(x+b)$. Then $\Delta_E$ and $N_E$ are equal to $(abc)^2$ and $\rad(abc)$ respectively, up to a bounded power of $2$ (cf. Section \ref{SecClassical} for details and references). Thus, Szpiro's conjecture in the case of Frey-Hellegouarch elliptic curves implies the $abc$ conjecture as stated above. 

In the rest of this introduction we will freely use Landau's notation  $O$, as well as Vinogradov's notation $\ll$ (we include a reminder of these definitions in Section \ref{SecNotation}). In particular, $X\ll Y$ means the same as $X=O(Y)$. If the implicit constants are not absolute and depend on some parameters, this will be indicated by subscripts. 

If there is any risk of confusion, we will use the subscript ``$(?)$'' in an equality or inequality to indicate that the claimed expression is conjectural.


\subsection{Results on modular curves and Shimura curves} The notation in this paragraph is standard, and it will be recalled in detail in Section \ref{SecClassical} and Section \ref{SecPrelim}.

For a positive integer $N$ we have the modular curve $X_0(N)$ over $\Q$ with Jacobian $J_0(N)$ and a standard embedding $X_0(N)\to J_0(N)$ defined over $\Q$ induced by the cusp $i\infty$. Write $\hfrak$ for the complex upper half plane. The space of weight $2$ cuspforms for $\Gamma_0(N)$ is denoted by $S_2(N)$. The complex uniformization $\hfrak\to \Gamma_0(N)\backslash \hfrak\cup\{cusps\} = X_0(N)_\C$ induces an identification $S_2(N)\simeq H^0(X_0(N)_\C,\Omega^1_{X_0(N)_\C/\C})$.  Given a newform $f\in S_2(N)$ with rational Hecke eigenvalues and normalized by requiring that the first Fourier coefficient be $1$, we consider the associated optimal quotient $q:J_0(N)\to A$ over $\Q$, where $A$ is an elliptic curve of conductor $N$. So we get the modular parametrization $\phi: X_0(N)\to A$. The pull-back of a global N\'eron differential $\omega_A$ of $A$ to $\hfrak$ has the form $2\pi i c_f f(z)dz$ for a certain rational number $c_f$ called the \emph{Manin constant} of $f$, which we assume to be positive by adjusting the sign of $\omega_A$. It is known that $c_f$ is an integer \cite{Edixhoven} and it is conjectured that $c_f=1$. The latter is proved in the case when $A$ has semi-stable reduction (i.e. $N$ squarefree) by work of Mazur \cite{MazurRatIsog}, Abbes-Ullmo-Raynaud \cite{AbbesUllmo}, and Cesnavicius \cite{Cesnavicius}. See \cite{ARSManin} for further results and references. 

In the additive reduction case (i.e. $N$ is allowed to have repeated prime factors) not much is known. We can mention that some special cases of additive reduction at primes $p> 7$ have been considered by Edixhoven \cite{Edixhoven}, but in general it is not clear how to control the Manin constant beyond the semi-stable case. We prove a general result in the case of additive reduction: If additive reduction is restricted to a fixed finite set of primes, then $c_f$ is  uniformly bounded.
\begin{theorem}[Bounding the Manin constant; cf. Corollary \ref{CoroManinCt}] \label{ThmManinCtIntro} Let $S$ be a finite set of primes. There is a positive integer $\Mcal_S$ that only depends on $S$ such that the following holds:

Let $N$ be a positive integer which is squarefree away from $S$ (i.e. if $p^2|N$ for a prime $p$, then $p\in S$). Let $f\in S_2(N)$ be a Fourier normalized Hecke newform with rational Hecke eigenvalues. Then the associated Manin constant satisfies $c_f\le \Mcal_S$.
\end{theorem}
Besides the applications in the theory discussed in this work, the previous result also fills a gap in the literature concerning the equivalence of the \emph{height conjecture} and the \emph{modular degree conjecture} for elliptic curves over $\Q$, see Remark \ref{RmkHtDeg}.

Take $N$ a positive integer and $f\in S_2(N)$ a Fourier normalized Hecke newform with rational eigenvalues as above. A factorization $N=DM$ is \emph{admissible} if $D$ the product of an even number of different primes and $(D,M)=1$. For such a factorization we have a Shimura curve $X_0^D(M)$ over $\Q$ attached to the quaternion algebra of discriminant $D$, with level $\Gamma_0^D(M)$. The Jacquet-Langlands correspondence applied to $f$ gives an optimal quotient $q_{D,M}:J_0^D(M)\to A_{D,M}$ over $\Q$, where $J_0^D(M)$ is the Jacobian of $X_0^D(M)$ and $A_{D,M}$ is an elliptic curve isogenous to $A$ over $\Q$. For an abelian variety $B$ over $\Q$ and a prime $p$ we write $\Phi_p(B)$ for the group of geometric components of the special fibre at $p$ of the N\'eron model of $B$. At a prime $p|D$ the reduction of $J_0^D(M)$ is purely toric and the map $\Phi_p(J_0^D(M))\to \Phi_p(A_{D,M})$ induced by $q_{D,M}$  has been considered by a number of authors \cite{BertoliniDarmon, RiTa, Takahashi, PaRa}. The question of whether this map is surjective has been considered in \cite{PaRa} in the case when $N$ is squarefree but it remains open, despite the fact that it has been implicitly assumed as known elsewhere in the literature. In fact, in \cite{PaRa} some closely related constructions are given (outside the realm of modular and Shimura curves) where the analogous map is not surjective, which makes the problem rather delicate. We prove that this map is in fact almost-surjective, in the sense that  the size of its cokernel can be uniformly bounded independently of $N$.
\begin{theorem}[Bounds for the cokernel on component groups; cf. Theorem \ref{ThmUnifCoker}] \label{ThmCokerIntro} There is a constant $\kappa$ such that if $N$ is a squarefree positive integer, $N=DM$ is an admissible factorization where $M$ is not a prime number, and $q_{D,M}:J_0^D(M)\to A_{D,M}$ is an optimal elliptic curve quotient associated to a Fourier normalized Hecke newform $f\in S_2(N)$ with rational eigenvalues, then for each prime $p|D$ we have that
$$
\#\mathrm{cokernel}(\Phi_p(J_0^D(M))\to \Phi_p(A_{D,M}))\le \kappa.
$$
\end{theorem}
Beyond the semi-stable case, we also prove such a uniform bound for the cokernel in the case of Frey-Hellegouarch elliptic curves, see Theorem \ref{ThmUnifCoker}. The proof of our cokernel bounds has a somewhat unusual arithmetic input: We need various results on generalized Fermat equations, including those of Darmon and Granville \cite{DarmonGranville} on the equation $Ax^p+By^q=Cz^r$.

Our interest in surjectivity of the induced map on components in the case of $p|D$, comes from the Ribet-Takahashi formula \cite{RiTa} which we now recall.

The endomorphism $q_{D,M}\circ q_{D,M}^\vee\in \End(A_{D,M})$ is multiplication by a positive integer that we denote by $\delta_{D,M}$. When $D=1$ one has $\delta_{1,N}=\deg(\phi : X_0(N)\to A)$ and in general one can relate $\delta_{D,M}$ to the degree of a suitably chosen morphism $X_0^D(M)\to A_{D,M}$, although it is not exactly equal. The precise comparison of $\delta_{D,M}$ with the degree of a suitably constructed modular parametrization $X_0^D(M)\to A_{D,M}$ over $\Q$ is achieved in Corollary \ref{CoroDegApproxQ}.

Ribet and Takahashi \cite{RiTa} compared $\delta_{1,N}$ to its quaternionic counterpart $\delta_{D,M}$ for general $D$ provided that $M$ is squarefree and it is not a prime. The formula is 
$$
\frac{\delta_{1,N}}{\delta_{D,M}}=\gamma \cdot \prod_{p|D}v_p(\Delta_A)
$$
where $\gamma$ is a certain rational number possibly depending on all the data, and supported only at primes $\ell$ where the Galois module $A[\ell]$ is reducible. For our applications it is crucial to control the error factor $\gamma$ at all primes, even in cases where some Galois modules $A[\ell]$ are reducible. More precisely, we need to show that $\gamma$ is a rational number of ``small'' height, since one of the main goals of this work is to give global estimates. Among other technical tools, Theorem \ref{ThmCokerIntro} (and more generally, Theorem \ref{ThmUnifCoker}) allows us to do so.

\begin{theorem}[Global comparison of $\delta_{1,N}$ and $\delta_{D,M}$; cf. Corollary \ref{CoroRT}]  \label{ThmRTintro} With the previous notation, suppose that either 
\begin{itemize}
\item[(i)] $N$ is squarefree and $M$ is not a prime number; or
\item[(ii)] $A$ is isogenous to a Frey-Hellegouarch elliptic curve and $M$ is divisible by at least two odd prime numbers.
\end{itemize}
Then we have 
$$
 \log\left( \prod_{p|D} v_p(\Delta_A) \right) = \log \delta_{1,N} -  \log\delta_{D,M} +  O\left(\frac{\log D}{\log \log D}\right)
$$
where the implicit constant is absolute (in particular, independent of $N$ or $A$).
\end{theorem}

We refer to Theorem \ref{ThmRT} for a more general result allowing additive reduction at a given finite set of primes.

Let us now mention a result which is intrinsic about the arithmetic of quaternionic modular forms, independet of the discussion on modular parametrizations of elliptic curves. Consider a positive integer $N$ and an admissible factorization $N=DM$ with $D>1$ so that we are not in the  case of the classical modular curve $X_0(N)$. Then $X_0^D(M)$ does not have cuspidal points and therefore, its modular forms do not have a Fourier expansion. Let us focus on $S_2^D(M)$, the space of weight $2$ modular forms on $\hfrak$ associated to $X_0^D(M)$. As in the case of $X_0(N)$, we have an identification $S_2^D(M)\simeq H^0(X_0^D(M)_\C,\Omega^1_{X_0^D(M)_\C/\C})$.  By lack of Fourier expansion, a natural definition of ``integral'' modular form can be given as follows. 

The curve $X_0^D(M)$ has  a standard integral model $\Xcal_0^D(M)$ which is flat and projective over $\Z$, coming from certain moduli problem of ``fake elliptic curves''. We define a modular form in $S_2^D(M)$ to be \emph{integral} if, under the previous identification, it belongs to $H^0(\Xcal_0^D(M)^0, \Omega^1_{\Xcal_0^D(M)/\Z})$ where  $\Xcal_0^D(M)^0$ is the smooth locus of $\Xcal_0^D(M)\to \Spec \Z$ (the locus where the relative differentials give an invertible sheaf). We denote by $\Scal_2^D(M)$ the $\Z$-module of integral modular forms in $S_2^D(M)$. 

The space $S_2^D(M)$ has a standard norm given by the (non-normalized) Petersson inner product, and a natural question of arithmetic relevance is: \emph{How small is the shortest non-zero vector in $\Scal_2^D(M)$?} The analogous question for $D=1$ can be approached by the $q$-expansion principle: Integrality on a formal neighborhood of the cusp $i\infty$ agrees with integrality of the Fourier expansion at this cusp, and the latter has direct implications for the problem of giving lower bounds for the Petersson norm. In absence of Fourier expansions the problem is much more difficult. We prove that the shortest non-zero vector of $\Scal_2^D(M)$ cannot be too small in terms of the level; more precisely, we prove a lower bound which is reciprocal of a power of $N=DM$, with no semi-stability restriction.

\begin{theorem}[Integral forms are not too small; cf. Theorem \ref{ThmIntegralityBound}] \label{ThmIntegralityIntro} Given  $\epsilon>0$ there is a constant $C_\epsilon>0$ depending only on $\epsilon$ such that the following holds:  

Let $N$ be a positive integer and consider an admissible factorization $N=DM$  with $D>1$. Every non-zero $h\in \Scal_2^D(M)$ has Peterson norm bounded from below by
$$
\frac{C_\epsilon}{ N^{\frac{5}{6}+\epsilon}\cdot M^{1/2}}.
$$
\end{theorem}
Our method extends to higher weight forms. We only consider the case of weight $2$ to simplify the exposition and because this is the case relevant for our intended applications.

The proof falls into the context of Arakelov geometry. First we compare the Petersson norm (which is an $L^2$-norm) to an $L^\infty$-norm, which requires lower bounds for the injectivity radius of the complex curve $X_0^D(M)_\C$ (cf. Section \ref{SecNorms}). Then we reduce the problem to finding small values of integral modular forms, which in turn reduces to the problem of finding algebraic points which are conveniently located in the arithmetic surface $\Xcal_0^D(M)$, and at the same time have small Arakelov height with respect to certain arithmetic Hodge bundle endowed with a hyperbolic metric. We construct the required points as Heegner points, for which height formulas are available when $M=1$ (by results of Kudla-Rapoport-Yang \cite{KRY, KRYbook} and Yuan-Zhang \cite{YuZh} in the context of Colmez's conjecture). This can be suitably extended to our case (general $M$) by a local analysis of integral models (cf. Section \ref{SecFinitePart}). However, the precise choice of Heegner points is not straightforward and it involves analytic number theory  (cf. Section \ref{SecCountingFields}) in order to ensure that the relevant points are in fact conveniently located (sieve theory) and that the height is small (averaging and zero-density estimates for $L$-functions due to Heath-Brown).

Most of the previous analysis regarding norms and Heegner points will also be  carried out in more generality for Shimura curves over totally real fields, as this is necessary for our study of Szpiro's conjecture over such fields in Section \ref{SecModApproachF}. The appendix (by R. Lemke Oliver and J. Thorner) provides a suitable zero-density estimate in this context.


\subsection{Effective height bounds and congruences} \label{Par1} The modularity theorem for elliptic curves (cf.  \cite{Wiles, TaylorWiles, BCDT}) opens the possibility of studying Szpiro's conjecture over $\Q$ by means of modular forms. Such an approach was already considered in the eighties by Frey \cite{FreyLinks}; let us briefly recall this classical approach (cf. Section \ref{SecClassical} for details).  

Let $E$ be an elliptic curve of conductor $N=N_E$. By the modularity theorem, there is an optimal modular parametrization $\phi : X_0(N)\to A$ for a certain elliptic curve $A$ isogenous to $E$ over $\Q$. A computation shows that $\log \Delta_E \le 12 h(E) +16$ where $h(E)$ is the Faltings height of $E$, and that  
\begin{equation}\label{EqModAppIntro}
h(E) \le  \frac{1}{2}\log \deg \phi +9.
\end{equation}
Hence, Szpiro's conjecture would follow if one can show that the modular degree $\delta_{1,N}=\deg \phi$ is bounded polynomially  on the conductor $N$. 

This approach was used by Murty and Pasten \cite{MurtyPasten} to give explicit and effective bounds of the form 
\begin{equation}\label{EqMPeff}
h(E)\ll N\log N
\end{equation}
with good implicit constants, by first bounding the modular degree $\delta_{1,N}$. This leads to effective estimates for solutions of the $S$-unit equation, Mordell's equation, and other diophantine problems. See also \cite{vonKanel}. These estimates are not only of theoretical interest; recently, they have been given a practical implementation by von K\"anel and Matschke \cite{vKM}. 

Similarly, given an elliptic curve $E$ over $\Q$ and an admissible factorization $N=DM$ of its conductor, we have the optimal quotients $q_{D,M}:J_0^D(M)\to A_{D,M}$ with $A_{D,M}$ isogenous to $E$ over $\Q$, and the associated quantities $\delta_{D,M}=\delta_{D,M}(E)$. A natural question is whether bounds for the numbers $\delta_{D,M}(E)$ are useful in the study of Szpiro's conjecture, beyond the classical case $D=1$. We show that this is indeed the case, despite the fact that for classical modular parametrizations the argument heavily uses $q$-expansions at cuspidal points of $X_0(N)$, while cuspidal points and $q$-expansions are not available when $D>1$.
\begin{theorem}[cf. Theorem \ref{ThmShimuraApproachQ}] \label{ThmSAI} Let $\epsilon>0$. For all elliptic curves $E$ of conductor $N\gg_\epsilon 1$ (with an effective implicit constant), and for any admissible factorization $N=DM$  we have
$$
\log |\Delta_E| <(6+\epsilon)\log \delta_{D,M}(E)  \quad \mbox{and}\quad   h(E) < \left(\frac{1}{2}+\epsilon\right)\log \delta_{D,M}(E).
$$
\end{theorem}
This result is obtained from Theorem \ref{ThmRT}, which is a stronger and more precise version of Theorem \ref{ThmRTintro}. The motivation for using $\delta_{D,M}$ instead of simply using $\delta_{1,N}$ is practical: For an admissible factorization $N=DM$, the curve $X_0^D(M)$ usually has genus smaller than that of $X_0(N)$, which can be expected to lead to bounds for $\delta_{D,M}$ of better quality than for $\delta_{1,N}$. In order to make this expectation precise we need to discuss congruences of modular forms. 

Classically, to bound the modular degree $\delta_{1,N}$ one relies on a theorem of Ribet \cite{Zagier, CoKa, ARSdeg}. Let $f\in S_2(N)$ be the Hecke newform attached to an elliptic curve $E/\Q$ of conductor $N$, and let $m_f$ be the largest positive integer such that for some $g\in S_2(N)$ orthogonal to $f$ (with respect to the Petersson inner product) with Fourier coefficients in $\Z$, we have that $f\equiv g\mod m_f$ in terms of Fourier expansion.  The number $m_f$ is called the congruence modulus of $f$ and Ribet's theorem says that $\delta_{1,N}|m_f$. Thus, an approach to bound $\delta_{1,N}$ consists of controlling possible congruences of $f$ with other modular forms in $S_2(N)$ with Fourier coefficients in $\Z$ ---this is precisely how the estimate \eqref{EqMPeff} is proved in \cite{MurtyPasten}. 

It is natural to expect that a similar approach should allow one to bound $\delta_{D,M}$ in terms of congruences, which in principle should lead to better bounds since the Jacquet-Langlands transfer gives an isomorphism of $S_2^D(M)$ with the $D$-new part of $S_2(N)$ (as opposed to the whole space). Unfortunately this is not clear: The lack of Fourier expansions when $D>1$ leaves us with no obvious analogue of Ribet's theorem.

We provide an alternative method based on congruences of systems of Hecke eigenvalues (thus only involving eigenforms) instead of Fourier expansions of all modular forms orthogonal to a given $f\in S_2(N)$. For an admissible factorization $N=DM$ we let $\T_{D,M}$ be the Hecke algebra acting on $S_2^D(M)$, and for a system of Hecke eigenvalues $\chi:\T_{D,M}\to \bar{\Q}$ we let $[\chi]$ be the class of its Galois conjugates. Given a $\Z$-valued system of Hecke eigenvalues $\chi_0:\T_{D,M}\to \Z$ and a different system of Hecke eigenvalues $\chi:\T_{D,M}\to \bar{\Q}$, we define the congruence modulus $\eta_{[\chi_0]}([\chi]):=[\Z : \chi_0(\ker(\chi))]$; this is a well-defined positive integer that measures congruences between $\chi_0$ and $\chi$ in a precise way (cf.  Proposition \ref{PropBoundCongr},  which also gives a way to estimate these numbers). 
\begin{theorem}[Hecke eigenvalue congruences; cf. Theorem \ref{ThmSpectral}] \label{ThmSpectralIntro} Let $\chi_0:\T_{D,M}\to \Z$ be the system of Hecke eigenvalues attached to an elliptic curve $E/\Q$ of conductor $N$, where $N=DM$ is an admissible factorization. The number $\delta_{D,M}$ divides 
$$
\prod_{[\chi]\ne [\chi_0]}\eta_{[\chi_0]}([\chi]).
$$
The product runs over all classes $[\chi]$ of systems of Hecke eigenvalues on $\T_{D,M}$ different from $[\chi_0]$.
\end{theorem}
This result is related to the notion of ``primes of fusion'' of Mazur, but for our purposes it is not enough to only have information of the primes that support these congruences. 

The previous theorem gives a way to estimate $\delta_{D,M}$ and thus to apply Theorem \ref{ThmSAI}. Let us state here the case of restricted additive reduction (this is of practical interest ---for instance, applications to $S$-unit equations only need to allow additive reduction at $2$). See Section \ref{SecBounds} for more precise and general results. Here, $\varphi$ is Euler's function.

\begin{theorem} (cf. Corollary \ref{CoroBoundSSQ}) Let $S$ be a finite set of primes and let $P$ be the product of the elements of $S$. For $\epsilon>0$ and $N\gg_{\epsilon,S} 1$ (with an effective implicit constant), if $E$ is an elliptic curve over $\Q$ semi-stable away from $S$, then we have
$$
h(E)<\begin{cases}
\frac{P}{\varphi(P)}\left(\epsilon + 1/48\right)\varphi(N)\log N & \mbox{unconditional} \\

\frac{P}{\varphi(P)}\left(\epsilon + 1/24\right)\varphi(N)\log \log N & \mbox{under GRH}.
\end{cases}
$$
If $S=\emptyset$ then the factor $P/\varphi(P)$ is $1$, and for $P$ large enough one has $P/\varphi(P)<2\log\log P$.
\end{theorem}

We observe that these estimates are better than \eqref{EqMPeff}, and that we were able to get an improvement under GRH. This last feature is due to the fact that Theorem \ref{ThmSpectralIntro} concerns congruences of Hecke eigenforms (whose $L$-functions and Rankin-Selberg $L$-functions are expected to satisfy GRH) while the method of using Ribet's congruence modulus $m_f$ is not well-suited for an application of GRH as it concerns modular forms that are not necessarily eigenforms. 

Although the classical case $D=1$ is not the main intended application of  Theorem \ref{ThmSpectralIntro}, let us mention what happens here.  When $D=1$, our result gives better numerical estimates than by using Ribet's theorem on the congruence modulus $m_f$ (the method used in \cite{MurtyPasten}). Since our estimates in the case $D=1$ are superior to the bounds recently used in \cite{vKM} for solving a number of Diophantine equations, we spell out the precise bounds in Theorem \ref{ThmBoundAllQ}. For the moment, we mention that we obtain bounds of the form 
$$
h(E)< \left(\frac{1}{48} + \epsilon \right) N\log N
$$
while the estimates in \cite{MurtyPasten} with the improvements of \cite{vKM} take the form
$$
h(E) < \left(\frac{1}{16} + \epsilon \right) N\log N.
$$


\subsection{Product of valuations} For a prime number $p$, we let $v_p:\Q^\times \to \Z$ be the $p$-adic valuation.

One of the main consequences of the new techniques introduced in this work is the fact that we provide a natural framework for proving global estimates on the products of $p$-adic valuations. For instance, in the context of the $abc$ conjecture we obtain the following unconditional result.
\begin{theorem}[Product of valuations for $abc$ triples; cf. Theorem \ref{ThmMainABC}] \label{ThmValABCIntro} Let $\epsilon >0$. There is a number $K_\epsilon>0$ depending only on $\epsilon$ such that for all coprime positive integers $a,b,c$ with $a+b=c$ we have
$$
\prod_{p|abc}v_p(abc)< K_\epsilon \cdot \rad(abc)^{\frac{8}{3}+\epsilon}.
$$
\end{theorem}
For comparison purposes, let us recall that the strongest unconditional results on the $abc$ conjecture available in the literature have been obtained from the theory of linear forms in ($p$-adic) logarithms and take the form
\begin{equation}\label{EqMaxVal}
\max_{p|abc} v_p(abc) < K_\epsilon\cdot \rad(abc)^{\alpha+\epsilon}
\end{equation}
for a fixed positive exponent $\alpha$, where the value $\alpha=15$ was obtained by Stewart-Tijdeman in 1986 \cite{ABC1}, $\alpha=2/3$ was obtained by Stewart-Yu in 1991 \cite{ABC2}, and $\alpha=1/3$ was obtained by Stewart-Yu in 2001 \cite{ABC3}. Theorem \ref{ThmValABCIntro} on the other hand gives a bound for the \emph{product} (instead of the maximum) of the $p$-adic valuations in terms of a power of the radical. This is an intrinsic feature of our method as it naturally leads to global estimates with simultaneous contribution of several primes, as opposed to linear forms in $p$-adic logarithms where the contribution of each prime needs to be studied independently. Our result seems beyond the scope of what transcendental methods (linear forms in logarithms) can prove; see Paragraph \ref{SecHeuristic} for details on this comparison, as well as the comments after Theorem \ref{ThmTamFudgeIntro}.

An equivalent (and more elementary) formulation of Theorem \ref{ThmValABCIntro} is the following.
\begin{theorem}[``$d(abc)$-Theorem''; cf.  Theorem \ref{ThmMainABC}] \label{ThmdABCIntro} Let $\epsilon >0$. There is a number $K_\epsilon>0$ depending only on $\epsilon$ such that for all coprime positive integers $a,b,c$ with $a+b=c$ we have
$$
d(abc)< K_\epsilon \cdot \rad(abc)^{\frac{8}{3}+\epsilon}.
$$
Here, $d(n)$ is the number of divisors of a positive integer $n$. 
\end{theorem}

The bound \eqref{EqMaxVal} provided by transcendental methods is equivalent to $\log c\ll_\epsilon \rad(abc)^{\alpha+\epsilon}$ (with the same exponent $\alpha$) due to the $\epsilon$ in the exponent, and it is often stated in this way. Thus, it might also be instructive to compare the quantities $\log c$ and $d(abc)$ not just on theoretical grounds, but on actual examples instead. The next table records this comparison for the top five highest quality (in the standard terminology) known $abc$ triples \cite{dSmit}:
$$
\begin{array}{c|c|c||c|c}
a & b & c & \log c& d(abc)\\ \hline
2 & 3^{10}109 & 23^5& 15.677...& 264 \\
11^2 & 3^2 5^6 7^3 & 2^{21}23& 17.691... & 11\,088\\
19\cdot 1307& 7\cdot 29^2 31^8 & 2^8 3^{22} 5^4 & 36.152... & 223\,560 \\
283 & 5^{11}13^2 & 2^8 3^8 17^3 & 22.833...& 23\,328 \\
1 & 2\cdot 3^7 & 5^4 7 & 8.383...& 160.
\end{array}
$$

Our methods also allow us to prove analogous results for elliptic curves, see Section \ref{SecProdVal}. For instance, we obtain
\begin{theorem}[Product of valuations for elliptic curves; cf. Theorem \ref{ThmProdSS}] \label{ThmValSSIntro}  Let $\epsilon>0$. There is a number $K_{\epsilon}>0$ depending only on $\epsilon$  such that for every semi-stable elliptic curve $E$ over $\Q$ we have
$$
\prod_{p|N_E}v_p(\Delta_E)< K_{\epsilon} \cdot N_E^{\frac{11}{2}+\epsilon}.
$$ 
\end{theorem}
We refer the reader to Section \ref{SecProdVal} for this and other more general estimates. In fact, we also obtain similar results when additive reduction is allowed on a fixed finite set of primes, which is necessary in the proof of Theorem \ref{ThmValABCIntro}.

Before outlining  an idea of how these results are proved, let us discus some applications. 

First, let us consider Ribet's level-lowering theory \cite{RibetInv100}. For an elliptic curve $E$ over $\Q$ and its associated newform $f\in S_2(N_E)$, this theory applies to primes $\ell$ dividing some exponent in the prime factorization of $\Delta_E$ (under some additional assumptions such as irreducibility of the Galois module $E[\ell]$) and it is interesting to know how many of these primes can an elliptic curve have. It turns out that not too many:
\begin{theorem}[Counting level-lowering primes; cf. Corollary \ref{CoroLL}] For an elliptic curve $E$ over $\Q$, let 
$$
L(E):=\{\ell \mbox{ prime }: \exists p\mbox{ prime such that }p|N_E\mbox{ and }\ell | v_p(\Delta_E)\}.
$$
There is a constant $c$ such that for all semi-stable elliptic curves $E$ over $\Q$ we have
$$
\# L(E) < \frac{6\log N_E}{\log \log N_E} +c.
$$
\end{theorem}

A second application concerns a folklore (and widely believed) conjecture. For a prime number $p$ and an elliptic curve $E$ over $\Q$, the local Tamagawa factor is $\mathrm{Tam}_p(E)=[E(\Q_p):E^0(\Q_p)]$ where $E^0(\Q_p)$ is the subgroup of points of $E(\Q_p)$ that extend to $\Z_p$-sections on the connected component of the N\'eron model of $E$ over $\Z_p$.  The global Tamagawa factor is then defined as $\mathrm{Tam}(E)=\prod_p \mathrm{Tam}_p(E)$ (usually called ``fudge factor'' in the context of the Birch and Swinnerton-Dyer conjecture). 
\begin{conjecture}[Fudge factors are small; folklore] \label{ConjFudge} Let $\epsilon>0$. There is a constant $K_\epsilon$ depending only on $\epsilon$ such that for all elliptic curves $E$ over $\Q$ we have $\mathrm{Tam}(E)< K_{\epsilon} \cdot N_E^\epsilon$.
\end{conjecture}
This conjecture often plays a role in the analysis of  the special value formula in the Birch and Swinnerton-Dyer conjecture, and it is known that it follows from Szpiro's conjecture. More precisely, de Weger \cite{deWeger} and Hindry \cite{Hindry} have observed the following: A local analysis of the numbers $\mathrm{Tam}_p(E)$ shows that $\mathrm{Tam}(E)\le d(\Delta_E)^2$. Then the elementary divisor bound $d(n)<\exp(\kappa (\log n)/\log \log n)$ (for a suitable constant $\kappa>0$; this is sharp up to the choice of $\kappa$) gives
$$
\mathrm{Tam}(E) <\exp\left(\frac{2\kappa \log \Delta_E}{\log \log \Delta_E}\right).
$$ 
Hence, Conjecture \ref{ConjFudge} would follow from a (conjectural) estimate of the form 
$$
\log \Delta_E =_{(?)} o\left((\log N_E)\log \log N_E\right).
$$ 
In particular, it would follow from Szpiro's conjecture.

In unconditional grounds, our results on products of valuations allow us to prove:
\begin{theorem}[Polynomial-in-conductor bound for the fudge factor; cf. Corollary \ref{CoroTamSa1} and Theorem \ref{ThmProdSS}] \label{ThmTamFudgeIntro} Let   $S$ be a finite set of primes and let $\epsilon>0$. There is a number $K_{S,\epsilon}>0$ depending only on $\epsilon$ and $S$ such that for every elliptic curve $E$ over $\Q$ satisfying
\begin{itemize} 
\item[(i)] is semi-stable away from $S$, and
\item[(ii)] $E$ has at least two primes of multiplicative reduction, 
\end{itemize}
we have
$$
\mathrm{Tam}(E)< K_{S,\epsilon} \cdot N_E^{\frac{11}{2}+\epsilon}.
$$ 
Furthermore, when $S=\emptyset$ (i.e. for semi-stable elliptic curves) assumption (ii) can be dropped.
\end{theorem}
We observe that if instead of using our results on products of valuations one wants to prove Theorem \ref{ThmTamFudgeIntro} by means of the de Weger-Hindry approach, then one would need partial results for Szpiro's conjecture of the form
$$
\log \Delta_E \le_{(?)} \left(\frac{11}{2} +\epsilon\right)(\log N_E)\log \log N_E + O_\epsilon(1).
$$
At present, this is far beyond reach; such an estimate would be \emph{exponentially} stronger than anything available today! In fact, the best available unconditional bounds for $\log \Delta_E$ in terms of $N_E$ are of the form $\log \Delta_E \ll_\epsilon N_E^{\alpha+\epsilon}$ for some fixed $\alpha>0$ (currently, the record is $\alpha=1$ for all elliptic curves over $\Q$ \cite{MurtyPasten}, or $\alpha=1/3$ for Frey-Hellegouarch elliptic curves \cite{ABC3}). It should be noted that Conjecture \ref{ConjFudge} is not a  mild conjecture;  it implies a sub-exponential version of Szpiro's conjecture of the form $\log \Delta_E\ll_\epsilon N_E^\epsilon$, which is currently an open problem.

Let us now give an idea of the technique used to prove our estimates on products of valuations. We recall that the classical modular approach to Szpiro's conjecture only focuses on modular parametrizations of the form $X_0(N)\to E$, while our results in Paragraph \ref{Par1} consider, more generally, maps of the form $X_0^D(M)\to E$. However, here we consider a very different strategy. Namely, instead of trying to bound the degree of a map $X\to E$ in terms of the conductor of $E$, we consider several maps of this form for a fixed elliptic curve $E$ while we let $X$ vary over Shimura curves. Then our aim is to bound the \emph{variation} of the degrees of such maps, not the degrees themselves. 

Although at present it seems that bounding each $\delta_{D,M}$ polynomially on the conductor $N$ is out of reach (in fact, such a bound would imply Szpiro's conjecture! cf. Theorem \ref{ThmSAI}), it turns out that the comparison ratios $\delta_{1,N}/\delta_{D,M}$ are more accessible and we can unconditionally bound them polynomially on $N$. Furthermore, these bounds give valuable arithmetic information thanks to our extension of the Ribet-Takahashi formula (cf. Theorem \ref{ThmRTintro} and more generally Theorem \ref{ThmRT}).  This is how we approach the problem of bounding products of valuations.

The previous idea, however, only gives a starting point. The precise way in which we bound the ratios $\delta_{1,N}/\delta_{D,M}$ is rather delicate: First we express this ratio as a ratio of Petersson norms, which inevitably comes with a contribution of the Manin constant (including cases of additive reduction such as for Frey-Hellegouarch elliptic curves). The latter is controlled by our Theorem \ref{ThmManinCtIntro}. Crucially, we need to show that the Petersson norm of certain quaternionic modular form is not too small, and this is done by Theorem \ref{ThmIntegralityIntro}, which, as explained before, requires a number of tools and techniques including Arakelov geometry, height formulas for Heegner points, and analytic number theory.


\subsection{Totally real case} A limitation of the classical modular approach to Szpiro's conjecture by means of bounding the degree of modular parametrizations of elliptic curves, is that so far it is only available over $\Q$, while Szpiro's conjecture can also be formulated over number fields. 

Namely, given a number field $F$ and an elliptic curve $E$ over $F$ we let $\Nfrak_E$ be the conductor ideal of $E$ and we let $N_E$ be its norm. Also, we let $\Delta_E$ be the norm of the discriminant ideal of $E$ (this is compatible with our earlier notation in the case $F=\Q$).

\begin{conjecture}[Szpiro \cite{SzpiroPresentation}]  Let $F$ be a number field. There is are constants $c, \kappa>0$ depending only on $F$ such that for all semi-stable elliptic curves $E$ over $F$ we have $\Delta_E < c\cdot N_E^\kappa$. 
\end{conjecture}

Following Szpiro's formulation in \cite{SzpiroPresentation}, we only state this conjecture for semi-stable elliptic curves, but of course one can make a similar conjecture without that assumption.

One can also ask about the dependence of $c$ and $\kappa$ on $F$, and Szpiro \cite{SzpiroPresentation} also suggests that any fixed $\kappa>6$ should work, in which case $c$ depends on the choice of $\kappa$ and on the number field $F$ in some unspecified manner. For a number field $F$ we let $d_F$ be the absolute value of its discriminant. General conjectures of Vojta \cite{VojtaLNM, VojtaCIME} suggest that one should be able to take $c=c(F)$ as a power of $d_F$ provided that $[F:\Q]$ remains bounded. Parshin \cite{Parshin} proved that such a dependence on $F$ would follow from a conjectural Bogomolov-Miyaoka-Yau inequality in arithmetic. See also \cite{Hindry} where Hindry notes that this expectation would also follow from a generalization of Szpiro's conjecture to higher dimensional abelian varieties and restriction of scalars.

Our methods allow us to provide a modular approach to Szpiro's conjecture over totally real number fields analogous to the estimate \eqref{EqModAppIntro} and to Theorem \ref{ThmSAI}. 
\begin{theorem}[Modular approach to Szpiro's conjecture over totally real number fields; cf. Theorem \ref{ThmSzpiroF}] \label{ThmSAtrIntro}
Let $n$ be a positive integer. There is a constant $c_n>0$ only depending on $n$ such that the following holds:

Let $F$ be a totally real number field with $[F:\Q]=n$. Let $E$ be a semi-stable elliptic curve over $F$ which is modular over $F$ and assume that the number of primes of bad reduction of $E$ has opposite parity to $n$. Let $X$ be the Shimura curve over $F$ attached to an indefinite $F$-quaternion algebra of discriminant $\Nfrak_E$ (which exists by the parity assumption) and let $\phi :X\to E$ be a non-constant morphism over $F$ (which exists by the modularity assumption). Let $h(E)$ be the Faltings height of $E$. Then we have
$$
h(E) < \frac{1}{2}\log (\deg \phi) + c_n\cdot \left(\log N_E + \log d_F\right)
$$
and
$$
\frac{1}{n}\log(\Delta_E) < 6\log (\deg \phi) + c_n\cdot \left(\log N_E + \log d_F\right).
$$
\end{theorem}

This reduces to Szpiro's conjecture over $F$ to an appropriate upper bound for the minimal degree of maps of the form $X\to E$, thus providing an approach to Szpiro's conjecture over number fields other than $\Q$ by means of modularity techniques.  The precise bound that we expect for the modular degree in this setting is formulated as Conjecture \ref{ConjModDegTR}.

Furthermore,  the dependence on $F$ that we obtain in Theorem \ref{ThmSAtrIntro} is precisely through $d_F$ in the expected way, which can be regarded as evidence towards this aspect of Szpiro's conjecture over a varying field $F$ of bounded degree and Vojta's conjecture for algebraic points of bounded degree. 

Let us make two technical remarks about the assumptions in Theorem \ref{ThmSAtrIntro}.
\begin{itemize}
\item The parity assumption might be removed with additional technical work by considering Shimura curves associated to non-maximal compact open subgroups ---for the sake of simplicity we only consider the maximal case here, but the necessary technical work is similar to our computations for the case $F=\Q$, done in Section \ref{SecFinitePart}. Alternatively, one can base change to a suitable auxiliary quadratic extension.

\item The modularity assumption is, by now, a rather reasonable working hypothesis, thanks to the spectacular available progress on this matter. For instance, now it is known that all elliptic curves over real quadratic fields are modular \cite{FLS}.
\end{itemize}

We implement the previous two remarks in a concrete case, obtaining:
\begin{theorem}[cf. Theorem \ref{ThmQuadraticCond}] \label{ThmQuadraticCondIntro} If the bound for the modular degree formulated in Conjecture \ref{ConjModDegTR} holds for modular elliptic curves over real quadratic fields and their totally real quadratic extensions, then there is an absolute constant $c>0$ such that the following holds:

For every real quadratic field $F$ and every semi-stable elliptic curve $E$ over $F$ which does not have everywhere good reduction, we have
$$
h(E)\le c \cdot \log (d_FN_E) \quad \mbox{ and }\quad  \Delta_{E}\le (d_FN_E)^{c}.
$$
\end{theorem}
Note that the conclusion of the previous result does not require $E$ to be modular; as we will see, modularity in the necessary cases holds by \cite{FLS} and \cite{DieFre}.

A modular parametrization $\phi$ as in Theorem \ref{ThmSAtrIntro}  does not need to be minimal (its degree appears as an upper bound). The existence of these modular parametrizations  under the assumption of modularity (for suitable conductors) is well-known in the context of the Birch and Swinnerton-Dyer conjecture, see for instance \cite{YZZbook}. For our purposes, however, it is desirable to make the degree of modular parametrizations as small as possible. We give such an economic construction in Theorem \ref{ThmDegreeTR} and Theorem \ref{ThmDegreeCompF}. 

Theorems \ref{ThmDegreeTR} and  \ref{ThmDegreeCompF} moreover compare $\deg \phi$ to a certain number $\delta_{E}$ which is analogous to the numbers $\delta_{D,M}$ from our discussion over $\Q$. This is relevant because our Theorem \ref{ThmSpectralIntro} admits a straightforward generalization to the setting of modular parametrizations over totally real fields, which allows us to bound $\delta_E$ in terms of congruences of systems of Hecke eigenvalues. From here, together with results on effective multiplicity one for $GL_2$,  we can obtain \emph{unconditional} bounds such as the following:
\begin{theorem}[cf. Theorem \ref{ThmFinalApp}] \label{ThmFinalIntro} Let $F$ be a totally real number field and let $\epsilon>0$.  For all semi-stable elliptic curves $E$ over $F$ satisfying that the number of places of bad reduction of $E$ has opposite parity to $[F:\Q]$, we have $\log \delta_E\ll_{F,\epsilon} N_E^{1+\epsilon}$, hence
$$
h(E) \ll_{F,\epsilon} N_E^{1+\epsilon} \quad \mbox{ and }\quad \log \Delta_E \ll_{F,\epsilon} N_E^{1+\epsilon}.
$$
\end{theorem}
Note that here we do not need to assume modularity. This is thanks to results in \cite{FLS}, cf. Theorem \ref{ThmFinalApp} below. Also, the parity assumption is not essential (it can be relaxed by an argument along the lines of Section \ref{SecFinitePart}, as mentioned above) and we include it only to simplify the proofs.

The dependence of number $K_{F,\epsilon}$ on $F$ comes from two sources: An application of Faltings theorem for rational points on curves used in the modularity theorems of \cite{FLS}, and the existing literature on effective multiplicity one for $GL_2$ over number fields. As proved in \cite{FLS}, for real quadratic $F$ one can avoid the use of Faltings theorem and obtain a completely explicit modularity theorem. Thus, in order to get an explicit dependence on $F$ (at least for real quadratic fields) in Theorem \ref{ThmFinalIntro} one would need effective multiplicity one for $GL_2$ over  \emph{varying} number fields, which seems to be an interesting open problem in analytic number theory. See \cite{Brumley, LiWa, YWang} for the case of a fixed number field, which is what we use for Theorem \ref{ThmFinalIntro}.

Let us briefly outline the ideas in the proof of Theorem \ref{ThmSAtrIntro}. Unlike Theorem \ref{ThmSAI}, when $F\ne \Q$ we are left with no choice of Shimura curve  with cuspidal points, which forces us to intrinsically work on the compact quotient case. We show that for a suitably metrized line sheaf $\hat{\omega}$ (which is \emph{not} the Arakelov canonical sheaf) on an integral model of the elliptic curve $E$ and for well-chosen algebraic points $P\in E(\bar{F})$, one has $h(E)=h_{\hat{\omega}}(P)$ where $h(E)$ is the Faltings height of $E$ and $h_{\hat{\omega}}(P)$ is the Arakelov height of $P$. Suitable choice of quadratic CM extensions $K/F$ and their associated Heegner points $P_K\in X(\bar{F})$ provide an acceptable choice of $P=\phi(P_K)$ and, up to carefully comparing the metrics, the problem now reduces to showing that such a Heegner point in $X$ of small height (with respect to a metrized arithmetic Hodge bundle) does indeed exist. 

The comparison of metrics yields a correction factor which accounts for the term $\log \deg \phi$ in the upper bound of Theorem \ref{ThmSAtrIntro}, and it uses our $L^2-L^\infty$ comparison of norms (cf. Section \ref{SecNorms}) already used in the proof of Theorem \ref{ThmIntegralityIntro}. 

On the other hand, the construction of the Heegner points of small height is done in a way similar to that in the proof of our Theorem \ref{ThmIntegralityIntro}. Here, the technical difficulties are again coming from integral models and analytic number theory. The theory of Yuan and Zhang \cite{YuZh} provides the necessary material on integral models, the arithmetic Hodge bundle, and height formulas for Heegner points in Shimura curves in the totally real case. The necessary facts from analytic number theory are obtained from two sources: The construction of auxiliary CM extensions of $F$ is reduced to a sieve theory problem over $\Q$ which is more approachable, while the zero-density estimates that we use (which are analogous to those of Heath-Brown that we use over $\Q$) were kindly provided to us by  Lemke Oliver and Thorner in the Appendix ---see also \cite{OliverThorner} by the same authors of the Appendix, which gives a zero-density estimate strong enough to prove, for instance, Theorem \ref{ThmSAtrIntro} in the case when $F$ is obtained as a tower of cyclic extensions, such as needed for Theorem \ref{ThmQuadraticCondIntro}.


\subsection{Other references} We conclude this introduction by briefly recalling some literature around the $abc$ conjecture, although only tenuously related to our work. We are not attempting to make a survey on the $abc$ conjecture and the list is by no means complete, but we feel that it is close in taste to the topics in this article.

First, we mention that in \cite{FoNaTe92} the set of elliptic curves failing Szpiro's conjecture (over $\Q$, for a given exponent) is shown to be small in a precise asymptotic sense.

In \cite{Ullmo}, Ullmo gave estimates for the Faltings height of $J_0(N)$ when $N$ is squarefree coprime to $6$. He conjectured that the Faltings height of $J_0(N)$ is somewhat evenly distributed among the simple factors of $J_0(N)$ according to their dimension, which would imply a form of Szpiro's conjecture thanks to his height estimates.  

In \cite{Prasanna}, Prasanna established integrality results in the context of the Jacquet-Langlands correspondence for Shimura curves, which allowed him to compute local contributions of the Faltings height of jacobians of Shimura curves over $\Q$ away from an explicit set of primes. 

In \cite{FiM}, Freixas i Montplet used the Jacquet-Langlands correspondence to compute Arakelov-theoretic invariants of certain Shimura curves over $\Q$, but only after discarding the contribution of finitely many primes. 

Despite all this progress, the computation (or asymptotic evaluation) in terms of the level of the Faltings height of the jacobian of quaternionic Shimura curves remains open. This is not used in our work.

From a completely different angle, S.-W. Zhang \cite{ZhangGS} showed that appropriate bounds on the height of the Gross-Schoen cycle in triple products of curves, would imply the $abc$ conjecture.

Ullmo \cite{UllmoHida} obtained higher dimensional analogues of part of the results of Ribet and Takahashi \cite{RiTa}, and used them to prove the existence of level-lowering congruences of modular forms in some cases by showing \emph{lower} bounds for the Hida constant ---a higher dimensional analogue of the modular degree. However, no connections with the $abc$ conjecture are established (the latter seems to be closer to the problem of giving upper bounds for the Hida constant).
 
Regarding the $abc$ conjecture for number fields other than $\Q$, the method of linear forms in logarithms is available and it has been worked out by Surroca \cite{Surroca} obtaining effective unconditional exponential bounds. Sharper exponential bounds have been proved by Gy\"ory \cite{Gyory}. However, as in the case of $\Q$, this approach has well understood technical limitations; see Baker's article \cite{Baker} for a discussion on conjectural strengthenings to the theory of linear forms in logarithms that would be needed in order to approach the $abc$ conjecture.




\section{General notation}\label{SecNotation}

For later reference, let us record here  the basic notation used throughout the paper. Further notation and preliminary material will be introduced as needed.

If $X$ and $Y$ are function on some set with $X$ complex valued and $Y$ positive real valued,  Landau's notation $X=O(Y)$ means that there is a constant $c>0$ such that $|X|\le c\cdot Y$ pointwise. This means exactly the same as Vinogradov's notation $X\ll Y$. For instance, $X=O(1)$ (or equivalently, $X\ll 1$) means that $X$ is a bounded function. The number $c$ in the previous definition is referred to as ``implicit constant''. If the implicit constant depends on any choice of other parameters, this will always be indicated as a subscript in the notation $O$ or $\ll$. It is also customary to write $X\asymp Y$  when both $X\ll Y$ and $Y\ll X$ hold.

We will be using the following arithmetic functions: $\varphi(n)$ is the Euler totient function, $d(n)$ is the number of divisors of $n$, $\sigma_1(n)$ is the sum of the divisors of $n$, $\omega(n)$ is the number of distinct prime divisors of $n$, and $\mu(n)$ is the M\"obius function.

Recall from the introduction that for a number field $F$ the absolute value of its discriminant is $d_F$. Also, if $E$ is an elliptic curve over $F$, we write  $\Nfrak_E$ for its conductor ideal, whose norm is denoted by $N_E$. We write $\Delta_E$ for the norm of the minimal discriminant ideal of $E$.

Note that when $F=\Q$,  the absolute value of  the minimal discriminant of $E$ is precisely $\Delta_E$, and the  conductor of $E$ equals $N_E$. We will often write simply $N$ instead of $N_E$ if there is no risk of confusion. We will be willing to discard finitely many (isomorphism classes of) elliptic curves in several of our arguments, which is the same as letting $N$ be large enough, by Shafarevich's theorem. Such a reduction is of course effective.

Let $N$ be a positive integer. We will say that a factorization $N=DM$ is \emph{admissible} if $D,M$ are coprime positive integers with $D$ being the product of an even number of distinct prime factors, possibly $D=1$.

Given a positive integer $N$ with an admissible factorization $N=DM$, we denote by $X_0^D(M)$ the canonical model over $\Q$ (determined by special points) of the compact Shimura curve associated to an Eichler order of index $M$ for the quaternion $\Q$-algebra of discriminant $D$. The particular case $D=1,M=N$ is simply denoted by $X_0(N)$. The Jacobian of $X_0^D(M)$ (resp. $X_0(N)$) is denoted by $J_0^D(M)$ (resp. $J_0(N)$). Later, in Section \ref{SecPrelim}, we will review in closer detail some relevant facts about Shimura curves that will be needed in the arguments.

Write $S_2(N)$ for the complex vector space of weight $2$ holomorphic cuspidal modular forms for the congruence subgroup $\Gamma_0(N)$. Given an elliptic curve $E$ over $\Q$, the modularity theorem  \cite{Wiles, TaylorWiles, BCDT} associates to $E$ a unique Fourier-normalized Hecke eigenform $f\in S_2(N)$ with rational Hecke eigenvalues. Here, $N=N_E$ and $f$ is a newform of level $N$. Associated to $f$, the Eichler-Shimura construction gives an optimal  quotient $q_{1,N}: J_0(N)\to A_{1,N}$ defined over $\Q$, with $A_{1,N}$ isogenous to $E$.  More generally, for each admissible factorization $N=DM$, the Jacquet-Langlands correspondence gives an optimal quotient  $q_{D,M}: J_0^D(M)\to A_{D,M}$ defined over $\Q$, with $A_{D,M}$ isogenous to $E$. (As usual, the word ``optimal'' means ``with connected kernel''.)

In this setting, the \emph{$(D,M)$-modular degree} of $E$ (or simply \emph{modular degree} if the pair $(D,M)$ is understood), denoted by $\delta_{D,M}$ is defined as follows: The dual map $q_{D,M}^\vee:A_{D,M}\to J_0^D(M)$ satisfies that $q_{D,M}q_{D,M}^\vee\in \End(E)$ is multiplication by a positive integer, and $\delta_{D,M}$ is defined to be this integer. For $D=1$, we simply call $\delta_{1,N}$ the modular degree of $E$. Although the notation $\delta_{D,M}$ does not indicate the elliptic curve $E$, these numbers  do depend on $E$, and sometimes it will be convenient to make this dependence explicit by writing $\delta_{D,M}(E)$ instead.

The integers $\delta_{D,M}$ are a central object in this paper. We will  see in Corollary \ref{CoroDegApproxQ} that they are ``almost'' equal to the degree of suitably chosen maps $\phi_{D,M} : X_0^D(M)\to E_{D,M}$, but a priori they cannot be interpreted in this way. In fact, for $D\ne 1$ the curve $X_0^D(M)$ does not have $\Q$-rational points (not even cuspidal points) so it is not a priori clear how to map $X_0^D(M)$ to $J_0^D(M)$. 

In the particular case $D=1$, however, the cusp $i\infty$ is $\Q$-rational and we can use it to define an embedding $j_N:X_0(N)\to J_0(N)$. The composition $\phi=q_{1,N}j_N:X_0(N)\to A_{1,N}$ is well-known to have degree $\delta_{1,N}$.

We will consider the pull-back of relative differentials in several contexts, so let us introduce the notation once and for all. Given morphisms of $B$-schemes $f:X\to Y$ for a base scheme $B$, the sheaf $f^*\Omega^1_{Y/B}$ is not in general a sub-sheaf of $\Omega^1_{X/B}$, and for sections $s\in H^0(V,\Omega^1_{Y/B})$ ($V$ open in $Y$) the section $f^*s$ belongs to $H^0(f^{-1}V,f^*\Omega^1_{X/B})$. Nevertheless, there is a canonical map $f_{X/Y/B}: f^*\Omega^1_{Y/B}\to \Omega^1_{X/B}$ sitting in  the fundamental exact sequence $f^*\Omega^1_{Y/B}\to \Omega^1_{X/B}\to \Omega^1_{X/Y}\to 0$, and we define $f^\bullet s\in H^0(f^{-1}V,\Omega^1_{X/B})$ to be the image of $f^*s$ under $f_{X/Y/B}$. A similar construction applies in the analytic setting.

\section{Review of the classical modular approach} \label{SecClassical}

In this section we recall the classical modular approach to the $abc$ conjecture and Szpiro's conjecture. Except for Remark \ref{RmkHtDeg} below (which fills a gap in the literature), everything in this section is well-known and we include it for later reference, and to clarify the similarities and differences with  our approach in other parts of the paper. In fact, none of our results stated in the introduction is obtained directly from the classical approach reviewed in this section, although  some of the ideas will be useful.

Given a triple $a,b,c$ of coprime positive integers with $a+b=c$,  the Frey-Hellegouarch elliptic curve $E_{a,b,c}$ is defined by the affine equation
$$
y^2=x(x-a)(x+b).
$$
One directly checks that $E_{a,b,c}$ is semi-stable away from $2$. Furthermore (cf. p.256-257 in \cite{Silverman2Ed}),  $\Delta_{E_{a,b,c}}=2^{s} (abc)^2$ and $N_{E_{a,b,c}}=2^t \rad(abc)$ for integers $s,t$ with $-8\le s\le 4$ and $-1\le t\le  7$.  See   \cite{DiamondKramer} for a detailed analysis of the local invariants at $p=2$ (possibly after twisting $E_{a,b,c}$ by $-1$). From here, it is clear that Conjecture \ref{ConjSzpiroIntro} implies Conjecture \ref{ConjABCIntro} and that any partial result for Conjecture \ref{ConjSzpiroIntro} which applies to Frey-Hellegouarch elliptic curves yields a partial result for the $abc$ conjecture.

For an elliptic curve $E$ over $\Q$, let $h(E)$ be the Faltings height of $E$ over $\Q$ (this is not the semi-stable Faltings height). For $\omega_E$ a global N\'eron differential of $E$ we have
$$
h(E)=-\frac{1}{2}\log\left(\frac{i}{2}\int_{E(\C)}\omega_E\wedge \overline{\omega_E}\right).
$$
Furthermore, by a formula of Silverman \cite{Silverman} we get
\begin{equation}\label{EqDiscH}
\log \Delta_E\le 12 h(E) + 16.
\end{equation}
At this point it is appropriate to recall:
\begin{conjecture}[Height conjecture, cf. \cite{FreyLinks}] There is a constant $\kappa$ such that for all elliptic curves $E$ over $\Q$ we have $h(E) < \kappa \cdot \log N_E$.
\end{conjecture}
By the previous discussion, the height conjecture implies Szpiro's conjecture.

We have the standard modular parameterization $\phi= q_{1,N}j_N: X_0(N)\to A_{1,N}$ whose degree is $\delta_{1,N}=\delta_{1,N}(E)$. The degree of a minimal isogeny between $A_{1,N}$ and $E$ is uniformly bounded by $163$ thanks to Mazur \cite{MazurRatIsog} and Kenku \cite{Kenku}, so Lemma 5 in \cite{Faltings} gives
$$
\left| h(A_{1,N})- h(E)\right| \le\frac{1}{2}\log 163< 3.
$$
Let $f\in S_2(N)$ be the associated normalized Hecke eigenform and let $c_f$ denote the (positive) Manin constant of the optimal quotient $A_{1,N}$, which is defined as the absolute value of the scalar $c$ satisfying that the pull-back of the N\'eron differential $\omega_{A_{1,N}}$ to $\hfrak$ via the composition $\hfrak\to X_0(N)\to A_{1,N}$ is $2\pi i c f(z)dz$. The Manin constant is a positive integer (cf. \cite{Edixhoven}); further details on the Manin constant will be discussed in Section \ref{SecManinCt}.

Let us fix the notation
$$
\mu_\hfrak(z)= \frac{dx\wedge dy}{y^2} \qquad (z=x+yi\in \hfrak)
$$
for the usual hyperbolic measure on $\hfrak$.

Frey \cite{FreyLinks} observed that by pulling-back the $(1,1)$-form $\omega_{A_{1,N}}\wedge \overline{\omega}_{A_{1,N}}$ from $A_{1,N}$ to $X_0(N)$, one obtains  
\begin{equation}\label{EqFrey}
\log \delta_{1,N}(E) = 2\log(2\pi c_f) +2\log( \|f\|_{2,\Gamma_0(N)}) +2h(A_{1,N})
\end{equation}
where   the Petersson norm of $f$ is given by
$$
\|f\|^2_{2,\Gamma_0(N)} := \int_{\Gamma_0(N)\backslash \hfrak} |f(z)|^2y^2 d\mu_\hfrak(z). 
$$
Since the Manin constant $c_f$ is a positive  integer and since $f$ has Fourier coefficients in $\Z$, one finds by integrating $f\cdot \bar{f}$ on $\{z\in\hfrak : |x|<1/2\mbox{ and } y>1\}$ that
$$
\log(2\pi c_f) +\log( \|f\|_{2,\Gamma_0(N)}) > -6.
$$ 
(Actually, as pointed out in \cite{MurtyBounds}, one has $2\log \|f\|_{2,\Gamma_0(N)}\sim \log N$ with an effective error term by \cite{HofLoc}, but the previous lower bound is trivial and sufficient for our current purposes.)
Therefore by \eqref{EqFrey} we obtain
\begin{equation}\label{EqHDeg}
h(E)\le h(A_{1,N}) +3  \le  \frac{1}{2}\log \delta_{1,N} + 9.
\end{equation}

Thus, the height conjecture (and hence, Szpiro's conjecture)  would follow from the following:
\begin{conjecture}[Modular degree conjecture; cf. \cite{FreyLinks}] \label{ConjModDegQ} As $E$ varies over elliptic curves over $\Q$, we have
$$
\log \delta_{1,N}(E)\ll \log N_E.
$$
\end{conjecture}

\begin{remark}\label{RmkHtDeg} Actually, Frey formulated in \cite{FreyLinks} the modular degree conjecture as the (seemingly weaker) bound $\log (\delta_{1,N}/c_f^2) \ll \log N_E$, which has exactly the same consequences for Szpiro's conjecture. It is this version involving the Manin constant which is known to be equivalent to the height  conjecture, see \cite{MaiMurty}. Other expositions (such as \cite{SzpiroDisc} or \cite{MurtyBounds}) formulate the degree conjecture as we did above in Conjecture \ref{ConjModDegQ}, but at present it is not clear that this version is equivalent to the height conjecture. Our bounds for the Manin constant (cf. Corollary \ref{CoroManinCt}) show that this is indeed the case if additive reduction is restricted to a fixed finite set of primes, such as for Frey-Hellegouarch elliptic curves. The latter fills a gap in Theorem 1 of \cite{MurtyBounds} (repeated elsewhere in the literature) concerning the equivalence of the $abc$ conjecture and the modular degree conjecture for Frey-Hellegouarch elliptic curves.
\end{remark}


\section{Preliminaries on Shimura curves}\label{SecPrelim}

In this section we recall some facts about the theory of Shimura curves. We work over $\Q$ for simplicity, as this is the case that will be used most of the time in this work. Just in Section \ref{SecModApproachF} we will need algebraic and arithmetic results about Shimura curves over totally real number fields, and the necessary facts will be recalled there.

The results in this section are by now standard. They can be found in \cite{Gross}, \cite{Shimura}, and \cite{ZhangAnn}, among other references. We include them to fix the notation and to simplify the exposition.

\subsection{Shimura curves}\label{SecNotationShiCur} Let $D$ be a squarefree positive integer with an even number of prime divisors, possibly $D=1$.  Let $B$ the unique (up to isomorphism) quaternion $\Q$-algebra ramified exactly at the primes dividing $D$. For each prime $p\nmid D$ fix an isomorphism $B_p=M_2(\Q_p)$, and under this identification we take the maximal order $O_p=M_2(\Z_p)$. For $v|D$ we also let $O_p$ denote a maximal order in the completion $B_p$. These choices determine a maximal order $O_B\subseteq B$ with $O_B\otimes \Z_p = O_p$ in $B_p$, for each prime $p$. We also fix an identification at the infinite place $B_{\infty} = M_2(\R)$, which determines an action of $B^\times $ on $\hfrak^{\pm}=\C\smallsetminus \R$ by fractional linear transformations.

Let $\A^\infty=\Q\otimes \hat{\Z}$ be the ring of finite adeles of $\Q$. Let $\B:=B\otimes \A^\infty$ and write $O_\B=O_B\otimes_\Z\hat{\Z}$, the maximal order in $\B$ induced by $O_B$. For each compact open subgroup $U\subseteq \B^\times$ let $X_U$ be the compactified Shimura curve over $\Q$ associated to $U$. More precisely, the set of complex points of $X_U$ is the complex curve
$$
X_U^{an}=B^\times\backslash \hfrak^{\pm}\times \B^\times/U \cup\{cusps\}
$$
and the model over $\Q$ is the canonical one defined in terms of special points. The cuspidal points are necessary only when $D=1$. The curve $X_U$ is irreducible over $\Q$.

\subsection{Projective system}\label{SecNotationPro} For $U\subseteq V$ compact open subgroups of $\B^\times$, the natural map $X_U\to X_V$ is defined over $\Q$, and they define a projective system $\{X_U\}_U$. The limit is a $\Q$-scheme $X$ whose complex points are given by
$$
X^{an}= B^\times\backslash \hfrak^{\pm} \times \B^\times \cup\{cusps\}.
$$
 The scheme $X$ comes with a right $\B^\times$ action by $\Q$-automorphisms, under which $X/U\simeq X_U$. This right $\B^\times$ action is compatible with the right action on the projective system $\{X_U\}_U$ given by the $F$-maps $\cdot g : X_U\to X_{g^{-1}Ug}$ for $g\in\B^\times$.

\subsection{Components}\label{SecNotationComp} Let $\Q^\times_+$ be the set of strictly positive rational numbers. Let $B^\times_+$ be the subgroup of $B^\times$ consisting of elements with  reduced norm in $\Q^\times_+$. Then $B^\times_+$ acts on the upper half plane $\hfrak$,  and the natural map
$$
B^\times_+\backslash \hfrak\times \B^\times/U\to B^\times\backslash \hfrak^{\pm}\times \B^\times/U
$$
is an isomorphism. So, we can identify $X_U^{an}$ with the compactification of the former. Let $\rn:\B\to \A^\infty$ be the reduced norm, then the connected components of $X_U^{an}$ are indexed by the class group $C(U):=\Q^\times_+\backslash \A^{\infty,\times}/ \rn(U)$ via the natural map induced by $\rn$. The number of connected components of $X_U^{an}$ is $h_U:=\#C(U)$. The component associated to $a\in C(U)$ is denoted by $X_{U,a}^{an}$.

Given $g\in \B^\times$ define
$$
\Gamma_{U,g}:=g U g^{-1}\cap B^\times_+.
$$
We see that $\Gamma_{U,g}\subseteq GL_2(\R)^+$, and its image $\tilde{\Gamma}_{U,g}$ in $PSL_2(\R)$ is a discrete subgroup acting on $\hfrak$ on the left. This gives rise to the (compactified) complex curve
$$
X_{U,g}^{an}:=  \tilde{\Gamma}_{U,g}\backslash \hfrak \cup \{cusps\}
$$
where the cusps are only needed if $D=1$. It comes with the obvious complex uniformization
$$
\xi_{U,g}: \hfrak\to X_{U,g}^{an}.
$$

For each $a\in C(U)$ choose $g_a\in\B^\times$ with $[\rn(g_a)]=a$ in $C(U)$.  One has the bi-holomorphic bijection
\begin{equation}\label{EqDecomposition}
\coprod_{a\in C(U)}  \tilde{\Gamma}_{U,g_a}\backslash \hfrak \cup \{cusps\}\to X_U^{an},\quad  \tilde{\Gamma}_{U,g_a}\cdot z\mapsto [z,g_a]
\end{equation}
respecting the projections onto $C(U)$. We can identify the component $X_{U,a}^{an}$ with $X_{U,g_a}^{an}$. Thus, after the choice of $g_a$ for each $a\in C(U)$ is made, there is no harm in simplifying the notation as follows: $\Gamma_{U,a}=\Gamma_{U,g_a}$,   $\tilde{\Gamma}_{U,a}=\tilde{\Gamma}_{U,g_a}$, and $X_{U,a}^{an}=X_{U,g_a}^{an}$.

Let $H_U$ be the field extension of $\Q$ associated to $C(U)$ by class field theory. Each component $X_{U,a}^{an}$ has a model $X_{U,a}$ over $H_U$, so that $X_U\otimes H_U= \coprod_{a\in C(U)} X_{U,a}$.

\subsection{Heegner points} \label{SubSecHeegner} We say that an imaginary quadratic field $K$  satisfies the \emph{Heegner hypothesis for $D$} if every prime  $p| D$ is inert in $K$. In particular, $K/\Q$ is unramified at primes dividing $D$. 

Let $K$ satisfy the Heegner hypothesis for $D$. Then there is a $\Q$-algebra embedding $\psi:K\to B$ which is \emph{optimal} in the sense that $\psi^{-1} O_B=O_K$; we fix such an optimal embedding. We have that  $\psi(K^\times)\subseteq B^\times_+$ and there is a unique $\tau_K\in\hfrak$ which is fixed by all elements of $\psi(K^\times)$ (a more accurate notation would be $\tau_{K,\psi}$, but any other choice of optimal embedding leads to an equivalent theory). These choices determine the point
$$
P_K = [\tau_K, 1]\in B^\times_+\backslash \hfrak\times \B^\times \subseteq X^{an}
$$
which maps to a point $P_{K,U}\in X_U^{an}$ for each compact open subgroup $U$. The point $P_K$ (hence, all the $P_{K,U}$) is algebraic and defined over $K^{ab}$, the maximal abelian extension of $K$. Since $\psi$ was chosen as an optimal embedding, it follows  from Shimura reciprocity that the point $P_{K, O_{B}^\times}$  of $X_{O_\B^\times}$ has residue field $H_K$, the Hilbert class field of $K$.

We will refer to these points as \emph{$D$-Heegner points}.


\subsection{Jacobians} We adopt the convention that the Jacobian of an irreducible curve \emph{is} the Albanese variety, so that its dual is the identity component of the Picard variety ---this clarification is relevant in terms of functoriality. The Jacobian of $X_U$ is denoted by $J_U$, and the torus of its complex points is $J_U^{an}$. Similarly, for $X_{U,a}$ its Jacobian is $J_{U,a}$, having $J_{U,a}^{an}$ as set of complex points. Note that $J_U$ is defined over $\Q$ and $J_{U,a}$ is defined over $H_U$. Furthermore $J_{U,a}^{an}$ is the Jacobian of the irreducible complex curve $X_{U,a}^{an}$, while the points of $J_{U}^{an}$ correspond to divisor classes on $X_U^{an}$ having degree $0$ on each component $X_{U,a}^{an}$, hence $J_{U}^{an}=\prod_{a\in C(U)}J_{U,a}^{an}$. This decomposition is defined over $H_U$, namely $J_U\otimes H_U=\prod_{a\in C(U)}J_{U,a}$.


\subsection{Modular forms and differentials} \label{SecNotationForms} 
The space of cuspidal holomorphic weight $2$ modular forms for the discrete subgroup $\tilde{\Gamma}_{U,a}$ acting on $\hfrak$ is denoted by $S_{U,a}$ (the cuspidal condition is only needed when $D=1$). All the curves $X_{U,a}^{an}$ have the same genus $g_U$, and the uniformization $\xi_{U,a}:\hfrak\to \tilde{\Gamma}_{U,a}\backslash \hfrak$ induces by pull-back an isomorphism  $\Psi_{U,g_a}=\Psi_{U,a}:  H^0(X^{an}_{U,a},\Omega^1)\to S_{U,a}$ given by the condition that the image of a differential $\alpha$ is  the modular form $\Psi_{U,a}(\alpha)\in S_{U,a}$ satisfying that 
$$
\xi_{U,a}^\bullet\alpha=\Psi_{U,a}(\alpha)dz
$$ 
with $z$ the complex variable on $\hfrak$. In particular $\dim S_{U,a}=g_U$. Thus we get isomorphisms
$$
H^0(X_U,\Omega^1)\otimes\C\simeq H^0(X_U^{an},\Omega^1)=\bigoplus_{a\in C(U)} H^0(X_{U,a}^{an},\Omega^1)\simeq \bigoplus_{a\in C(U)} S_{U,a}.
$$
The inner product $(-,-)_U$ on $H^0(X_U^{an},\Omega^1)$ is defined so that the direct summands $H^0(X_{U,a}^{an}, \Omega^1)$ are orthogonal, and for $\omega_1,\omega_2\in H^0(X_{U,a}^{an},\Omega^1)$ 
$$
(\omega_1,\omega_2)_U := \frac{i}{2}\int_{X_{U,a}^{an}}\omega_1\wedge\overline{\omega_2}.
$$
The (non-normalized) Petersson inner product $\langle-,-\rangle_{U,a}$ on $S_{U,a}$ is defined by 
$$
\langle h_1,h_2\rangle_{U,a} := \int_{\tilde{\Gamma}_{U,a}\backslash \hfrak} h_1(z)\overline{h_2(z)}\Im(z)^2d\mu_\hfrak(z)
$$
for $h_1,h_2\in S_{U,a}$. We note that for $\omega_1,\omega_2\in H^0(X_{U,a}^{an},\Omega^1)$ we get 
$$
(\omega_1,\omega_2)_U=\langle \Psi_{U,a}(\omega_1),\Psi_{U,a}(\omega_2)\rangle_{U,a}.
$$ 
Thus, $(-,-)_U$ will also be called the Petersson inner product.

\subsection{Classical subgroups} 
Using the maximal order $O_\B$ of $\B$ we recall some standard choices for $U$. For a positive integer $n$, define $U^D(n)=(1+n O_\B)^\times$. These open compact subgroups $U^D(n)$ determine a cofinal system in the projective system $\{X_U\}_U$. We say that a compact open $U$ has level $n$ if $U^D(n)\subseteq U$ and for every other positive integer $n'$ with $U^D(n')\subseteq U$ we have $n|n'$.

For $M$ coprime to $D$ we define $U_0^D(M)$ prime by prime as follows. For $p\nmid M$ finite, we let $U_0^D(M)_p= O_p$, i.e. maximal at $p$. For $p|M$ the identification $B_p=M_2(\Q_p)$ allows us to define 
$$
U_0^D(M)_p = \left\{ g\in GL_2(\Z_p) : g\equiv \left[\begin{array}{cc}
* & *\\
0 & *
\end{array}\right]\mod M \right\}.
$$
The subgroup $U_1^D(M)$ is defined similarly, except that the last condition is replaced by
$$
U_1^D(M)_p = \left\{ g\in GL_2(\Z_p) : g\equiv \left[\begin{array}{cc}
* & *\\
0 & 1
\end{array}\right]\mod M \right\}.
$$
Both $U_0^D(M)$ and $U_1^D(M)$ have level $M$.

For the groups $U^D(n)$, $U_0^D(M)$, and $U_1^D(M)$, the notation $X_U$ is replaced by $X^D(n)$, $X_0^D(M)$, and $X_1^D(M)$ respectively (the upper script $D$ keeps track of the quaternion algebra, up to isomorphism ---all other relevant choices are implicit). Similarly for their Jacobians, sets of complex points, components, spaces of modular forms (always cuspidal of weight $2$), etc. To simplify notation, we may omit the upper script $D$ in the case $D=1$.

The reduced norm maps $U=U_0^D(M_1)\cap U_1^D(M_2)$ surjectively onto the maximal open compact subgroup $\hat\Z^\times\subseteq \A^{\infty,\times}$, so $C(U)$ is trivial in this case, as the narrow class number of $\Q$ is $1$. Thus, the Shimura curve associated to a compact open subgroup of the form $U_0^D(M_1)\cap U_1^D(M_2)$ is geometrically connected.

\subsection{Hecke action} (cf. \cite{ZhangAnn} Sec. 1.4, 3.2) 
Let $U$ be a compact open subgroup of $\B^\times$ of level $m$. Recall that for each positive integer $n$ coprime to $Dm$ one has a Hecke correspondence $T_{U,n}^c$ on $X_U$. It is defined over $\Q$ and has degree $\sigma_1(n)$, the sum of divisors of $n$, in the sense that the induced push-forward map $Div(X_U)\to Div(X_U)$ multiplies degrees by $\sigma_1(n)$. 

Hecke correspondences generate a commutative ring $\T_U^c$. This ring acts (contravariantly) by endomorphisms on $H^0(X_U,\Omega^1)$ via pull-back and trace. Let $\T_U$ be the image of $\T_U^c$ in $\End_\Q H^0(X_U,\Omega^1)$, and denote by $T_{U,n}$ the image of $T_{U,n}^c$. The ring $\T_U^c$ also acts (covariantly) by endomorphisms on $J_U$, and the image of $\T_U^c$ in $\End_\Q J_U$ is isomorphic to $\T_U$. In fact, the action of $\T_U$ on $J_U$ is compatible with that on $H^0(X_U,\Omega^1)$ in the following sense: The action of $\T_U$ on $J_U$ induces by pull-back an action on $H^0(J_U,\Omega^1)$, and the canonical isomorphism $ H^0(J_U,\Omega^1)\simeq H^0(X_U,\Omega^1)$ is $\T_U$-equivariant for this action.

\subsection{Systems of Hecke eigenvalues} 
For $M$ coprime to $D$ and for $U=U^D_0(M)$ we write $\T_{D,M}$ instead of $\T_U$, similarly for the Hecke operators $\T_{D,M,n}$. The action of $\T_{D,M}$ on $V_{D,M}:=H^0(X_0^D(M),\Omega^1)\otimes \C$ is simultaneously diagonalizable, which gives rise to systems of Hecke eigenvalues $\chi:\T_{D,M}\to \C$. Such a $\chi$ takes values in the ring of integers of a totally real number field. The associated isotypical subspaces $V_{D,M}^\chi$ are orthogonal to each other for the Petersson product.

If $m|M$ the map $X_0^D(M)\to X_0^D(m)$ induces by pull-back an inclusion $V_{D,m}\subseteq V_{D,M}$, which satisfies that for $n$ coprime to $DM$ one has $T_{D,M,n}|_{V_{D,m}} = T_{D,m,n}$. Hence we can lift  systems of Hecke eigenvalues from $\T_{D,m}$ to $\T_{D,M}$. Those $\chi:\T_{D,M}\to \C$ arising in this way for $m|M$ with $m\ne M$ are called old, and the remaining $\chi$ are called new. The minimal $m|M$ from which $\chi$ arises is called the level of $\chi$. Multiplicity one holds for the new eigenspaces, that is, if $\chi$ is new then $\dim V_{D,M}^\chi=1$ (cf. \cite{ZhangAnn}).

In the previous discussion, one can of course replace $V_{D,M}$ by $S_2^D(M):=S_{U_0^D(M)}$.

\subsection{Jacquet-Langlands} Write $N=DM$ with $M$ coprime to $D$ and observe that our notation gives $S_2^1(N)=S_2(N)$ when $D=1$, where as usual $S_2(N)$ denotes the space of weight $2$ holomorphic cuspidal modular forms for the congruence subgroup $\Gamma_0(N)$.

For $d$ a divisor of $N$, write $S_2(N)^d=\oplus_{d|\chi}S_2(N)^\chi$ where the notation `` $d|\chi$ " means that $\chi$ varies over the systems of Hecke eigenvalues of $\T_{1,N}$ whose level is divisible by $d$. The Jacquet-Langlands correspondence gives an isomorphism
$$
JL: S_2^D(M)\to S_2(N)^D
$$
such that for every $n$ coprime to $N$ we have 
$$
JL\circ T_{D,M,n}= T_{1,N,n}|_{S_2(N)^D} \circ JL.
$$
Hence,  $JL$ induces a quotient map $\T_{1,N}\to \T_{D,M}$ which realizes the systems of Hecke eigenvalues of $\T_{D,M}$ precisely as the systems of Hecke eigenvalues of $\T_{1,N}$ with level divisible by $D$.


\subsection{The Shimura construction} We say that two systems of Hecke eigenvalues $\chi_1,\chi_2: \T_{D,M}\to \bar{\Q}$ are equivalent if $\chi_1=\sigma\chi_2$ for some automorphism $\sigma\in G_\Q$. The equivalence class of $\chi$ is denoted by $[\chi]$. Note that the degree of the field generated by the values of $\chi$ is $\#[\chi]$, and the property that $[\chi]$ be new is well-defined. All the elements of $[\chi]$ have the same kernel, which we denote by $\I_{[\chi]}$.

Given a class $[\chi]$ of systems of Hecke eigenvalues, define the (connected) abelian subvariety 
$$
K_{[\chi]}:=\I_{[\chi]}\cdot J_0^D(M)\le J_0^D(M). 
$$
Then $K_{[\chi]}$ is defined over $\Q$ and one obtains the quotient $q_{[\chi]}:J_0^D(M)\to A_{[\chi]}$ with kernel $K_{[\chi]}$, also defined over $\Q$. If $[\chi]$ is new, then the abelian variety $A_{[\chi]}$ is simple over $\Q$ and has dimension $\#[\chi]$. Since $K_{[\chi]}$ is $\T_{D,M}$-stable, $\T_{D,M}$ also acts on $A_{[\chi]}$ making $q_{[\chi]}$ Hecke equivariant.

Furthermore, $\End(A_{[\chi]})$ contains a ring $O_{[\chi]}$ isomorphic to $\chi(\T_{D,M})$ (for any $\chi$ in $[\chi]$) which is an order in a totally real field of degree $\#[\chi]$. Fixing any choice of $\chi$ and isomorphism $O_{[\chi]}\simeq \chi(\T_{D,M})$, we have that $\T_{D,M}$ acts on $A_{[\chi]}$ via $\chi$. Thus, the kernel of $\T_{D,M} \to \End(A_{[\chi]})$ is precisely $\I_{[\chi]}$.

The maps $q_{[\chi]}$ induce an isogeny $J_0^D(M)\to \prod_{[\chi]} A_{[\chi]}$. Given any $[\chi]$ define $B_{[\chi]}$ as the identity component of the kernel of $J_0^D(M)\to \prod_{[\chi']\ne[\chi]} A_{[\chi']}$. That is, $B_{[\chi]}$ is the identity component of $\bigcap_{[\chi']\ne[\chi]} K_{[\chi']}$. The abelian variety $B_{[\chi]}$ is defined over $\Q$, it is stable under the action of $\T_{D,M}$, and it is the largest abelian subvariety of $J_0^D(M)$ on which $\I_{[\chi]}$ acts as $0$. The composition $B_{[\chi]}\to J_0^D(M)\to A_{[\chi]}$ is an isogeny defined over $\Q$, thus, $J_0^D(M)=\sum_{[\chi]}B_{[\chi]}$.



\section{The degree of Shimura curve parameterizations}\label{SecModDeg}

\subsection{The modular degree} In this section we fix an elliptic curve $A$ over $\Q$ of conductor $N$ which is an optimal quotient  $q:J_0^D(M)\to A$ for an admissible factorization $N=DM$. The associated modular degree $\delta_{D,M}$ will be simply denoted by $\delta$. The kernel of $q$ is denoted by $K$, the image of $q^\vee : A\to J_0^D(M)$ is denoted by $B$, and the system of Hecke eigenvalues attached to $A$ is denoted by $\chi_0$. Note that $\chi_0$ is $\Z$-valued. Also, since $q^\vee$ is an injective, we get an isomorphism $A[\delta]\to B[\delta]=K\cap B$ which respects both the Galois action and the Hecke action.

The goal of this section is to develop some tools to handle the quantity $\delta$.

\subsection{Maps to the Jacobian} \label{SecMapsQ} Since $X_0^D(M)$ does not have a canonical embedding into $J_0^D(M)$ when $D>1$, we cannot a priori interpret  $\delta$ as the degree of a  map $X_0^D(M)\to A$, unlike the case of classical modular parameterizations. 

Let us momentarily address the general case. Let $U\subseteq \B^\times$ be a compact open subgroup of level $m$. Let $t:X_U\leadsto X_U$ be an algebraic correspondence defined over $\Q$ of degree $0$ (meaning that it multiplies degrees of divisors by $0$) with the following additional property (*): $t$ maps complex points $x\in X_U^{an}$ to degree zero divisors $t(x)$ supported on the same complex component containing $x$. Then for each $a\in C(U)$ one gets a morphism $j_{t,a}:X_{U,a}^{an}\to J_{U,a}^{an}$, and via the canonical isomorphism $J_U^{an}=\prod_a J_{U,a}^{an}$ they descend to a morphism $j_t : X_U\to J_U$ defined over $\Q$. By construction, for each $a\in C_U^{an}$ the following diagram commutes:
$$
\begin{CD}
X_{U,a}^{an} @>>> X_{U}^{an}\\
@VVV @VVV\\
Jac_\C(X_{U,a}^{an}) @>>> J_U^{an}.
\end{CD}
$$
We obtain a morphism $j_t^\bullet:H^0(J_U,\Omega^1)\to H^0(X_U,\Omega^1)$ which has no reason to be an isomorphism (e.g. take $t=0$).  Since the ring $\T_U^c$ of Hecke correspondences is commutative, we get that if in addition $t\in \T_U^c$, then for every $n$ coprime to $Dm$ we have $j_t^\bullet\circ T_{U,n}=T_{U,n}\circ j_t^\bullet$. In particular, the map $j_t^\bullet$ is $\T_U$-equivariant.

A particularly convenient choice of $t\in\T_U^c$ of degree $0$ is given by the \emph{Eisenstein correspondences}
$$
E^c_{U,n}:=T^c_{U,n} - \sigma_1(n)\cdot \Delta_U
$$ 
where $\Delta_U$ is the diagonal correspondence of $X_U$, and $n$ is coprime to $Dm$. 

In the particular case $U=U_0^D(M)$, taking $n$ coprime to $N=DM$ and recalling that $C(U_0^D(M))$ is trivial (thus, condition (*) trivially holds), we can   define $j_{D,M,n}:X_0^D(M)\to J_0^D(M)$ as the map $j_t$ with $t=E^c_{D,M,n}:=E^c_{U_0^D(M),n}$. This need not be an embedding.


\subsection{Shimura curve parameterizations} 

Returning to our setting $U=U_0^D(M)$ and the elliptic curve $A$ arising as an optimal quotient $q:J_0^D(M)\to A$, for each $n$ coprime to $N$ we obtain a map 
$$
\phi_{D,M,n}=qj_{D,M,n}:X_0^D(M)\to A
$$
defined over $\Q$.

\begin{proposition} \label{PropShimuraParam} The morphism $\phi_{D,M,n}$ is non-constant of degree $(a_n(A)- \sigma_1(n))^2\cdot \delta$, where $a_n(A)$ is the $n$-th coefficient of the $L$-function of the elliptic curve $A$.
\end{proposition}
\begin{proof}
Choose any complex point $p_0\in X_0^D(M)^{an}$ and consider the embedding $j_{p_0}:X_0^D(M)^{an}\to J_0^D(M)^{an}$  over $\C$, given by $x\mapsto [x-p_0]$. One easily checks that the $\C$-map $\phi_{p_0}=qj_{p_0}:X_0^D(M)^{an}\to A_\C$ has degree equal to the modular degree $\delta$. 

Let $E_{D,M,n}$ be the image of $E^c_{D,M,n}$ under $\T^c_{D,M}\to\T_{D,M}\subseteq \End(J_0^D(M))$. A direct computation shows that for each complex point $x\in X_0^D(M)^{an}$
$$
j_{D,M,n}(x)=(E_{D,M,n}\circ j_{p_0})(x) + j_{D,M,n}(p_0).
$$
By the Eichler-Shimura congruence relation we have that $E_{D,M,n}$ acts on $A$ as multiplication by $a_n(A)- \sigma_1(n)$. Composing with $q$ we get that for every complex point $x\in X_0^D(M)^{an}$
$$
\phi_{D,M,n}(x)= (a_n(A)- \sigma_1(n))\cdot \phi_{p_0}(x) + \phi_{D,M,n}(p_0).
$$
Note that $a_n(A)- \sigma_1(n)\ne 0$ by Hasse's bound. The result follows. 
\end{proof}

\begin{lemma}\label{LemmanmidGh} For every positive integer $N$ there is a prime $\ell \le 2+ 2\log N$ with $\ell\nmid N$.
\end{lemma}
\begin{proof} This is an elementary fact, cf. \cite{Ghitza}.
\end{proof}
\begin{corollary}\label{CoroDegApproxQ} With the previous notation, there is a non-constant morphism $\phi:X_0^D(M)\to A$ defined over $\Q$ such that $\delta$ divides $\deg \phi$ and moreover $\delta\le \deg \phi\le  (9\log N)^2 \cdot\delta$.
\end{corollary}
\begin{proof} In Proposition \ref{PropShimuraParam} we can choose $n=\ell$ a prime not dividing $N=DM$ of size $\ell \le 2+2\log N< 3\log N$ (as $N\ge 11$). Finally, by Hasse's bound we get
$$
|a_\ell(A) - (\ell+1)|^2\le (\ell +2 \ell^{1/2} +1)^2 < 9\ell^2<81(\log N)^2.
$$
\end{proof}

\subsection{Spectral decomposition of the modular degree}  

Given $[\chi]$ an equivalence class of systems of eigenvalues on $\T_{D,M}$ different to $[\chi_0]$, let us define the \emph{congruence modulus} 
$$
\eta_{[\chi_0]}([\chi]):= [\Z\cdot f : \I_{[\chi]}\cdot f]
$$
where $f$ is any non-zero element of $V_{D,M}^{\chi_0}$ (unique up to scalar). It is easy to see that the primes dividing $\eta_{[\chi_0]}([\chi])$ correspond to congruence primes in a suitable sense, but more is true. This positive integer measures the congruences between the systems of Hecke eigenvalues $\chi$ and $\chi_0$ in the following precise sense.
\begin{proposition}\label{PropBoundCongr} The number $\eta_{[\chi_0]}([\chi])$ is the largest positive integer $m$ with the following property:

Given any polynomial $P\in \Z[x_1,\ldots,x_\ell]$ and positive integers $n_1,\ldots,n_\ell$ coprime to $N$, if 
$$
P(\chi(T_{D,M,n_1}),\ldots,\chi(T_{D,M,n_\ell}))=0, 
$$
then $m$ divides the integer $P(\chi_0(T_{D,M,n_1}),\ldots,\chi_0(T_{D,M,n_1}))$, which equals $P(a_{n_1}(A),\ldots,a_{n_\ell}(A))$.
\end{proposition}
\begin{proof} Clear from the definitions and the relation between the values of $\chi_0$ and the $L$-function of $A$ given by the Eichler-Shimura congruence.
\end{proof}
The previous proposition not only justifies the name of the congruence modulus, but also, it will be useful for estimating the size of $\eta_{[\chi_0]}([\chi])$. In the classical setting of the modular curve $X_0(N)$, the numbers $\eta_{[\chi_0]}([\chi])$ measure congruences between \emph{eigenforms} and can be computed using Fourier expansions (of course, in the more general setting the Fourier expansions are not available).

The relation with the $(D,M)$-modular degree is given by

\begin{theorem}\label{ThmSpectral} The $(D,M)$-modular degree $\delta$ divides 
$$
\prod_{[\chi]\ne [\chi_0]}\eta_{[\chi_0]}([\chi]).
$$
The product runs over all classes $[\chi]$ of systems of Hecke eigenvalues on $\T_{D,M}$ different to $[\chi_0]$.
\end{theorem}

We give two different proofs of this fact, as they can be of independent interest.


\subsection{First proof: torsion of abelian varieties}  Enumerate the classes of systems of Hecke eigenvalues on $\T_{D,M}$ as $c_0,c_1,\ldots,c_s$ with $c_0=[\chi_0]$. Write $A_j$ and $B_j$ instead of $A_{c_j}$ and $B_{c_j}$, in particular $A_0=A$ and $B_0=B$. Define the following abelian sub-varieties of $J_0^{D}(M)$:
$$
C_i:= \sum_{j=i+1}^s B_j, \quad i=-1,0,1,\ldots, s.
$$

For each $0\le j\le s$ define $G_j=B_j\cap C_j$. Note that  $B_j/G_j$ is an abelian variety defined over $\Q$ with a natural $\T_{D,M}$-action, such that the quotient $B_j\to B_j/G_j$ is a $\T_{D,M}$-equivariant isogeny.

Consider the addition maps $\sigma_j: B_j\times C_j\to C_{j-1}$. Geometrically, their kernels are 
$$
\ker(\sigma_j)=\{(P,-P) : P\in G_j(\bar{F})\}.
$$
If we compose the induced isomorphism $C_{j-1}\simeq (B_j\times C_j)/\ker(\sigma_j)$ with the projection onto $B_j/G_j$ and restrict to $B_0\cap C_{j-1}$, then for each  $1\le j\le s$ we obtain a map 
$$
u_j: B_0\cap C_{j-1}\to B_j/G_j
$$
defined over $\Q$.  
Note that $u_j$ is $\T_{\Dfrak,\Mfrak}$-equivariant with respect to the $\T_{D,M}$-actions on $B_0\cap C_{j-1}$  and on $B_j/G_j$. 

\begin{proof}[First proof of Theorem \ref{ThmSpectral}] Let $p$ be a prime number and let $\alpha=v_p(\delta)$. We will show that $p^\alpha$ divides $\prod_{j=1}^s \eta_{c_0}(c_j)$.

 We inductively define algebraic points $Q_j\in J_0^D(M)$ and integers $\gamma_j\ge 0$ for $0\le j\le s$. Let $Q_0\in B_0\cap C_0=B_0[\delta]$ be any algebraic point of order exactly $p^\alpha$ and set $\gamma_j=0$. For $0\le j<s$, let 
$$
\gamma_{j+1}:=\min\{\gamma\ge 0 : p^\gamma Q_j\in C_{j+1}\}\quad\mbox{and}\quad Q_{j+1}:=p^{\gamma_{j+1}}Q_j.
$$
Note that $\gamma_{j+1}$ exists and it is at most $\alpha$, and that for each $j$ we have that $Q_j\in C_j$. Also,  $Q_0$ has order $p^\alpha$ and $Q_s=0$ because $C_s=(0)$. All the points $Q_j$ are multiples of $Q_0$ so that they have order a power of $p$, and more precisely for $1\le j\le s$ we have that 
$$
\ord(Q_{j})=\ord(Q_{j-1})/p^{\gamma_j}.
$$
Hence, for each $0\le j\le s$ we have $\ord(Q_j)=p^{\alpha - \sum_{i\le j} \gamma_i}$ and taking $j=s$ we get $\alpha = \sum_{i=0}^s \gamma_i$. Thus, it suffices to show that $p^{\gamma_j}$ divides $\eta_{c_0}(c_j)=[\Z f : \I_{c_j}f]$ for each $1\le j\le s$.

Let us fix an index $1\le j\le s$. We have $Q_{j-1}\in B_0$ as it is a multiple of $Q_0$. Hence $Q_{j-1}\in B_0\cap C_{j-1}$, thus we can evaluate $u_j$ at $Q_{j-1}$ to get the point $u_j(Q_{j-1})\in B_j/G_j$. 

We claim that $p^{\gamma_j}$ divides $\ord(u_j(Q_{j-1}))$. If $\gamma_j=0$ then there is nothing to prove, so let us consider the case $\gamma_j>0$. Observe that $\ord(u_j(Q_{j-1}))$ is a power of $p$, so it suffices to show that given any integer $0\le \gamma <\gamma_j$ one has $p^\gamma  u_j(Q_{j-1})\ne 0$. For the sake of contradiction, suppose that $p^\gamma  u_j(Q_{j-1})=0$ for some $0\le \gamma < \gamma_j$. Recall that $Q_{j-1}\in C_{j-1}$, and choose some $(Q_{j-1}',Q_{j-1}'')\in \sigma_j^{-1}(Q_{j-1})$. Then $p^\gamma Q_{j-1}'\in G_j$ because $p^\gamma  u_j(Q_{j-1})=0$. This means that $(-p^\gamma Q'_{j-1},p^\gamma Q'_{j-1})\in \ker(\sigma_j)$, and since $(Q_{j-1}',Q_{j-1}'')\in \sigma_j^{-1}(Q_{j-1})$, we deduce that the following point is in $\sigma_j^{-1}(p^\gamma Q_{j-1})\subseteq B_j\times C_j$:
$$
(-p^\gamma Q'_{j-1},p^\gamma Q'_{j-1}) + p^\gamma(Q_{j-1}',Q_{j-1}'') = (0, p^\gamma Q_{j-1}).
$$
 In particular, $p^\gamma Q_{j-1}\in C_j$ which contradicts the minimality of $\gamma_j$. Thus, $p^{\gamma_j}$ divides $\ord(u_j(Q_{j-1}))$.

Finally, take any $t\in \I_{c_j}$ and $0\ne f\in V_{D,M}^{\chi_0}$. Then $tf=\chi_0(t) f$. Since $t$ annihilates $B_j$, it also annihilates $B_j/G_j$, thus $t\cdot u_j(Q_{j-1})=0$. As $Q_{j-1}\in B_0$, we get $tQ_{j-1}=\chi_0(t) Q_{j-1}$ so that 
$$
0=t\cdot u_j(Q_{j-1})=u_j(tQ_{j-1})=u_j(\chi_0(t) Q_{j-1})=\chi_0(t) u_j(Q_{j-1}).
$$
Thus $\ord(u_j(Q_{j-1}))$ divides $\chi_0(t)$, which gives $p^{\gamma_j}|\mu_t$. Therefore $p^{\gamma_j}$ divides $[\Z f : \I_{c_j}f]$.
\end{proof}


\subsection{Second proof: projectors and the Hecke algebra} \label{SecSecondProofDeg}

Write $\E=\End(J_0^D(M))$. Then $\E$ has a right action via pull-back on $H^0(J_0^D(M)^{an},\Omega^1)$ which is canonically isomorphic to $V_{D,M}$. This right action on $V_{D,M}$ extends the action of the commutative subring $\T_{D,M}\subseteq \E$. Since $\E$ acts faithfully on $V_{D,M}$ we can identify $\E^{op}$ with its image in $\End_\C(V_{D,M})$.

Let $\pi_0\in \End_\C(V_{D,M})$ be the orthogonal projection (with respect to the Petersson inner product) onto the subspace $V_{D,M}^{\chi_0}$. Equivalently, $\pi_0$ is the projection onto the factor $V_{D,M}^{\chi_0}$ in the direct sum decomposition $V_{D,M} = \oplus_\chi V_{D,M}^\chi$.
\begin{lemma}\label{LemmaDen} $\pi_0\in \E^{op}\otimes \Q$. Furthermore, the denominator of $\pi_0$ with respect to $\E^{op}$ is the modular degree $\delta$. That is, $\delta$ is the least positive integer $m$ with the property that $m\pi_0\in\E^{op}$.
\end{lemma}
\begin{proof} In the case of the classical modular curve $X_0(N)$ this is proved in \cite{CoKa}, see also \cite{ARSdeg}. The general case is not harder, and we include a proof for the convenience of the reader.

Consider the map $\varpi = q^\vee \circ  q:J_0^D(M)\to J_0^D(M)$ where we recall that $q: J_0^D(M)\to A$. Note that $\varpi\in\E$ and $\varpi |_B:B\to B$ is multiplication by $\delta$, by definition of the modular degree. Taking pull-back of holomorphic differentials and recalling the isogeny decomposition $J_0^D(M)\to \prod_{[\chi]}A_{[\chi]}$ we see that $\varpi=\delta \pi_0$ in $\End_\C(V_{D,M})$.

It only remains to prove that given any positive integer $m$ such that $\varpi_m:=m\pi_0$ is in $\E$, on has that $\delta$ divides $m$. In fact, take such an $m$. From the isogeny decomposition of $J_0^D(M)$ one deduces that $\varpi_m\cdot J_0^D(M)=B\subseteq J_0^D(M)$. Hence, we obtain a map $ J_0^D(M)\to B\simeq A$ defined over $\Q$  whose restriction to $B$ is multiplication by $m$. Recalling the definition of $\delta$ in terms of the optimal quotient map $q$, we see that $\delta$ divides $m$.
\end{proof}

\begin{proof}[Second proof of Theorem \ref{ThmSpectral}] For each system of Hecke eigenvalues $\chi:\T_{D,M}\to\bar{\Q}$, let $L_\chi$ be the totally real number field generated by $\chi(\T_{D,M})$. The rule $t\mapsto (\chi(t))_\chi$ defines a ring morphism
$$
\psi: \T_{D,M}\to \prod_\chi L_\chi
$$
which is injective because $\T_{D,M}$ acts faithfully on $V_{D,M}=\oplus_\chi V_{D,M}^\chi$. 

For each class $c$ of systems of Hecke eigenvalues on $\T_{D,M}$, define the ring $R_c:=\prod_{\chi\in c} L_\chi$ so that we can re-write the previous ring morphism as
\begin{equation}\label{EqRingMap}
\psi: \T_{D,M}\to \prod_{c}R_c.
\end{equation}

Fix a choice of non-zero $f\in V_{D,M}^{\chi_0}$. For each class $c'\ne [\chi_0]$, choose an element $t_{c'}\in \I_{c'}$ such that $t_{c'}\cdot f= \eta_{[\chi_0]}(c')\cdot f$. Then we have that the $[\chi_0]$-component of $\psi(t_{c'})$ in \eqref{EqRingMap} is $\eta_{[\chi_0]}(c')$, while the $c'$-component of $\psi(t_{c'})$ is $0$.

Finally, let $\varpi =\prod_{c\ne [\chi_0]} t_c\in \T_{D,M}\subseteq \E$ and observe that $\psi(\varpi)\in \prod_c R_c$ has the integer $\prod_{c\ne [\chi_0]} \eta_{[\chi_0]}(c)$ in the $[\chi_0]$-component, and $0$ in all other components. It follows that 
$$
\varpi = \left(\prod_{c\ne [\chi_0]} \eta_{[\chi_0]}(c) \right) \pi_0,
$$ 
and by Lemma \ref{LemmaDen} we get that $\delta$ divides $\prod_{c\ne [\chi_0]} \eta_{[\chi_0]}(c)$.
\end{proof}


\section{A refinement of the Ribet-Takahashi formula}

\subsection{Notation}\label{NotationRT} Consider an elliptic curve $E$ defined over $\Q$ with conductor $N=N_E$. Recall that for each admissible factorization $N=DM$ we have the optimal quotients $q_{D,M}:J_0^D(M)\to A_{D,M}$ with modular degree $\delta_{D,M}$ respectively. The elliptic curves $A_{D,M}$ are isogenous over $\Q$ to $E$, although they need not be isomorphic. 

In this section we consider $E$ (and $N$) as varying. So, the implicit constants in error terms will always be independent of $E$. In fact, as always in this paper, any dependence of the implicit constants on other parameters will be indicated explicitly as a subscript.


\subsection{The formula} Ribet and Takahashi (cf. \cite{RiTa} and \cite{Takahashi}) proved a formula for the fraction $\delta_{1,N}/\delta_{D,M}$, which in the case when $E$ has no non-trivial rational cyclic isogeny and $M$ is squarefree but not prime, reads
$$
\frac{\delta_{1,N}}{\delta_{D,M}}=\prod_{p|D} v_p(\Delta_E).
$$
The requirement that $M$ be squarefree can be relaxed to some extent.  However, the above equality is known to be false in cases when $E$ has non-trivial isogenies, and the  formula of Ribet and Takahashi in the general case includes a correction factor which takes into account reducible residual Galois representations. 

The correction factor is a rational number supported on primes $\ell$ for which the Galois representation $E[\ell]$ is reducible, although the exponents of those primes are not controlled in the literature. 

For our purposes, we will need a version of the Ribet-Takahashi formula with finer control on this correction factor, not just the primes of its support. In our result, additional efforts are made in order to have refined estimates for semi-stable elliptic curves and for Frey-Hellegouarch elliptic curves, which are our main cases of interest ---nevertheless, we also provide good estimates in the general case.

For an elliptic curve $E$ over $\Q$ of conductor $N$ with an admissible factorization $N=DM$, we define the positive rational number $\gamma_{D,M, E}$ by the formula
\begin{equation}\label{EqDefGamma}
\frac{\delta_{1,N}}{\delta_{D,M}}=\gamma_{D,M, E} \cdot \prod_{p|D} v_p(\Delta_E).
\end{equation}
This is the correction factor that we need to control. 

For a positive rational number $x\in \Q_{>0}$ the numerator and denominator of $x$ are the unique coprime positive integers $a,b$ (respectively) with $x=a/b$. With this notation, we prove:

\begin{theorem}\label{ThmRT} Let $E$ be an elliptic curve over $\Q$ of conductor $N$ and let $N=DM$ be an admissible factorization. The numerator of $\gamma_{D,M,E}$ is supported on primes $\le 163$ and it is bounded from above by $163^{\omega(D)}$. In particular, we have
\begin{equation}\label{EqUpperRT}
\log \delta_{1,N} \le  \log\delta_{D,M}  + \log\left( \prod_{p|D} v_p(\Delta_E) \right) + 5.1\cdot \omega(D).
\end{equation}
Furthermore, the following upper bounds for the denominator of $\gamma_{D,M,E}$ hold:

\begin{itemize}
\item[(a)]  Let $S$ be a finite set of primes. There is a positive integer $\kappa_S$ depending only on $S$ and supported on primes $\le 163$ such that if $E$ is semi-stable away from $S$, then the denominator of $\gamma_{D,M,E}$ is less than or equal to $\kappa_S^{\omega(D)} D$. In particular, 
$$
 \log\left( \prod_{p|D} v_p(\Delta_E) \right) \le \log \delta_{1,N} -  \log\delta_{D,M} + \log D + O_S(\omega(D))
$$
where the implicit constant only depends on $S$.

\item[(b)]  There is an absolute integer constant $\kappa\ge 1$ supported on primes $\le 163$ which satisfies the following: Suppose that either 
\begin{itemize}
\item[(b.1)] $E$ is semi-stable and $M$ is not a prime number; or
\item[(b.2)] $E$ is a Frey-Hellegouarch elliptic curve and $M$ is divisible by at least two odd primes.
\end{itemize}
Then the denominator of $\gamma_{D,M,E}$ divides $\kappa^{\omega(D)}$. In particular,
$$
 \log\left( \prod_{p|D} v_p(\Delta_E) \right) \le \log \delta_{1,N} -  \log\delta_{D,M} +  O(\omega(D))
$$
where the implicit constant is absolute.
\end{itemize}
\end{theorem}
An immediate consequence is the following:
\begin{corollary}\label{CoroRT} Let $E$ be an elliptic curve of conductor $N$ and consider an admissible factorization $N=DM$. Suppose that either 
\begin{itemize}
\item[(i)] $E$ is semi-stable and $M$ is not a prime number; or
\item[(ii)] $E$ is a Frey-Hellegouarch elliptic curve and $M$ is divisible by at least two odd prime numbers.
\end{itemize}
Then we have that the logarithmic height of the rational number $\gamma_{D,M,E}$ is
$$
h(\gamma_{D,M,E})\ll \omega(D)\ll \frac{\log D}{\log \log D}
$$
and in particular
$$
 \log\left( \prod_{p|D} v_p(\Delta_E) \right) = \log \delta_{1,N} -  \log\delta_{D,M} +  O\left(\frac{\log D}{\log \log D}\right)
$$
where the implicit constants are absolute.
\end{corollary}
At this point, let us make a heuristic (and admittedly naive) remark. The estimates for $\gamma_{D,M,E}$ in the previous results say $\delta_{1,N}/\delta_{D,M}$ is approximately equal to $\prod_{p|D}v_p(\Delta_E)$ (after taking logarithms). This can be seen as an \emph{arithmetic analogue of partial derivatives} in the following sense: 

Given a monomial $x_1^{e_1}\cdots x_n^{e_n}$ and a subset $J\subseteq \{1,\ldots,n\}$ we can recover the product of the exponents $e_j$ for $j\in J$ as a ``rate of change in the direction of the selected variables $x_j$ for $j\in J$''
$$
\prod_{j\in J}e_j = \frac{\partial^{| J |}}{\prod_{j\in J}\partial x_j} x_1^{e_1}\cdots x_n^{e_n}|_{{\bf x}=(1,...,1)}.
$$


If we write the factorization $\Delta_E=\prod_{p|N} p^{v_p(\Delta_E)}$ and consider the expression on the right as a monomial, then ``formal partial derivatives with respect to the primes $p|D$ '' would produce the factor  $\prod_{p|D} v_p(\Delta_E)$. On the other hand, the fraction $\delta_{1,N}/\delta_{D,M}$ can be seen as a ``rate of change in the direction of the selected primes $p|D$'', since it measures the variation of the degree of a modular parametrization of the elliptic curve $E$ under consideration when we change the modular curve $X_0(N)$ by the Shimura curve $X_0^D(M)$.

While this heuristic might be relevant in the context of the usual analogies between function fields and number fields, we remark that the arguments in this paper do not depend on it. 

\subsection{Lemmas on congruences}

If $A$ is an elliptic curve over $\Q$ and $m$ is a positive integer, the Galois representation on $A[m]$ is denoted by
$$
\rho_{A[m]}:G_\Q\to GL_2(\Z/m\Z).
$$
This depends on a choice of isomorphism $A[m]\simeq \Z/m\Z\times\Z/m\Z$, and the image of $\rho_{A[m]}$ is well-defined up to conjugation. 

\begin{lemma}\label{LemmaCongIrr} Let  $\ell>163$ be a prime. For every elliptic curve $A$ over $\Q$ of conductor $N$ there are infinitely many primes $r\nmid \ell N$ satisfying
$$
a_r(A)\not\equiv r+1 \mod \ell.
$$
\end{lemma}
\begin{proof} As  $\ell>163$, from \cite{MazurRatIsog} we have that $A[\ell]$ is an irreducible Galois module, and the claim follows from the proof of Theorem 5.2(c) in \cite{RibetInv100}. 
\end{proof}
\begin{lemma} \label{LemmaCongEff} Let $\ell\le 163$ be a prime and let $A$ be an elliptic curve of conductor $N$. There is a positive integer $e=e(\ell, A)$ satisfying:
\begin{itemize}
\item $\ell^e<2050\log(N)$
\item There are infinitely many primes $r\nmid \ell N$ such that  
$$
a_r(A)\not\equiv r+1 \mod \ell^e.
$$
\end{itemize}
\end{lemma}
\begin{proof} By Lemma \ref{LemmanmidGh} there is a prime $r_0\nmid \ell N$ with 
$$
r_0\le 2+2\log(\ell N)\le 2+ 2\log(163 N).
$$
Take $e$ so that $\ell^e \ge 4\log(163 N)$, then $\ell^e > |r_0+1 - a_{r_0}(A)|\ne 0$ (by Hasse's bound) obtaining $\ell^e\nmid r_0+1-a_{r_0}(A)$. As $\ell\le 163$ and $N\ge 11$, we see that this can be achieved with $\ell^e< 2050\log(N)$.

The existence of this particular $r_0$ implies the existence of infinitely many primes $r$ as wanted, by Chebotarev's theorem applied to the character
$$
\mathrm{Tr}\, \rho_{A[\ell^e]} - \mathrm{Tr}\, (1\oplus \chi_{cyc})
$$
where $\chi_{cyc}:G_\Q\to (\Z/\ell^c\Z)^\times$ is the cyclotomic character.
\end{proof}
\begin{lemma} \label{LemmaFaltingsMod} Let $\ell$ be a prime. There is a positive integer $c(\ell)$ depending only on $\ell$ such that for any given $n\ge c(\ell)$ there is a finite set $\Fcal(\ell,n)\subseteq \Q$ with the following property:

Let $A$ be an elliptic curve defined over $\Q$ with $j$-invariant $j_A$. If $j_A\notin \Fcal(\ell,n)$, then
\begin{equation}\label{EqBig}
im(\rho_{A[\ell^n]})\supseteq \{\gamma\in SL_2(\Z/\ell^n\Z) : \gamma\equiv I \mod \ell^{c(\ell)}\}.
\end{equation}
\end{lemma}
\begin{proof} For each positive integer $m$  and each subgroup $H\le GL_2(\Z/m\Z)$ such that $\det (H)=(\Z/m\Z)^\times$, there is a congruence subgroup $\Gamma_H\le SL_2(\Z)$ and a geometrically connected,  open congruence modular curve $Y_H$ defined over $\Q$ with a map $\pi_H : Y_H\to Y(1)=\A^1$ over $\Q$. The set of complex points of $Y_H$ is $\Gamma_H\backslash \hfrak$. Distinct $m$ and $H$ can give the same $\Gamma_H$. It is a classical result that for congruence subgroups the genus grows with the level (cf. \cite{CoxParry}). 

If $A$ is an elliptic curve over $\Q$ with $\rho_{A[m]}(G_\Q)$ conjugate to a subgroup of such an $H$, then it gives rise to a $\Q$-rational point $P_H(A,m)\in Y_H(\Q)$ satisfying $\pi_H(P_H(A,m))=j_A$.

Given a positive integer $m$ and an elliptic curve $A/\Q$  we define the subgroup 
$$
H_{A,m}= \rho_{A[m]}(G_\Q)\le GL_2(\Z/m\Z).
$$ 
Observe that $\det(H_{A,m})=(\Z/m\Z)^\times$ because $\det \rho_{A[m]}(Frob_p)=p\mod m$ for all but finitely many primes $p$, and note also that $P_{H_{A,m}}(A,m)\in Y_{H_{A,m}}(\Q)$.

Given a positive integer $c$ and taking  $m=\ell^n$ for some $n\ge c$, the failure of \eqref{EqBig} can be expressed by saying that $H_{A,\ell^n}$ does not contain the group
$$
S_{c,\ell^n}:=\{\gamma\in SL_2(\Z/\ell^n\Z) : \gamma\equiv I \mod \ell^c\}.
$$
We note that the level of $\Gamma_{H_{A,\ell^n}}$ is certain power of $\ell$. If \eqref{EqBig} fails, then the level of $\Gamma_{H_{A,\ell^n}}$ is larger than $\ell^c$, for otherwise $H_{A,\ell^n}$ would contain $S_{c,\ell^n}$. Thus, suitable choice of $c$ will ensure that whenever \eqref{EqBig} fails for $A$ and $n\ge c$, one has that $Y_{H_{A,\ell^n}}$ has geometric genus at least $2$. For $\ell$ and $n\ge c$ fixed, there are only finitely many groups $H_{A,\ell^n}\le GL_2(\Z/\ell^n\Z)$ as we vary $A$, all of them giving modular curves of geometric genus at least $2$. Let $H_1,...H_t$ be these finitely many subgroups, then we can take $\Fcal(\ell,n)=\cup_{i=1}^t \pi_{H_i}(Y_{H_i}(\Q))$, which is finite by Faltings's theorem. We take $c(\ell)=c$.
\end{proof}

\begin{lemma}\label{LemmaMatrix} Let $\ell$ be a prime. There is a positive integer $b=b_\ell$ and a finite set $\Gcal_\ell\subseteq \Q$, both depending only on $\ell$, such that the following holds:

Let $A$ be an elliptic curve over $\Q$ with $j_A\notin \Gcal_\ell$. There are infinitely many primes $r\nmid N\ell$ with $a_r(A)\not\equiv r+1 \mod \ell^{b}$.
\end{lemma}
\begin{proof}  Let $c=c(\ell)$ be as in Lemma \ref{LemmaFaltingsMod}. We claim that $b=4c$ and $\Gcal_\ell=\Fcal(\ell,b)$ have the desired property. For the sake of contradiction, let $A$ be an elliptic curve over $\Q$ with $j_A\notin \Gcal_\ell$ and suppose that all sufficiently large primes $r$ satisfy 
$$
X^2-a_r(A)X+r\equiv (X-1)(X-r) \mod \ell^{b}.
$$
The Cayley-Hamilton theorem applied to the action on $A[\ell^b]$ of the Frobenius element $F_r=\rho_{A[\ell^b]}(Frob_r)$,  gives that $(F_r- 1)\circ (F_r - r)$ is the zero endomorphism on $A[\ell^b]$ for all but finitely many primes $r$. Note that by Lemma \ref{LemmaFaltingsMod}
$$
\gamma:=\left[
\begin{array}{cc}
1+\ell^{2c} &\ell^c\\
\ell^c &1
\end{array}
\right]\in im(\rho_{A[\ell^b]})
$$
so the Chebotarev density theorem gives infinitely many primes $r$ for which $F_r=\gamma$. We have
$$
\ker(\gamma-1)=\ker \left[
\begin{array}{cc}
\ell^{2c} &\ell^c\\
\ell^c &0
\end{array}
\right] = \ell^{b-c}A[\ell^b]
$$
so that $\# \ker(\gamma-1)= \ell^{2c}$. On the other hand, we have that
$$
\gamma-r=\left[
\begin{array}{cc}
1-r+\ell^{2c} &\ell^c\\
\ell^c &1-r
\end{array}
\right] 
$$
has image of size at least $\ell^{b-c}$ (by looking at the top-right entry, say), so that $\#\ker(\gamma-r)\le \ell^{b+c}$. It follows that for each of the infinitely many primes $r$ with $F_r=\gamma$ we have
$$
\ell^{2b}=\#\ker((F_r-1)\circ (F_r-r))\le \#\ker(\gamma-1)\cdot\#\ker(\gamma-r)\le \ell^{b+3c} 
$$
which is not possible, because $b=4c$. 
\end{proof}
\begin{lemma} \label{LemmaCongIneff} Let $S$ be a finite set of primes. For all primes $\ell$ there is an integer $\beta_S(\ell)$ with the following properties:
\begin{itemize}
\item[(i)] If $\ell> 163$, then we have $\beta_S(\ell)=1$.
\item[(ii)] For every elliptic curve $A/\Q$ semi-stable away from $S$, there are infinitely many primes $r$ such that $a_r(A)\not\equiv r+1\mod \ell^{\beta_S(\ell)}$.
\end{itemize}
\end{lemma}
\begin{proof}  Given a finite set of primes $S$ and a finite set of rational numbers $J$, there are only finitely many elliptic curves over $\Q$ with $j$-invariant in $J$ and semi-stable reduction outside $S$. Then, up to finitely many elliptic curves over $\Q$ (depending on $S$ and $\ell$), the result now follows from Lemmas \ref{LemmaCongIrr} and \ref{LemmaMatrix}. For those possible finitely many exceptional elliptic curves, we conclude by Lemmas \ref{LemmaCongIrr} and \ref{LemmaCongEff}.
\end{proof}

\subsection{Component groups} Let $A$ be an abelian variety over $\Q$ and let $\Acal$ be its N\'eron model. For a prime $p$, we denote by $\Phi_p(A)$ the group of geometric connected components of $\Acal_p$, the special fibre at $p$. Then $\Phi_p(A)$ is a finite abelian group with a $G_{\F_p}$-action, and the rule $A\mapsto \Phi_p(A)$ is functorial. Given $\theta:A\to B$ a morphism of abelian varieties over $\Q$, we write $\theta_{p,*}: \Phi_p(A)\to \Phi_p(B)$ for the induced map. In the particular case of multiplication by an integer $n$, that is $[n]:A\to A$, we have that $[n]_{p,*}$ is multiplication by $n$ on $\Phi_p(A)$.

Let $X_p(A)$ be the character group $Hom_{\bar{\F}_p}(\tau_p(A),\G_m)$ where $\tau_p(A)$ is the toric part of $\Acal_p$. One has the monodromy pairing $u_{A,p}:X_p(A)\times X_p(A^\vee)\to \Z$ which, in the case when $A$ comes with an isomorphism $A\simeq A^\vee$ (e.g. when $A$ is the Jacobian of a curve, in particular when $A$ is an elliptic curve) becomes a $\Z$-valued bilinear pairing on $X_p(A)$.

Suppose now that $A$ is an elliptic curve over $\Q$ with multiplicative reduction at $p$. Then $\Phi_p(A)$ is a cyclic group of order $c_p(A):=v_p(\Delta_A)$.

\begin{lemma} \label{LemmaHtIsog} Let $A$ and $B$ be elliptic curves which are isogenous  over $\Q$ and suppose that $p$ is a prime of multiplicative reduction for one (hence both) of them. Then $c_p(A)/c_p(B)$ is a rational number whose multiplicative height is at most $163$. In particular, it is supported on primes $\le 163$.
\end{lemma}
\begin{proof} Let $\alpha:A\to B$ be an isogeny of minimal degree; be results of Mazur \cite{MazurRatIsog} and Kenku \cite{Kenku} we know that $n:=\deg(\alpha) \le 163$. Let $\beta:B\to A$ be the dual isogeny, so that $\beta\alpha=[n]$ on $A$. Since $A$ and $B$ have multiplicative reduction at $p$, we have isomorphisms of abstract groups $\Phi_p(A)=\Z/c_p(A)\Z$ and $\Phi_p(B)=\Z/c_p(B)\Z$, under which we obtain a commutative diagram
\begin{eqnarray*}\begin{tikzpicture}[scale=1]
\node (A) at (0,1) {$\Z/c_p(A)\Z$};
\node (B) at (3,1) {$\Z/c_p(B)\Z$};
\node (C) at (3,0) {$\Z/c_p(A)\Z$.};
\path[->,font=\scriptsize,>=angle 90]
(A) edge node[above]{$\alpha_{p,*}$} (B)
(A) edge node[below]{$n\cdot$} (C)
(B) edge node[right]{$\beta_{p,*}$} (C);
\end{tikzpicture}\end{eqnarray*}
From this we get that $c_p(A)/(n,c_p(A))=\# im(n\cdot)$ divides $\# im(\beta_{p,*})$, which divides $c_p(B)$. Thus, the numerator of $c_p(A)/c_p(B)$ divides $n$, and similarly for the denominator using $\alpha\beta$ instead.
\end{proof}


\subsection{Some Diophantine equations} Our proof of Theorem \ref{ThmRT} requires control on integer solutions of certain Diophantine problems which are variations of Fermat's last theorem. 
\begin{lemma} \label{LemmaDG1} Let $S$ be a finite set of primes. Let $A,B,C$ be non-zero integers. Let $p,q,r$ be positive integers satisfying 
$$
\frac{1}{p} + \frac{1}{q} + \frac{1}{r} <1.
$$
 The equation
$$
Ax^p + By^q = Cz^r
$$
has only finitely many integer solutions $(x,y,z)\in \Z^3$ under the restriction that $x,y,z$ do not have common prime factors outside $S$.
\end{lemma}
\begin{proof} The proof is the same as for Theorem 2 in \cite{DarmonGranville}. One has to repeat the argument given in p.526-527 of \emph{loc. cit.} replacing the set of places $V_{ABC}$ by $V_{ABC}\cup S$.  
\end{proof}
\begin{lemma}\label{LemmaDG2} Let $L\ge 7$ be an integer and let $S$ be a finite set of primes. There is a number $N_0=N_0(L,S)$ depending only on $L$ and $S$ such that there is no elliptic curve $E$ over $\Q$ with conductor $N_E\ge N_0$, semi-stable reduction away from $S$, and with minimal discriminant of the form $\Delta_E= n\cdot k^L$ with $n$ and $k$ integers such that all the prime factors of $n$ are in $S$.
\end{lemma}
\begin{proof} For $E$ semi-stable away from $S$, the quantities $c_4$ and  $c_6$ associated to a minimal Weierstrass equation of $E$ can only share prime factors from $S'=S\cup\{2,3\}$, cf. Exercise 8.21 in \cite{Silverman2Ed}. Recall that $\pm 1728 \Delta_E= c_4^3-c_6^2$ (cf. p.259 \emph{loc. cit.}), then elliptic curves $E$ with $\Delta_E= n\cdot k^L$ for some integers $n,k$ with $n$ an $S$-unit give integer solutions for equations of the form 
$$
A \cdot x^L = y^3-z^2
$$
where the integer $A\ne 0$ is only divisible by primes in $S'$, and $\gcd(y,z)$ is only divisible by primes in $S'$. Furthermore, for $S$ and $L$ fixed, one only needs to consider finitely many equations of this sort, as it is only necessary to consider coefficients $A$ whose prime factorization has exponents $<L$ (this reduction might modify the coordinate $x$ of a solution, but it does not modify the coordinates $y,z$). Since $1/2+1/3+1/L>1$, Lemma \ref{LemmaDG1} gives that each one of these finitely many equations has only finitely many solutions, and we obtain only a finite list of possible quantities $c_4$ and $c_6$ (coming from the coordinates $y,z$ of these finitely many solutions). 

Thus, for fixed $S$ and $L$, there are only finitely many elliptic curves $E$ over $\Q$ with $\Delta_E= n\cdot k^L$ for some integers $n,k$ with $n$ an $S$-unit.
\end{proof}

\begin{lemma} \label{LemmaDiophSS}
Let $\ell\ge 11$ be a prime number. Let $E$ be a semi-stable elliptic curve over $\Q$. Then  $\Delta_E$ is not a perfect $\ell$-th power.
\end{lemma}
\begin{proof} Since $E$ is semi-stable, the residual Galois representation $\rho_{E[\ell]}$ is surjective onto $GL_2(\F_\ell)$ by Theorem 4 in \cite{MazurRatIsog}. Hence it is absolutely irreducible and we can apply Ribet's level-lowering results \cite{RibetInv100} as in the proof of Fermat's last theorem. Thus $\Delta_E$ is not a perfect $\ell$-th power. (See also Lemma 2 in \cite{RiTa}.)
\end{proof}

\begin{lemma}\label{LemmaGenFermat} Let $a,b,c$ be coprime positive integers with $a+b=c$ and $(a,b,c)\ne (1,1,2)$. Let $E$ be the associated Frey-Hellegouarch elliptic curve and write $2^{v_2(\Delta_E)}\Delta'_D = \Delta_E$ (that is, $\Delta'_E$ is the odd part of $\Delta_E$). Let $\ell\ge 3$ be a prime number. Then $\Delta'_E$ is not a perfect $\ell$-th power.
\end{lemma}
\begin{proof} Suppose that $\Delta'_E$ is a perfect $\ell$-th power. Recall that $\Delta_E=2^s(abc)^2$ for some integer $s$ with $-8\le s\le 4$ (cf. Section \ref{SecClassical}). Since $a,b,c$ are pairwise coprime we see that the odd parts of them are perfect $\ell$-th powers. Exactly one of $a,b,c$ is even, so the equation $a+b=c$ yields a solution of
$$
x^\ell + 2^m y^\ell +z^\ell=0
$$
in pairwise coprime integers $x,y,z$ with $x$ and $z$ odd, $xyz\ne 0$, and $0\le m < \ell$. 

By Wiles's theorem we see that $m\ne 0$.

By Ribet's Theorem 3 in \cite{RibetOnTheEqn} we see that $1<m<\ell$ is not possible either.

So $m=1$. Then by Darmon and Merel \cite{DarmonMerelAP} we obtain that the solution must be trivial in the sense that $xyz=\pm 1$. This has the consequence that $abc=2$, hence, $(a,b,c)=(1,1,2)$, which was excluded.
\end{proof}


\subsection{Set-up for the proof of Theorem \ref{ThmRT}}  For given elliptic curve $E$ of conductor $N$ we will vary over admissible factorizations $N=DM$. This has the effect that in the discussion below, our notation will not explicitly refer to $E$ but instead we only keep track of the admissible factorization under consideration. 

Given an admissible factorization $N=DM$ and a prime $p$, the map $q_{D,M}: J_0^D(M)\to A_{D,M}$  induces the map
$$
q_{D,M,p,*}: \Phi_p(J_0^D(M)) \to \Phi_p(A_{D,M})
$$
on the groups of geometric components, and we will need to analyze the size of the image and cokernel of these maps for various primes $p$:
$$
\begin{aligned}
i_p(D,M) & = \# \mathrm{image} (q_{D,M,p,*})\\
j_p(D,M) & = \# \mathrm{cokernel} ( q_{D,M,p,*} ).
\end{aligned}
$$
A first result towards Theorem \ref{ThmRT} is the following one, already included in \cite{RiTa}.
\begin{proposition}\label{PropRTprelim}
Let $N=DM$ be an admissible factorization of the conductor of $E$, and suppose that $p,r$ are two distinct primes dividing $D$. Let us write $D=dpr$. Then
$$
\begin{aligned}
\frac{\delta_{d,prM}}{\delta_{dpr,M}} & = \frac{c_p(A_{d,prM})\cdot c_r(A_{dpr,M})}{i_p(d,prM)^2\cdot j_r(dpr,M)^2}\\
& = \frac{c_r(A_{d,prM})\cdot c_p(A_{dpr,M})}{i_r(d,prM)^2\cdot j_p(dpr,M)^2}.
\end{aligned}
$$
\end{proposition}
\begin{proof}
Except for the notation, this is Theorem 2 in \cite{RiTa} ---symmetry on $p$ and $r$ follows from the fact that $\delta_{d,prM}/\delta_{dpr,M}$ is symmetric on $p$ and $r$.

 At this point one must note that the proof in \emph{loc. cit.} does not need $M$ to be squarefree, specially, Proposition 1 does not use this assumption and one just needs multiplicative reduction at the two primes  $p,q$ according to the notation of \emph{loc. cit.} (In fact, \emph{loc. cit.} only uses the assumption that $M$ is squarefree in p.11113 for the proof of the second part of Theorem 1.)
\end{proof}
The main technical difficulty in the proof of Theorem \ref{ThmRT} is to show that the terms $i_p(d,prM)$ and $ j_r(dpr,M)$ from the previous result are ``small'', even when the relevant elliptic curves have some reducible residual Galois representation. At this point we can already take care of the image term.

\begin{lemma} \label{LemmaBdImage} Let $S$ be a finite set of primes. If $N=DM$ is squarefree away from $S$ and $p$ exactly divides $M$, then for every prime $\ell$ we have 
$$
v_\ell(i_p(D,M))\le \beta_S(\ell)-1
$$
with $\beta_S(\ell)$ as in Lemma \ref{LemmaCongIneff}. In particular, if $N=DM$ is squarefree away from $S$, and $p$ exactly divides $M$, then $i_p(J_0^D(M),\chi_{D,M})$ divides an integer $\kappa_S$ which only depends on the set $S$ and, moreover, $\kappa_S$ is supported on the primes $\le 163$.
\end{lemma}
Before proving this, we remark that \cite{RiTa} and \cite{Takahashi} omit the analysis of the primes $\ell$ for which the Galois representation $E[\ell]$ is reducible. In our case this is a very serious issue; for instance, $E[2]$ is always reducible for Frey-Hellegouarch curves. Since our ultimate goal is to establish global bounds, we cannot omit the contribution of any prime.
\begin{proof}[Proof of Lemma \ref{LemmaBdImage}]
We follow the idea of the proof of Proposition 3 in \cite{RiTa}. The group $\Phi_p(J_0^D(M))$ is Eisenstein in the sense that for $r\nmid N$, the Hecke operator $T_r$ acts on it as multiplication by $r+1$ (cf. \cite{RibetEisenstein}). On the other hand, since $A_{D,M}$ is the optimal quotient associated to  $\chi_{D,M}$, the action of $T_r$ on $J_0^D(M)$ induces multiplication by $\chi_{D,M}(T_r)=a_r(A_{D,M})$ on $A_{D,M}$. Hence, $r+1-a_r(A_{D,M})$ acts as $0$ on $im(\xi_{p,*})$, which is a cyclic group of order $i_p(J_0^D(M),\chi_{D,M})$ because it is a subgroup of the cyclic group $\Phi_p(A_{D,M})$. 

It follows that $i_p(J_0^D(M),\chi_{D,M})$ divides $r+1-a_r(A_{D,M})$ for every prime $r\nmid N$. Let $\ell$ be a prime, then, under the assumption that $N$ is squarefree away from $S$,  Lemma \ref{LemmaCongIneff} affords infinitely many primes $r$ for which $v_\ell(r+1-a_r(A_{D,M}))<\beta_S(\ell)$. Taking any of these primes $r$ gives the result.
\end{proof}

The cokernel term, however, is much more delicate and we analyze it the next paragraph. 

We remark that in Theorem 2.4 of \cite{Takahashi} it is asserted that $q_{D,M,p,*}: \Phi_p(J_0^D(M)) \to \Phi_p(A_{D,M})$ is surjective when $p|D$. This would imply that $j_r(D,M)=1$. Unfortunately the proof in \cite{Takahashi} has a serious gap as explained in detail in \cite{PaRa}, see the last paragraph of Section 1 in \cite{PaRa}. This gap affects the main result of \cite{Takahashi} ---we leave it to the reader to see what other references in the literature are affected by this issue. For our purposes this is not a serious problem, and we  can control the cokernel by other means.


\subsection{Switching primes} For simplicity, we write $\Pcal=\{2,3,5,...\}$ for the set of prime numbers.
\begin{lemma} \label{LemmaSwitchOnD} Let $S$ be a finite set of primes. There is a function $\alpha_{S,1}:\Pcal\to\Z_{\ge 0}$ supported on primes $\le 163$ (i.e. taking the value $0$ at each prime larger than $163$) and depending only on the choice of $S$, such that the following holds:

Let $E$ be an elliptic curve over $\Q$, semi-stable away from $S$, and of conductor $N$. Let $N=DM$ be an admissible factorization. If $p,r$ are two (possibly equal) primes dividing $D$, then for every prime $\ell$ we have
$$
v_\ell(j_p(D,M))\le v_\ell(c_r(E))+ \alpha_{S,1}(\ell).
$$
\end{lemma}
\begin{proof} If $p=r$, then by definition $j_p(D,M)$ divides $c_p(A_{D,M})=\#\Phi_p(A_{D,M})$, so the result follows from Lemma \ref{LemmaHtIsog}.

Suppose now that $p\ne r$.  By Proposition \ref{PropRTprelim}, Lemma \ref{LemmaHtIsog}, and Lemma \ref{LemmaBdImage}, we obtain
$$
\left| v_\ell(j_p(D,M)) - v_\ell(j_r(D,M)) \right| \le \alpha_{S}(\ell)
$$
for some function $\alpha_S : \Pcal\to\Z_{\ge 0}$ supported on primes $\le 163$. The desired estimate now follows from the previous case.
\end{proof}
\begin{lemma} \label{LemmaSwitchOnM} Let $S$ be a finite set of primes. There is a function $\alpha_{S,2}:\Pcal\to\Z_{\ge 0}$ supported on primes $\le 163$ and depending only on the choice of $S$, such that the following holds:

Let $E$ be an elliptic curve over $\Q$, semi-stable away from $S$, and of conductor $N$. Let $N=DM$ be an admissible factorization and suppose that $M$ is divisible by at least two primes of multiplicative reduction for $E$. If $p,r$ are two primes of multiplicative reduction for $E$ with $p|D$ and $r|M$, then for every prime $\ell$ we have
$$
v_\ell(j_p(D,M))\le v_\ell(c_r(E))+ \alpha_{S,2}(\ell).
$$
\end{lemma}
\begin{proof} Since $D$ has an even number of prime factors and by our hypothesis on $M$, there are primes $q,t$ of multiplicative reduction for $E$ with $q\ne p$ and $t\ne r$, such that $q|D$ and $t|M$. Let us write $D=pqd$ and $M=rtm$.

 From the equations
$$
\frac{\delta_{d,pqrtm}}{\delta_{pqd,rtm}}\cdot \frac{\delta_{pqd,rtm}}{\delta_{pqrtd,m}} = \frac{\delta_{d,pqrtm}}{\delta_{pqrtd,m}} = \frac{\delta_{d,pqrtm}}{\delta_{prd,qtm}} \cdot \frac{\delta_{prd,qtm}}{\delta_{pqrtd,m}}
$$
and Proposition \ref{PropRTprelim}, we deduce that the expression 
\begin{equation}\label{EqEXP1}
\frac{c_q(A_{d,pqrtm})\cdot c_p(A_{pqd,rtm})}{i_q(d,pqrtm)^2\cdot j_p(pqd,rtm)^2} \cdot \frac{c_r(A_{pqd,rtm})\cdot c_t(A_{pqrtd,m})}{i_r(pqd,rtm)^2\cdot j_t(pqrtd,m)^2}
\end{equation}
is equal to the expression
\begin{equation}\label{EqEXP2}
\frac{c_p(A_{d,pqrtm})\cdot c_r(A_{prd,qtm})}{i_p(d,pqrtm)^2\cdot j_r(prd,qtm)^2} \cdot \frac{c_q(A_{prd,qtm})\cdot c_t(A_{pqrtd,m})}{i_q(prd,qtm)^2\cdot j_t(pqrtd,m)^2},
\end{equation}
and we observe that $j_t(pqrtd,m)$ appears in both.

By Lemma \ref{LemmaHtIsog}, Lemma \ref{LemmaBdImage}, and the equality of \eqref{EqEXP1} and \eqref{EqEXP2}, we see that there is a function $\alpha'_S:\Pcal\to \Z_{\ge 0}$ depending only on the choice of $S$ and supported on primes $\le 163$ such that for every prime $\ell$ we have
$$
\left|  v_\ell(j_p(pqd,rtm)) - v_\ell(j_r(prd,qtm))  \right| \le \alpha'_S(\ell).
$$
From this and Lemma \ref{LemmaSwitchOnD} (on the prime $r|prd$) we deduce
$$
v_\ell(j_p(pqd,rtm)) \le v_\ell(j_r(prd,qtm)) + \alpha'_S(\ell)\le v_\ell(c_r(E)) + \alpha'_S(\ell) + \alpha_{S,1}(\ell).
$$
\end{proof}


\subsection{Bounding the cokernel} 

The following result shows that $j_p(D,M)$ is uniformly bounded for $p|D$ in the cases on which we are mainly interested, and without assuming irreducibility of residual Galois representations.
\begin{theorem} \label{ThmUnifCoker} There is a function $\alpha:\Pcal\to\Z_{\ge 0}$ supported on primes $\le 163$, such that the following holds:

Let $E$ be an elliptic curve over $\Q$ of conductor $N$. Let $N=DM$ be an admissible factorization and suppose that either
\begin{itemize} 
\item[(i)] $E$ is semi-stable and $M$ is not a prime number, or
\item[(ii)] $E$ is a Frey-Hellegouarch elliptic curve and $M$ is divisible by at least two odd primes (which are necessarily of multiplicative reduction).
\end{itemize}
Let $p$ be a prime with $p|D$. Then for every prime $\ell$ we have
$$
v_\ell(j_p(D,M))\le \alpha(\ell).
$$
Hence, there is an absolute  integer constant $\kappa\ge 1$ supported on  primes $\le 163$ such that $j_p(D,M)$ divides $\kappa$ under these assumptions. 

In particular, $j_p(D,M)$ is uniformly bounded under these assumptions.
\end{theorem}
\begin{proof} Let us first consider the case (i). If $\ell \ge 11$ then Lemma \ref{LemmaDiophSS}  gives that there is some prime $r|N$ such that $\ell\nmid c_r(E)$. On the other hand, we see from Lemma \ref{LemmaDG2} (with $S=\emptyset$ and $L=\ell^3$) that if $\ell \le 7$ then  there is some prime $r|N$ such that $\ell^3\nmid c_r(E)$, except, perhaps, for finitely many elliptic curves which can be discarded without affecting the result. In either case, we choose such an $r|N$. If $r|D$ we invoke Lemma \ref{LemmaSwitchOnD}, and if $r|M$ we appeal to Lemma \ref{LemmaSwitchOnM}, so that in either case we conclude that 
$$
v_\ell(j_p(D,M))\le v_\ell(c_r(E)) + \alpha_1(\ell) \le \alpha_2(\ell)
$$
 for certain functions $\alpha_1,\alpha_2:\Pcal\to \Z_{\ge 0}$  supported on primes $\le 163$ and independent of the choice of any elliptic curve. This proves the result in case (i).

The proof in case (ii) is similar, but  we apply Lemma \ref{LemmaGenFermat} instead of Lemma \ref{LemmaDiophSS}, and then  we apply Lemma \ref{LemmaDG2} using the finite set of primes $S=\{2\}$ instead of $S=\emptyset$.
\end{proof}
One also has the following weaker estimate which nonetheless works in complete generality.
\begin{lemma} \label{LemmaGenCoker} Let $E$ be an elliptic curve over $\Q$ of conductor $N$ and consider an admissible factorization $N=DM$. Let $p$ be a prime with $p|D$. Then $j_p(D,M)$ divides $p-1$ if $p$ is odd, and it divides $2$ if $p=2$. In particular, $j_p(D,M)\le p$. 
\end{lemma}
\begin{proof} We base-change the map $\phi_{D,M}:J_0^D(M)\to A_{D,M}$ to $\Q_p$ and consider the special fibre at $p$ of the corresponding N\'eron models. After an unramified quadratic twist (which does not affect the computation of $j_p(D,M)$ as this quantity concerns the geometric components of special fibres) we reduce to the split toric reduction case, and then the result follows from Corollary 3.5 in \cite{PaRa}.
\end{proof}


\subsection{Bounding the correction factor}

\begin{proof}[Proof of Theorem \ref{ThmRT}] Consider any factorization of $N$ of the form $N=dprm$ where $d$ is squarefree with an even number of prime factors, $p$ and $r$ are distinct primes not dividing $d$, and $m$ is coprime to $dpr$. By Proposition \ref{PropRTprelim} and Lemma \ref{LemmaHtIsog} we have
\begin{equation}\label{EqSequentially}
\frac{\delta_{d,prm}}{\delta_{dpr,m}} = \frac{u_{d,p,r,m}}{i_p(d,prm)^2 j_r(dpr,m)^2} \cdot c_p(E)c_r(E)
\end{equation}
where $u_{d,p,r,m}$ is certain rational number supported on primes $\le 163$ with multiplicative height at most $163$. Repeated applications of this observation (to sequentially remove prime factors from $D$) give the result regarding the numerator of $\gamma_{D,M,E}$. More precisely, choosing a prime factorization $D=p_1r_1\cdots p_nr_n$ we apply the previous analysis to the expression
$$
\frac{\delta_{1,N}}{\delta_{D,M}}=\frac{\delta_{1,N}}{\delta_{p_1r_1,N/(p_1r_1)}}\cdot \frac{\delta_{p_1r_1,N/(p_1r_1)}}{\delta_{p_1r_1p_2r_2,N/(p_1r_1p_2r_2)}}\cdots \frac{\delta_{D/(p_nr_n), p_nr_nM}}{\delta_{D,M}}.
$$
The estimate \eqref{EqUpperRT} follows since an upper bound for the numerator of $\gamma_{D,M,E}$ is also an upper bound for this rational number.

The result of item (a) regarding the denominator of $\gamma_{D,M,E}$ follows in a similar fashion. Here one also uses Lemma \ref{LemmaBdImage} on the factors $i_p(d,prm)^2$ occurring in the various applications of \eqref{EqSequentially}, and Lemma \ref{LemmaGenCoker} on the factors  $j_r(dpr,m)^2$. For this we choose a prime factorization $D=p_1r_1\cdots p_nr_n$  in some order satisfying $r_i< p_i$ for each $i$. Then we sequentially apply \eqref{EqSequentially} and the previous estimates to the pairs of factors $(p_i,r_i)$. By Lemma \ref{LemmaGenCoker}, the total contribution of the cokernel factors to the denominator of $\gamma_{D,M,E}$ is bounded by $r_1^2\cdots r_n^2 < p_1r_1\cdots p_nr_n=D$ since we chose $r_i< p_i$.

Finally, the result of item (b) is also obtained by a similar argument. Here,  the cokernel factors $j_r(dpr,m)^2$ are controlled by Theorem \ref{ThmUnifCoker} instead of Lemma \ref{LemmaGenCoker}, giving the claimed stronger estimates in this case.
\end{proof}



\section{Bounding the modular degree}\label{SecBounds}

Our goal in this section is to give bounds for the modular degrees $\delta_{D,M}$ associated to an elliptic curve $E$ of conductor $N$, with $N=DM$ an admissible factorization (both unconditionally and under the Generalized Riemann Hypothesis) and to use them in the context of Szpiro's conjecture. This last point needs some attention; the classical modular approach to Szpiro's conjecture translates bound for $\delta_{1,N}$ into bounds for the Faltings height of an elliptic curve, but the analogous transition is not available in the literature  for the more general case of $\delta_{D,M}$ due to the lack of a Fourier expansion for the modular forms in $S_2^D(M)$ when $D>1$. 

\subsection{Counting systems of Hecke eigenvalues}

Let $s(n)=\dim S_2(n)^{new}$ and let $r_{D,M}$ be the number of systems of Hecke eigenvalues on $\T_{D,M}$. Here, $N$ is a positive integer and $N=DM$ is an admissible factorization. By multiplicity one in $S_2(n)^{new}$ and by the Jacquet-Langlands correspondence we have 
$$
r_{D,M}=\sum_{m|M} s(Dm).
$$  
On the other hand, from Theorem 1 in \cite{GregMartin} (see also Appendix B in \cite{HalKra}) together with Lemma 17 in \cite{GregMartin}, we have the following bound for $s(n)$:
$$
s(n)\le \frac{\varphi(n)}{12} + \frac{7}{12}\cdot 2^{\omega(n)} + \mu(n).
$$
The functions $\varphi, 2^\omega,\mu$ are mutiplicative, $(D,M)=1$ and $D$ is squarefree, so we find
$$
\begin{aligned}
r_{D,M}&\le \frac{\varphi(D)}{12}\sum_{m|M}\varphi(m) + \frac{7\cdot 2^{\omega(D)}}{12}\sum_{m|M}2^{\omega(m)} + \mu(D)\sum_{m|M}\mu(m)\\
&\le \frac{1}{12}\cdot \varphi(D)M + \frac{7}{12}\cdot d(D M^2) + 1.
\end{aligned}
$$

Writing $N=DM$, we deduce:

\begin{proposition}\label{PropDimension} We have 
$$
r_{D,M}\le \frac{1}{12}\cdot \varphi(D)M + \frac{7}{12}\cdot d(D M^2) + 1.
$$
Thus, given $\epsilon>0$, for $N\gg_\epsilon 1$ with an effective implicit constant, we have
$$
r_{D,M}< \left(\frac{1}{12} +\epsilon \right)\cdot \varphi(D)M.
$$
\end{proposition}

(The asymptotic bound follows by recalling that $\varphi(n)\gg n/\log\log n$.) We remark that similar methods give $r_{D,M}\gg \varphi(D)M$.


\subsection{Unconditional bound}  

\begin{theorem}\label{ThmBoundUncond} Given any elliptic curve $E$ over $\Q$ with conductor $N$ and an admissible factorization $N=DM$, we have
$$
\log \delta_{D,M}(E) \le \left(\frac{1}{12}\cdot \varphi(D)M + \frac{7}{12}\cdot d(DM^2)\right)\left(\log N + \frac{4\log N}{\log \log N}\right).
$$
Furthermore, given any $\epsilon>0$, for $N\gg_\epsilon 1$ with an effective implicit constant, we have
$$
\log \delta_{D,M} < \left(\frac{1}{24}+\epsilon \right)\varphi(D) M \log N.
$$

\end{theorem}
\begin{proof} First we bound $\eta_{[\chi_{D,M}]}(c)$ for $c\ne [\chi_{D,M}]$ a class of systems of Hecke eigenvalues on $\T_{D,M}$. Let $f\in S_2(N)$ be the normalized newform corresponding to $\chi_{D,M}$ by Jacquet-Langlands. Take any  $\chi\in c$, let $d|M$ be its level and let $g\in S_2(Dd)^{new}\subseteq S_2(N)$ be the only normalized newform  which corresponds to $\chi$ by Jacquet-Langlands.

We have $f\ne g$. Furthermore, the modular forms $F,G$ obtained by deleting the Fourier coefficients of $f,g$ (respectively) of index not coprime to $N$, have level dividing $N^2$ (cf. the proof of Theorem 1 in \cite{AtkinLehner}). Hence,  the Fourier expansions of $F$ and $G$ differ at some index bounded by $2(\dim S_2(N^2) -1)$. It follows that there is some integer $n_c$ coprime to $N=DM$ satisfying $a_{n_c}(f)\ne a_{n_c}(g)$ and 
\begin{equation}\label{EqKitten}
n_c\le \frac{N^2}{6}\prod_{p|N}\left(1+\frac{1}{p}\right)\le\frac{1}{6}\cdot N^2(1+\log N)< N^3.
\end{equation}
Let $P(x)\in\Z[x]$ be the (monic) minimal polynomial of $a_{n_c}(g)=\chi(T_{D,M,{n_c}})$. Then $\deg(P)\le \# c$ and 
$$
0\ne |P(a_{n_c}(f))|\le (2d(n_c)\cdot n_c^{1/2})^{\#c}
$$
where we have used the Hasse-Weil bound on the Fourier coefficients of a normalized eigenform of weight $2$. From \cite{RobinDivisor} we have the following explicit bound for the divisor function $d(n)$, valid for $n>2$:
\begin{equation}\label{EqDivisorBound}
d(n)< \exp\left(1.5379\cdot (\log 2)\frac{\log n}{\log \log n}\right)< 3^{\log n / \log \log n}.
\end{equation}
By Proposition \ref{PropBoundCongr}, the previous divisor bound, and the inequalities in \eqref{EqKitten} we obtain 
$$
\begin{aligned}
\eta_{[\chi_{D,M}]}(c)&\le (2d(n_c)\cdot n_c^{1/2})^{\#c}\\
&< 3^{\#c\cdot (\log N^3)/(\log\log N^3)}\cdot  N^{\# c}\cdot (1+\log N)^{\#c/2}\\
&< 3.7^{\#c\cdot (\log N^3)/(\log\log N^3)}\cdot  N^{\# c}\\
&<3.7^{\#c\cdot (\log N^3)/(\log\log N)}\cdot  N^{\# c}.
\end{aligned}
$$
So we get
$$
\log \eta_{[\chi_{D,M}]}(c)< \#c\cdot \left(\log N + \frac{4\log N}{\log \log N}\right).
$$
Varying $c\ne [\chi_{D,M}]$ over the classes of systems of Hecke eigenvalues on $\T_{D,M}$, Theorem \ref{ThmSpectral} gives
$$
\begin{aligned}
\log \delta_{D,M}&\le \sum_{c\ne [\chi_{D,M}]} \log \eta_{[\chi_{D,M}]}(c)\\
&<  (r_{D,M}-1)\cdot \left(\log N + \frac{4\log N}{\log \log N}\right).
\end{aligned}
$$
Here we used the fact that $r_{D,M}=\sum_c \#c$, summing over all classes $c$ of systems of Hecke eigenvalues on $\T_{D,M}$. By Proposition \ref{PropDimension}, we obtain the claimed explicit bound.

The proof of the asymptotic bound is similar, but using instead the (effective) estimate
$$
n_c\ll_\epsilon N^{1+\epsilon}
$$ 
from Lemma 11 in \cite{MurtyBounds}. 
\end{proof}

\subsection{Under GRH}

The following result follows from Proposition 5.22 in \cite{IwaniecKowalski} specialized to classical modular forms ---the necessary properties for Rankin-Selberg $L$-functions in this setting have been established in \cite{WLi79}. See also \cite{GolHof}.
\begin{theorem} \label{ThmIK} There is an effective constant $C$ such that the following holds:

Let $f,g$ be normalized Hecke newforms of weight $2$ and level dividing $N$. Assume that the Generalized Riemann Hypothesis holds for the Rankin-Selberg $L$-functions $L(s,f\otimes f)$ and $L(s,f\otimes g)$. Then there is a prime number $p_{f,g}<C\cdot \left(\log N\right)^2$ not dividing $N$, satisfying $a_{p_{f,g}}(f)\ne a_{p_{f,g}}(g)$.
\end{theorem}

Using this, we get
\begin{theorem} \label{ThmBoundCond} Suppose that the Generalized Riemann Hypothesis for Rankin-Selberg $L$-functions of modular forms holds. Let $\epsilon>0$. Then for $N\gg_\epsilon 1$ with an effective implicit constant, we have
$$
\log \delta_{D,M} \le \left(\frac{1}{12}+\epsilon \right)\varphi(D) M \log\log N.
$$
\end{theorem}
\begin{proof} The proof is similar to that of Theorem \ref{ThmBoundUncond}, except that we replace the integer $n_c$ by the prime $p_{f,g}$, with $c$, $f$ and $g$ as in the cited proof, and $p_{f,g}$ as in Theorem \ref{ThmIK}. 
\end{proof}


\subsection{Application to Szpiro's conjecture: Classical modular parameterizations} 

The classical modular approach to Szpiro's conjecture (cf. Section \ref{SecClassical}, especially the estimates \eqref{EqDiscH} and \eqref{EqHDeg}) together with our bounds for $\delta_{D,M}$ specialized to $D=1,M=N$, give:

\begin{theorem}\label{ThmBoundAllQ} For all elliptic curves $E$ over $\Q$ of conductor $N$ we have
$$
h(E)  \le \frac{1}{24}\left(N + 7d(N^2)\right)\left(\log N + \frac{4\log N}{\log \log N}\right) +9
$$
and
$$
\log |\Delta_E|  \le \frac{1}{2}\left(N + 7d(N^2)\right)\left(\log N + \frac{4\log N}{\log \log N}\right) + 124.
$$
Furthermore, given $\epsilon>0$, for $N\gg_\epsilon 1$ with an effective implicit constant we have
$$
h(E) <\left(\frac{1}{48} + \epsilon \right)N\log N \quad \mbox{and} \quad \log |\Delta_E| <\left(\frac{1}{4} + \epsilon \right)N\log N. 
$$
Finally, if we assume GRH, for $N\gg_\epsilon 1$ with an effective implicit constant we have 
$$
h(E) <\left(\frac{1}{24} + \epsilon \right)N\log \log N\quad \mbox{and} \quad\log |\Delta_E| <\left(\frac{1}{2} + \epsilon \right)N\log \log N. 
$$
\end{theorem}
 

\subsection{Application to Szpiro's conjecture: Shimura curve parameterizations}

Here is an extension of the modular approach to Szpiro's conjecture, using Shimura curve parameterizations coming from $X_0^D(M)$ instead of the classical modular parameterization from $X_0(N)$.
\begin{theorem}[The Shimura curve approach to $abc$]\label{ThmShimuraApproachQ} Let $\epsilon>0$. For all elliptic curves $E$ of conductor $N\gg_\epsilon 1$ (with an effective implicit constant), and for any admissible factorization $N=DM$  we have
$$
\log |\Delta_E| <(6+\epsilon)\log \delta_{D,M}(E)  \quad \mbox{and}\quad   h(E) < \left(\frac{1}{2}+\epsilon\right)\log \delta_{D,M}(E).
$$
\end{theorem}
\begin{proof} The case $D=1$ is known, so we can assume $D>1$. The classical modular approach gives upper bounds for $h(E)$ and $\log|\Delta_E|$ in terms of $\log \delta_{1,N}$ (cf. \eqref{EqDiscH} and \eqref{EqHDeg}). The result now follows from the first (effective) inequality in Theorem \ref{ThmRT}, together with the estimates
$$
\begin{aligned}
\log \prod_{p|D}v_p(\Delta_E)&\le \log d(\Delta_E)  & \\
&<\frac{1.07\log |\Delta_E|}{\log \log |\Delta_E|}  &\mbox{(by the divisor bound \eqref{EqDivisorBound})}\\
&\le \frac{1.07\log |\Delta_E|}{\log \log N}  &\mbox{(for the first bound)}\\
&\le  \frac{1.07\left(12 h(E) +16 \right)}{\log \log N}  &\mbox{(by \eqref{EqDiscH}; for the second bound)} .
\end{aligned}
$$
\end{proof}
Of course, instead of the current asymptotic formulation, Theorem \ref{ThmShimuraApproachQ} can be given an exact formulation more amenable for computations with a completely explicit error term  if desired.

Theorem \ref{ThmShimuraApproachQ} together with our bounds for the modular degree (cf. Theorems \ref{ThmBoundUncond} and \ref{ThmBoundCond}) give: 
\begin{theorem}\label{ThmImprovedBounds} For $\epsilon>0$ and $N\gg_\epsilon 1$ (with an effective implicit constant), for each admissible factorization $N=DM$  we have the following bounds valid for all elliptic curves $E$ over $\Q$ with conductor $N$:
$$
h(E)<\begin{cases}
\left(\epsilon + 1/24\right)\varphi(D)M\log N & \mbox{with an accessible error term}\\
\left(\epsilon + 1/48\right)\varphi(D)M\log N & \mbox{unconditional} \\
\left(\epsilon + 1/24\right)\varphi(D)M\log \log N & \mbox{under GRH}
\end{cases}
$$
and similarly
$$
\log|\Delta_E|<\begin{cases}
\left(\epsilon + 1/2\right)\varphi(D)M\log N & \mbox{with an accessible  error term}\\
\left(\epsilon + 1/4\right)\varphi(D)M\log N & \mbox{unconditional} \\
\left(\epsilon + 1/2\right)\varphi(D)M\log \log N & \mbox{under GRH}.
\end{cases}
$$
\end{theorem}
The comment ``with an accessible error term'' refers to the fact that the bound can be obtained with completely explicit lower order terms if desired, thanks to the first estimate in Theorem \ref{ThmBoundUncond}.

Let us remark that in the almost semi-stable case one can replace $\varphi(D)M$ by $\varphi(N)$ by suitable choice of admissible factorization  $N=DM$. Let us record the result here.
\begin{corollary} \label{CoroBoundSSQ} Let $S$ be a finite set of primes and let $P$ be the product of the elements of $S$. For $\epsilon>0$ and $N\gg_{\epsilon,S} 1$ (with an effective implicit constant), if $E$ is an elliptic curve over $\Q$ semi-stable away from $S$, then we have
$$
h(E)<\begin{cases}
\frac{P}{\varphi(P)}\left(\epsilon + 1/48\right)\varphi(N)\log N & \mbox{unconditional} \\

\frac{P}{\varphi(P)}\left(\epsilon + 1/24\right)\varphi(N)\log \log N & \mbox{under GRH}
\end{cases}
$$
and similarly
$$
\log|\Delta_E|<\begin{cases}
\frac{P}{\varphi(P)}\left(\epsilon + 1/4\right)\varphi(N)\log N & \mbox{unconditional} \\
\frac{P}{\varphi(P)}\left(\epsilon + 1/2\right)\varphi(N)\log \log N & \mbox{under GRH}.
\end{cases}
$$
Note that if $S=\emptyset$ (i.e., for semi-stable elliptic curves) the factor $P/\varphi(P)$ is $1$, and that for $P\gg_\epsilon 1$ one has $P/\varphi(P)<(e^\gamma +\epsilon)\log\log P<2\log\log P$.
\end{corollary}
\begin{proof} Let $p_N$ be the largest prime factor of $N$ away from $S$. Then for all but finitely many $E$ we have that $p_N$ exists. Furthermore,  $p_N\to \infty$ as $N\to \infty$ by Shafarevich's theorem. 

Let $N$ be sufficiently large, so that $p_N$ exists. If $N$ has an  even number of prime factors away from $S$, take $D$ to be the product of them. Otherwise, take $D$ as the product of them except $p_N$.  Let $M=N/D$, then
$$
\varphi(D)M= \varphi(N)\prod_{p|M}\left(1-\frac{1}{p}\right)^{-1}\le \varphi(N)\prod_{p|p_NP}\left(1-\frac{1}{p}\right)^{-1}=\varphi(N)\cdot\frac{p_N}{p_N-1}\cdot \frac{P}{\varphi(P)}.
$$
The result now follows from Theorem \ref{ThmImprovedBounds}. Note that this argument is effective.
\end{proof}

\section{Norm comparisons} \label{SecNorms}

\subsection{The result}\label{SecNormsQuat1} This section is purely analytic and there is no additional difficulty in momentarily considering a more general case.

Let $F$ be a totally real number field of degree $n$ with real embeddings $\tau_j$ ($1\le j\le n$). Let $B$ be a quaternion \emph{division} $F$-algebra with exactly one split place at infinity, say $\tau_1$. Let $\B=B\otimes \hat{\Z}$ and let $B^\times_+$ be the elements of $B^\times$ with totally positive reduced norm. For each compact open subgroup $U\subseteq \B^\times$ and for each $g\in \B^\times$ consider the group $\Gamma_{U,g}=g U g^{-1}\cap B^\times_+$ and its image $\tilde{\Gamma}_{U,g}$ in $PSL_2(\R)$ via $\tau_1$. The quotient $X_{U,g}^{an}=\tilde{\Gamma}_{U,g}\backslash \hfrak$ is a compact, connected,  complex curve, because $\B$ is a division algebra.

Let $S_{U,g}$ be the space of weight $2$ holomorphic modular forms for the action of $\tilde{\Gamma}_{U,g}$ on $\hfrak$. Since we are assuming that $B$ is a division algebra, the cuspidality condition would be vacuous. 

On $S_{U,g}$ we have the $L^2$-norm induced by the Petersson inner product
$$
\|h\|_{U,g,2}:=\left(\int_{\tilde{\Gamma}_{U,g}\backslash \hfrak} |h(z)|^2\Im(z)^2 d\mu_\hfrak(z) \right)^{1/2}.
$$
We also have the supremum norm on $S_{U,g}$
$$
\|h\|_{U,g,\infty}:=\sup_{z\in\hfrak} | h(z)| \Im(z).
$$
(Observe that the $L^2$-norm is not normalized, so it is not invariant by shrinking $U$.) Our goal in this section is to compare these two norms.
\begin{theorem}\label{ThmNorms} Keep the previous notation. There is a number $\nu_n>0$ depending only on the degree $n=[F:\Q]$, such that if   $\tilde{\Gamma}_{U,g}$ acts freely on $\hfrak$, then for all $h\in S_{U,g}$ we have
$$
\|h\|_{U,g,\infty}\le \nu_n \cdot \|h\|_{U,g,2}.
$$
\end{theorem}
Note that $\tilde{\Gamma}_{U,g}$ acts freely on $\hfrak$ if and only if the quotient map $\hfrak\to X_{U,g}^{an}$ is unramified. This can always be achieved by suitably shrinking $U$; however, for our purposes later, we will need to be precise about this point.

These two norms have been compared in the case of classical modular curves (non-compact case, see for instance \cite{BloHol}) or in the compact case assuming that the weight is large (cf. \cite{DasSen}). As in the non-compact case, one might expect that for every $\epsilon>0$, the quantity $\nu_n$ should be replaced by a factor $\ll_{n,\epsilon} Vol(X_{U,g})^{-1/2+\epsilon}$, where the volume is taken with respect to $d\mu_\hfrak$. Improvements in this direction are available in the non-compact case (see for instance \cite{BloHol}), but the techniques do not work in the absence of Fourier expansions. In any case, Theorem \ref{ThmNorms} suffices for our purposes.

\subsection{Injectivity radius} Let $U$ be an open compact subgroup of $\B^\times$, let $g\in\B^\times$, and write $\Gamma:=\Gamma_{U,g}=gUg^{-1}\cap B^\times_+$. 

For $\gamma\in B^\times_+$ we write $\gamma^*$ and $\tilde{\gamma}$ for its image in $SL_2(\R)$ and in $PSL_2(\R)$ respectively (via $\tau_1$).

The systole of $\Gamma$ is defined as
$$
\sigma_\Gamma:= \min\{d_\hfrak(x,\tilde{\gamma}\cdot  x) : x\in \hfrak, \gamma\in\Gamma, x\ne \tilde{\gamma}\cdot  x\}
$$
where $d_\hfrak$ is the hyperbolic distance in $\hfrak$. Note that $\sigma_\Gamma$ is twice the injectivity radius $\rho_\Gamma$ of $X_{U,g}^{an}$. 

Recall that if $u\in SL_2(\R)$ is hyperbolic (i.e. $|tr(u)|>2$), then for every $x\in \hfrak$ we have
$$
d_{\hfrak}(x,\gamma x)=2\log |\lambda_u|
$$
where $\lambda_u$ is the eigenvalue of $u$ with largest absolute value. We say that $\gamma\in B^\times_+$ is hyperbolic if $\gamma^*$ is, and we write $\lambda_\gamma=\lambda_{\gamma^*}$.
\begin{lemma} Let $\gamma\in\Gamma$ be hyperbolic. Define $\beta:=\rn(\gamma)^{-1}\gamma^2$, where, as before, $\rn$ denotes the reduced norm. Then we have the following:
\begin{itemize}
\item[(i)] $\beta\in B^\times_+$,  $\rn(\beta)=1$ and it is in some maximal order of $B$.
\item[(ii)] $\beta$ is hyperbolic and $d_\hfrak(x,\tilde{\beta} x)=2d_\hfrak(x,\tilde{\gamma} x)$ for every $x\in \hfrak$
\item[(iii)] Let $K=F(\lambda_\beta)$. Then $[K:F]=2$.
\item[(iv)] $\lambda_\beta\in O_K^\times$.
\item[(v)] $1/\lambda_\beta$ is a conjugate of $\lambda_\gamma$ over $F$.
\item[(vi)] all other conjugates of $\lambda_\beta$ over $\Q$ (if any) have modulus $1$. 
\end{itemize}
\end{lemma}
\begin{proof} Since $\gamma\in\Gamma\subseteq B^\times_+$, we have the first two assertions in (i). Also, since $\gamma\in gUg^{-1}\cap B^\times$ we see that $\rn(\gamma)\in O_F^\times$, so $\rn(\gamma)^{-1}\gamma^2$ is in some maximal order of $B$, proving (i). Item (ii) follows because $\tilde{\beta}=\tilde{\gamma}^2$. 

In view of  (i), the properties (iii)-(vi) for $\beta$ are known, see for instance Section 12.3 of \cite{MacRei}.
\end{proof}
In particular, $\lambda_\beta$ as in the previous lemma is a \emph{Salem number} of degree $2n$, where $n=[F:\Q]$. Using the available partial progress on Lehmer's conjecture for the Mahler measure, one gets
\begin{lemma} \label{LemmaInjRadius} If $\tilde{\Gamma}_{U,g}$ acts freely on $\hfrak$, then
$$
\rho_\Gamma \ge \frac{1}{2(\log(6n))^3}.
$$
\end{lemma}
\begin{proof} Since $B$ is a division algebra, $\tilde{\Gamma}_{U,g}$ contains no non-trivial parabolic elements, and since it acts freely on $\hfrak$, it  contains no elliptic elements either. Thus,  every non-hyperbolic $\gamma\in\Gamma$  satisfies $\tilde{\gamma}=I$. It follows that there is $\gamma\in\Gamma$ hyperbolic with
$$
2\rho_\Gamma = d_\hfrak(i,\tilde{\gamma}\cdot i) = \frac{1}{2}d_\hfrak(i,\tilde{\beta}\cdot i)=\log |\lambda_\beta|\qquad (i=\sqrt{-1}\in\hfrak)
$$ 
with $\beta$ as in the previous lemma. The bound now follows from Corollary 2 in \cite{Voutier}. (We remark that even weaker bounds would suffice for our purposes, but we choose to use Voutier's result for the sake of concreteness.)
\end{proof}
\subsection{Proof of the norm comparison} Recall that on $\hfrak$ (with complex variable $z=x+iy$) we consider the volume form
$$
d\mu_\hfrak(z) = \frac{dx\wedge dy}{y^2}.
$$
On the open unit disc $\D$ (with complex variable $w=u+vi$), let us consider
$$
d\mu_\D(w) = \frac{4du\wedge dv}{(1-(u^2+v^2))^2}.
$$
These volume forms come from the usual hyperbolic metrics $d_\hfrak$ and $d_\D$ on $\hfrak$ and $\D$ respectively, with constant curvature $-1$. For $z\in\hfrak$, $w\in \D$ and $r>0$, we let $B_\hfrak(z,r)$ and $B_\D(w,r)$ be the corresponding balls of hyperbolic radius $r$ centered at $z$ and $w$ respectively.

For any $\tau\in\hfrak$, the biholomorphic map
$$
c_\tau:\D\to\hfrak,\quad c_\tau(w)=\Re(\tau) - \Im(\tau)i\cdot\frac{w+i}{w-i}
$$
is an isometry for $d_\D$ and $d_\hfrak$. Furthermore, it satisfies $c_\tau(0)=\tau$ and $c_\tau^*d\mu_\hfrak = d\mu_\D$.

\begin{lemma} Let $t>0$ and let $B(0,t)$ be the Euclidean ball in $\C$ centered at $0$ with radius $t$. Let $h$ be holomorphic on a neighborhood of $B(0,t)$. Then
$$
\pi t^2 |h(0)|^2\le \int_{B(0,t)} |h(z)|^2 dx\wedge dy.
$$
\end{lemma}
\begin{proof} This is immediate from expanding $h$ as a power series and integrating $h\cdot \bar{h}$.
\end{proof}
\begin{lemma}\label{LemmaHyperbolic} Let $f:\hfrak\to\C$ be holomorphic, let $\tau\in\hfrak$ and let $r>0$. Then
$$
|f(\tau)|^2\Im(\tau)^2\le \frac{e^{2r}}{4\pi (\tanh(r/2))^2}\int_{B_{\hfrak}(\tau,r)} |f(z)|^2y^2d\mu_\hfrak(z).
$$
\end{lemma}
\begin{proof} Let $f_\tau=c_\tau^* f$. It is a holomorphic function on $\D$ and we have
$$
\begin{aligned}
\int_{B_{\hfrak}(\tau,r)} |f(z)|^2y^2d\mu_\hfrak(z)&=\int_{B_{\D}(0,r)} |f_\tau(w)|^2\Im(c_\tau(w))^2d\mu_\D(w) \\
&=\int_{B_{\D}(0,r)} |f_\tau(w)|^2\Im(c_\tau(w))^2\frac{4du\wedge dv}{(1-(u^2+v^2))^2}\\
&\ge 4 \int_{B_{\D}(0,r)} |f_\tau(w)|^2\Im(c_\tau(w))^2du\wedge dv.
\end{aligned}
$$
Note that
$$
\inf \{\Im(c_\tau(w)): w\in B_\D(0,r)\}=\inf \{\Im(z): z\in B_\hfrak(\tau,r)\}=e^{-r}\Im(\tau)
$$
and that if $B(0,t)$ denotes the Euclidean ball of radius $t$ centered at $0$, then we have $B(0,t)= B_\D(0,r)$ for $t= \tanh(r/2)$. So, from the previous lemma we get
$$
\begin{aligned}
\int_{B_{\hfrak}(\tau,r)} |f(z)|^2y^2d\mu_\hfrak(z) &\ge  4e^{-2r}\Im(\tau)^2\int_{B(0,\tanh(r/2))} |f_\tau(w)|^2 du\wedge dv\\
&\ge 4\pi e^{-2r}(\tanh(r/2))^2 \Im(\tau)^2  |f(\tau)|^2.
\end{aligned}
$$
\end{proof}

\begin{proof}[Proof of Theorem \ref{ThmNorms}] Choose $\tau_0\in \hfrak$ such that $|h(\tau_0)|\Im(\tau_0)=\|h\|_{U,g,\infty}$ and apply  Lemma \ref{LemmaHyperbolic} with $\tau=\tau_0$ and $2r=1/(\log(6n))^3$. This choice of $r$ together with Lemma \ref{LemmaInjRadius} ensure 
$$
\int_{B_{\hfrak}(\tau,r)} |h(z)|^2y^2d\mu_\hfrak\le \|h\|_{U,g,2}^2.
$$
\end{proof}

\section{Differentials on integral models} \label{SecFinitePart}

\subsection{Relative differentials on Shimura curves} \label{SecFinPart1} We now return to  the case $F=\Q$.

Given  a compact open subgroup $U\subseteq \B^\times$ contained in $O_\B^\times$, we write $m_U$ for its level and $X_0^D(M,U)$ for the Shimura curve over $\Q$ associated to the compact open subgroup $U\cap U_0^D(M)$. Here, $N=DM$ is an admissible factorization, $D$ is the discriminant of $B$, and we always assume $m_U$ coprime to $D$.

Let $\Xcal_0^D(M,U)$ be the standard integral model for $X_0^D(M,U)$ over $\Z[m_U^{-1}]$, constructed as coarse moduli scheme for the moduli problem of abelian surfaces with quaternionic multiplication (i.e. fake elliptic curves) and $U_0^D(M)\cap U$-structure. (This direct  moduli approach to integral models is available as we are working over $\Q$.) 

Let $p\nmid m_U$ be a prime. For $p|D$ the reduction of $\Xcal_0^D(M,U)$ is of  Cerednik-Drinfeld type, for $p|M$  it is of Deligne-Rapoport type, and for $p\nmid Nm_U$ the integral model has good reduction.

Let $\Xcal_0^D(M, U)^0$ be the smooth locus of $\Xcal_0^D(M,U)\to\Spec\Z[m_U^{-1}]$. It is obtained from $\Xcal_0^D(M, U)$ by removing the supersingular points in the special fibres of characteristic dividing $D$, and removing supersingular points and non-reduced components of the special fibres of characteristic dividing $M$ (non-reduced components in characteristic $p$ only occur when $p^2| M$).

There is a natural forgetful $\Z[m_U^{-1}]$-map 
$$
\pi^D_{M,U}:\Xcal_0^D(M,U)\to \Xcal_0^D(1,U). 
$$

Given $N=DM$, we say that $U$ is \emph{good enough} for $(D,M)$ if the following conditions are satisfied:
\begin{itemize}
\item[(i)] $m_U$ is coprime to $N$;
\item[(ii)] $\rn (U)=\hat{\Z}$;
\item[(iii)] $U\subseteq U_1^D(\ell)$ for some prime  $\ell \ge 5$ with $\ell\nmid D$.
\end{itemize}
See \cite{Buzzard} for details on these conditions. Thus, when $U$ is good enough for $(D,M)$, we have that both $\Xcal_0^D(1,U)/\Z[m_U^{-1}]$ and $\Xcal_0^D(M,U)/\Z[m_U^{-1}]$ are fine moduli spaces for the associated moduli problems (by (i) and (iii)), and they have geometrically connected generic fibre (by (i) and (ii)). In this section we prove:
\begin{theorem} \label{ThmRelDiff}  Suppose that $U$ is good enough  for $(D,M)$. The integer $M$ anihilates the sheaf of relative differentials $\Omega^1_{\pi^D_{M,U}}$ on $\Xcal_0^D(M,U)^0$.
\end{theorem}
Assuming this result for the moment, we obtain:
\begin{corollary} \label{CoroRelDual} Suppose that $U$ is good enough for $(D,M)$. On $\Xcal_0^D(M, U)^0$,  the canonical morphism of sheaves 
$$
(\pi^D_{M,U})^*\Omega^1_{\Xcal^D_0(1,U)^\circ/\Z[m_U^{-1}]}\to \Omega^1_{\Xcal^D_0(M,U)^\circ/\Z[m_U^{-1}]}
$$ 
is injective and its cokernel is annihilated by $M$.
\end{corollary}
\begin{proof} The map is injective because it is non-zero, and on the smooth locus the two sheaves are invertible.

The assertion about the cokernel follows from the fundamental exact sequence of relative differentials and the previous theorem.
\end{proof}

\subsection{Fibres at primes dividing $M$}

Suppose that $U\subseteq O_{\B}^\times$ is good enough for $(D,M)$.

Let $p$ be a prime dividing $M$, and let $n,m$ be positive integers defined by $p\nmid m$ and $M=p^nm$. The fibre of $\Xcal_0^D(M,U)$ at $p$ is described as follows (cf. Chapter 13 in \cite{KatzMazur}, suitably adapted to moduli of fake elliptic curves; see \cite{Buzzard, Helm} for this adaptation):

Put $k=\F_p$. The fibre $\Xcal_0^D(M,U)\otimes k$ is formed by $n+1$ copies of $\Xcal_0^D(m,U)\otimes k$ with suitable multiplicities, crossing at supersingular points. More precisely, for any pair of integers $a,b\ge 0$ with $a+b=n$, there is a geometrically integral closed sub-scheme $F_{a,b}$ of $\Xcal_0^D(M,U)\otimes k$, isomorphic to the curve $\Xcal_0^D(m,U)\otimes k$ and occurring with multiplicity $\varphi(p^{\min\{a,b\}})$ in $\Xcal_0^D(M,U)\otimes k$, where $\varphi$ is Euler's function. Then the irreducible components of $\Xcal_0^D(M,U)\otimes k$ are precisely the curves $F_{a,b}$, with multiplicity $\varphi(p^{\min\{a,b\}})$, crossing  at the super-singular points of $\Xcal_0^D(M,U)\otimes k$. Note that $F_{n,0}$ and $F_{0,n}$ occur with multiplicity $1$.

The map $\Xcal_0^D(M,U)\to \Xcal_0^D(m,U)$ induces maps $f_{a,b}:F_{a,b}\to \Xcal_0^D(m,U)\otimes k$. The morphism $f_{n,0}$ is an isomorphism, and $f_{0,n}$ has degree $p^{n}$.

For the sake of exposition, let us recall that in the simplest case $U=O_\B^\times$, $D=1$, $m=1$,  the above mentioned facts correspond to Kronecker's congruence for the modular polynomials  $\Phi_{p^n}$, namely
$$
\Phi_{p^n}(X,Y)\equiv \prod_{\substack{a,b\ge 0 \\ a+b=n\\ c=\min\{a,b\}}} (X^{p^{a-c}}-Y^{p^{b-c}})^{\phi(p^c)}\mod p.
$$

\subsection{Annihilation of relative differentials}

\begin{lemma} \label{LemmaLoc} If $U$ is good enough for $(D,M)$, then the sheaf $\Omega^1_{\pi^D_{M,U}}$ is supported on the special fibres at primes dividing $M$, i.e., 
$$
\Omega^1_{\pi^D_{M,U}}|_{\Xcal_0^D(M, U)[M^{-1}]}=0.
$$ 
\end{lemma}
\begin{proof}
Note that $\Xcal_0^D(M,U)\otimes \Z[M^{-1}]=\Xcal_0^D(1,U')$ with $U'=U_0(M)\cap U$. Since $U$ is good enough for $(D,M)$, we get that $U'$ is good enough for $(D,1)$. Since $U$ and $U'$ are contained in $U_1^D(\ell)$ with $\ell\ge 5$, the morphism $\Xcal_0^D(1, U')\to \Xcal_0^D(1, U)[M^{-1}]$ induced by the inclusion $U'\subseteq U$ is \'etale, hence the result. \end{proof}


\begin{proof}[Proof of Theorem \ref{ThmRelDiff}] First we perform a preliminary reduction.

Write $M=p^nm$ with $p$ a prime, $p\nmid m$ and $n\ge 1$. Consider the forgetful map $\pi:\Xcal_0^D(M,U)\to \Xcal_0^D(m,U)$ and the factorization $\pi^D_{M,U}=\pi^D_{m,U}\pi$. We have the exact sequence
$$
\pi^*\Omega^1_{\pi^D_{m,U}}\to \Omega^1_{\pi^D_{M,U}}\to \Omega^1_{\pi}\to 0.
$$
Note that upon inverting $p$, the map $\pi$ becomes the forgetful $\Z[(pm_U)^{-1}]$-morphism 
$$
\Xcal_0^D(m,U')\to \Xcal_0^D(m,U)\otimes \Z[p^{-1}]
$$ 
with $U'=U\cap U_0^D(p^n)$, which is \'etale because $U$ is good enough. So, $\Omega^1_{\pi}[p^{-1}]=0$. Furthermore, it also follows that away from characteristic $p$ we have $\pi^{-1}\Xcal_0^D(m,U)^0[1/p]=\Xcal_0^D(M,U)^0[1/p]$. Thus, if one knows that $m$ annihilates $\Omega^1_{\pi^D_{m,U}}|_{\Xcal_0^D(m,U)^0}$, then the  previous exact sequence would show that $m$ also annihilates  $\Omega^1_{\pi^D_{M,U}}|_{\Xcal_0^D(M,U)^0}$ away from characteristc $p$. To complete the argument, it would suffice  that $p^n$ annihilates $\Omega^1_{\pi^D_{M,U}}|_{\Xcal_0^D(M,U)^0}$ after inverting $m$. This would follow if we show that $p^n$ annihilates $\Omega^1_{\pi^D_{p^n,U''}}|_{\Xcal_0^D(p^n,U'')^0}$ where $U''=U\cap U_0^D(m)$, because $\Xcal_0^D(1,U'')\to \Xcal_0^D(1,U)\otimes \Z[1/m]$ is \'etale.

From the previous analysis, we see ---by induction on the number of distinct prime factors of $M$--- that it suffices to prove the result in the particular case $M=p^n$ for $p$ a prime and $n\ge 1$. So let's assume that this is the case.

By Lemma \ref{LemmaLoc} we only need to show that $p^n$ annihilates $\Omega^1_{\pi^D_{p^n,U}}|_{\Xcal_0^D(p^n,U)^0}$ \'etale-locally on the fibre at $p$ of $\Xcal_0^D(p^n,U)^0$,
that is, for points on the ordinary locus of the components $F_{n,0}$ and $F_{0,n}$  of the special fibre at $p$. 

Write $k=\F_p$, let $K$ be an algebraic closure of $k$ and let $W$ be the ring of Witt vectors of $K$.  Let $y_0\in \Xcal_0^D(p^n,U)^0$ be a closed point with residue characteristic $p$ and let $x_0\in \Xcal_0^D(1,U)^0$ be its image under $\pi^D_{p^n,U}$. Let $y$ be a $K$-valued point in $\Xcal_0^D(p^n,U)^0\otimes W$ lying above $y_0$ and let $x$ be its image in $\Xcal_0^D(1,U)^0\otimes W$. Let $\Acal=W[[T]]$ so that we have an isomorphisms of complete local rings involving the completed strict henselization of $\Ocal_{\Xcal_0^D(1,U)^0,x_0}$ (cf. p.133-134 in \cite{KatzMazur}):
\begin{equation}\label{EqStalk}
(\Ocal^{s.h.}_{\Xcal_0^D(1,U)^0,x_0})^\wedge\simeq \widehat{\Ocal}_{\Xcal_0^D(1,U)^0\otimes W, x}\simeq \Acal.
\end{equation}
Since $x$ is in the ordinary locus, Serre-Tate theory (cf. section 8.9 in \cite{KatzMazur}, see also \cite{Buzzard} for an adaptation to deformations of fake elliptic curves) gives $\Acal$ the structure of a $\Z[q,q^{-1}]$-algebra by letting $q$ be the Serre-Tate parameter of the pull-back  of the universal family over $\Xcal_0^D(1,U)$ (recall that $U$ is good enough). Denote the image of $q$ in $\Acal$ again by $q$, so that $q\in \Acal^\times$. 

By the isomorphism \eqref{EqStalk}, we have that $(\Ocal^{s.h.}_{\Xcal_0^D(p^n,U)^0,y_0})^\wedge$ is a finite $\Acal$-algebra under the pull-back map. By Theorem 13.6.6 in \cite{KatzMazur}, this $\Acal$-algebra structure can be described as follows:
$$
(\Ocal^{s.h.}_{\Xcal_0^D(p^n,U)^0,y_0})^\wedge\simeq \Bcal:= \begin{cases}
\Acal  & \mbox{ if } y_0\in F_{n,0}\cap \Xcal_0^D(p^n,U)^0\\
\Acal[Z]/(Z^{p^n}-q) &\mbox{ if } y_0\in F_{0,n}\cap \Xcal_0^D(p^n,U)^0.
\end{cases}
$$
Finally, let us check that $p^n$ annihilates the \'etale stalk $\Omega^{1,et}_{\pi^D_{p^n,U}, y}$. It suffices to check this after completion. We have
$$
(\Omega^{1,et}_{\pi^D_{p^n,U}, y})^\wedge \simeq \Omega^{1}_{\pi^D_{p^n,U}, y_0}\otimes (\Ocal^{et}_{\Xcal_0^D(p^n,U)^0,y})^\wedge\simeq \Omega^{1}_{\pi^D_{p^n,U}, y_0}\otimes (\Ocal^{s.h.}_{\Xcal_0^D(p^n,U)^0,y_0})^\wedge\simeq \Omega^1_{\Bcal/\Acal}.
$$
If the ordinary point $y_0$ belongs to $F_{n,0}$, then the previous module is $\Omega^1_{\Acal/\Acal}=(0)$. On the other hand, if $y_0\in F_{0,n}$ then using the fact that $\frac{d}{dZ}(Z^{p^n}-q)=p^nZ^{p^n-1}$ we find
$$
\Omega^1_{\Bcal/\Acal}\simeq \frac{\Acal[Z]}{(Z^{p^n}-q, p^nZ^{p^n-1})}.
$$ 
Since $p^n q = Z\cdot p^nZ^{p^n-1} - p^n\cdot(Z^{p^n}-q)$ and $q\in\Acal^\times$, we see that the previous module is a quotient of $\Acal[Z]/p^n\Acal[Z]$, hence, it is annihilated by $p^n$.
\end{proof}

\section{Bounds for the Manin constant}\label{SecManinCt}

\subsection{The Manin constant}

Given an elliptic curve $A$ over $\Q$ with conductor $N$, which is an optimal quotient $q:J_0(N)\to A$ with associated normalized newform $f\in S_2(N)$, we write $c_f$ for its Manin constant (cf. Section \ref{SecClassical}). Thus, letting $\omega_A$ be a global N\'eron differential for $A$, the pull-back of $\omega_A$ under $\hfrak\to X_0(N)\to A$ is $2\pi i c_f f(z)dz$. Here, the map $X_0(N)\to A$ is $\phi=qj_N$. Multiplying $\omega_A$ by $-1$ if necessary, we assume that $c_f$ is positive.

Edixhoven \cite{Edixhoven} proved that $c_f$ is a non-zero integer. After the work of Mazur \cite{MazurRatIsog} and Abbes, Ullmo, and Raynaud \cite{AbbesUllmo} we know that if $v_p(N)\le 1$ then $v_p(c_f)=0$, except, perhaps, for $p=2$ in which case the assumption $v_2(N)\le 1$ only gives $v_2(c_f)\le 1$. (See \cite{ARSManin} and the references therein for more results on the Manin constant.) This last caveat at $p=2$ has been removed by recent work of Cesnavicius \cite{Cesnavicius}, so that now one knows that for every prime $p$ the following implication holds: 
$$
v_p(N)\le 1\Rightarrow p\nmid c_f.
$$

Since the conductor of the elliptic curve $A$ is $N$, we know that in the relevant cases
\begin{equation}\label{EqValN}
v_p(N)\le \begin{cases}
8 & \mbox{ if }p=2  \\
5 & \mbox{ if }p=3  \\
2 & \mbox{ if }p\ge 5.
\end{cases}
\end{equation}
It is desirable to have control on $v_p(c_f)$ at all primes, not just when $v_p(N)\le 1$. However, not much is known about $v_p(c_f)$ in the general case. See \cite{Edixhoven} for some additional results when $p> 7$ is a prime of additive reduction\footnote{In recent personal communication, Bas Edixhoven outlined a strategy that seems promising for obtaining further progress on these matters. Our methods, however, are completely different.}.

We will prove:
\begin{theorem}\label{ThmManinCt} Let $S$ be a finite set of primes and let $p$ be a prime number. There is a constant $\mu_{S,p}$ depending only on $S$ and $p$, such that for every optimal elliptic curve $A$ over $\Q$ with semi-stable reduction outside $S$ and with associated newform $f\in S_2(N)$, we have 
$$
v_p(c_f)\le \mu_{S,p}.
$$
\end{theorem}
The following is an immediate consequence of the previous theorem and the known results about the Manin constant at primes with $v_p(N)\le 1$.
\begin{corollary}\label{CoroManinCt} Let $S$ be a finite set of primes. There is a constant $\Mcal_S$ depending only on $S$ such that for every optimal elliptic curve $A$ defined over $\Q$ with semi-stable reduction away from $S$ and with associated newform $f\in S_2(N)$, we have $c_f\le \Mcal_S$.
 
\end{corollary}
The proof of these results is motivated by the existing literature, especially \cite{ ARSManin, CesOld, Edixhoven}. Our main new contribution is the idea of working with towers of suitable modular curves with additional level structure to get good integral models, and our method to deal with the case when the additional level structure makes the relevant eigenform an old form.

\subsection{Setup for the proof of Theorem \ref{ThmManinCt}} Let us fix a set of primes $S$, a prime number $p$, a positive integer $n\ge 2$, and an auxiliary prime number $\ell\ge 5$ different from $p$. We will prove:

\begin{theorem}\label{ThmReductionManin} There is a bound $\Bcal=\Bcal(S,p,n,\ell)$ such that for any given optimal elliptic curve $A$ with conductor $N=p^nm$ and with associated newform $f\in S_2(N)$, satisfying that $p\nmid m$ and that $m$ squarefree away from $S$, one has $v_p(c_f)\le \Bcal$.
\end{theorem}

Note that the existence of such an $A$ and our assumption $n\ge 2$ force $p\in S$.

For notational convenience, we will fix an optimal elliptic curve $A$ as in Theorem \ref{ThmReductionManin} for the rest of this section. So we need to state explicitly the parameters on which each bound depends, and we will do so by adding appropriate subscripts to the asymptotic notation $\ll$, $O(-)$, $\asymp$.

Theorem \ref{ThmManinCt} will follow by fixing the choice $\ell=5$ unless $p=5$ in which case we take $\ell =7$, and from the fact that $n\le 8$ by \eqref{EqValN}. The cases of semi-stable reduction at $p$ (that is, $n=0$ or $1$) follow from the existing literature.

Since $\ell$ has to be chosen uniformly bounded, we are forced to consider the cases $\ell \mid N$ and $\ell\nmid N$ (equivalently, $\ell \mid m$ and $\ell\nmid m$) separately, the second being the more laborious.

Let $R=\Z_{(p)}$. Given an $R$-module $M$, we define 
$$
v(M)=\begin{cases}
\infty &\mbox{ if no power of }p\mbox{ annihilates }M_{tor};\\
\min\{k\ge 0 : p^k\cdot M_{tor} =(0)\}&\mbox{ otherwise.}
\end{cases}
$$
 Since $R$ is DVR, whenever $M$ is finitely generated we have that $v(M)$ is finite and there is an element $x\in M$ such that $v(M)=v(\langle x\rangle )$.
 
  When $N$ is a free $R$-module, we say that $x\in N$ is primitive if $x\ne 0$ and $v(N/\langle x\rangle)=0$.
 
More generally, for a $\Z$-module $G$, we write $v(G):=v(G[p^{\infty}])$. We observe that when $G$ is a finitely generated $\Z$-module, $v(G)$ is the $p$-adic valuation of the exponent of the finite group $G_{tor}$.

\subsection{A projective system of curves}

In this section, $m$ will always denote a positive integer coprime to $p$ (possibly divisible by $\ell$), and we will consider the curves  $X_m:=X_{U_0(p^nm)\cap U_1(\ell)}$. The notation $X_m$ (and related notation to be introduced below) will be used with this meaning only in the present Section \ref{SecManinCt}.

The curves $X_m$ are defined over $\Q$ and are geometrically irreducible. The cusp $i\infty$ defines a $\Q$-rational point in $X_m$ (possibly after conjugation of the open compact group $U_1(\ell)$, depending on conventions; cf. Variant 8.2.2 in \cite{DiamondIm}). The integral model over $\Z[1/\ell]$ provided by the theory of Deligne-Rapoport \cite{DeRa}, Katz-Mazur \cite{KatzMazur} and Cesnavicius \cite{CesInt} can be base-changed to $R$, obtaining an integral model that we denote by $\Xcal_m/R$. Then $\Xcal_m$ is regular (since $\ell\ge 5$) and $\Xcal_m\to \Spec(R)$ is flat and proper. Hence, $\Xcal_m\to \Spec(R)$ is Gorenstein, and Grothendieck's duality theory (cf. \cite{DeRa}) applies. The relative dualizing sheaf is denoted by $\omega_m$ and it is invertible. Furthermore, for $m|m'$ both coprime to $p$, the forgetful map $\Xcal_{m'}\to \Xcal_m$ is etale and the pull back of $\omega_m$ is $\omega_{m'}$.

Let $J_m$ be the Jacobian of $X_m$ over $\Q$, and let $\Jcal_m$ be the N\'eron model over $R$.

Since $\Xcal_m/R$ has sections (e.g. the one induced by the cusp $i\infty$) and has some fibre components with multiplicity $1$ (from the standard description of the special fibre at $p$) we see from Theorem 1, Sec. 9.7 \cite{Boschetal} that $\Pic^0_{\Xcal_m/R}$ is a scheme, the canonical map $\Pic^0_{\Xcal_m/R}\to\Jcal_m^0$ to the identity component of $\Jcal_m$ is an isomorphism, and there are canonical identifications $H^1(\Xcal_m, \Ocal_{\Xcal_m})=\Lie(\Pic^0_{\Xcal_m/R})\simeq\Lie(\Jcal_m)$.

Dualizing we obtain the $R$-isomorphisms
\begin{equation}\label{EqDuality}
H^0(\Jcal_m,\Omega^1)\simeq \Lie(\Jcal_m)^\vee\simeq H^1(\Xcal_m, \Ocal_{\Xcal_m})^\vee \simeq H^0(\Xcal_m,\omega_m).
\end{equation}

The cusp $i\infty$ defines a $\Q$-rational point of $X_0(p^nm,\ell)$, hence, an $R$-section $[i\infty]$ on $\Xcal_m$. Let $\Xcal_m^\infty$ be the open set of $\Xcal_m$ obtained by deleting from $\Xcal_m$ the fibre components that do not meet $[i\infty]$. In this way, $[i\infty]$ induces an $R$-morphism $j_m:\Xcal_m^\infty\to \Pic^0(\Xcal_m/R)=\Jcal_m^0\subseteq \Jcal_m$. On $\Xcal_m^\infty$ we have a canonical isomorphism between $\Omega^1_{\Xcal_m/R}$ and $\omega_m$, hence we obtain by pull-back
$$
j_m^\bullet:H^0(\Jcal_m,\Omega^1)\to H^0(\Xcal_m^\infty,\omega_m).
$$
One can check (say, by base change to $\C$) that this map factors through  \eqref{EqDuality}.  In particular
$$
v(\coker(j_m^\bullet))=v\left(\frac{H^0(\Xcal_m^\infty,\omega_m)}{H^0(\Xcal_m,\omega_m)}\right).
$$
On the \'etale projective system $\{\Xcal_m\}_{p\nmid m}$ the invertible sheaves $\omega_m$ are compatible by pull-back as explained above. One can check that the theory of Conrad (cf. \cite{ConradAppendix}, specially Theorem B.3.2.1) for comparing integral structures applies in this slightly modified setting, which gives:
\begin{theorem}
As $m$ varies over integers coprime to $p$, we have
\begin{equation}\label{EqConrad}
v(\coker(j_m^\bullet))=v\left(\frac{H^0(\Xcal_m^\infty,\omega_m)}{H^0(\Xcal_m,\omega_m)}\right)\ll_{p,n,\ell} 1.
\end{equation}
\end{theorem}
In our application, note that we will be taking $n\le 8$ and $\ell=5$ or $7$, so the implicit constant will be bounded just in terms of $p$.


\subsection{Some reductions} Let us write $X_{0,m}=X_0(p^nm)$, $J_{0,m}=J_0(p^nm)$, and consider the standard integral model $\Xcal_{0,m}=\Xcal_0(p^nm)\otimes R$ as well as the N\'eron model over $R$ of $J_{0,m}$, which we denote by $\Jcal_{0,m}$.

Let $A$ be an elliptic curve of conductor $p^nm$ and assume that we have an optimal quotient $q:J_{0,m}\to A$. Let $\Acal$ be the Neron model of $A$ over $R$, and let $\omega\in H^0(\Acal,\Omega^1)$ be a Neron differential. 

Consider  the embedding $j:X_{0,m}\to J_{0,m}$ induced by the cusp $i\infty$, and the modular parameterization $\phi=qj$. These extend to maps
$$
\Xcal_{0,m}^\infty\to \Jcal_{0,m}\to \Acal
$$
(that we still call $j$, $q$, $\phi$) by the N\'eron mapping property, where $\Xcal_{0,m}^\infty$ is obtained from $\Xcal_{0,m}$ by deleting the fibre components that do not meet the section $[i\infty]$. The special fibre of $\Xcal_{0,m}^\infty$ is irreducible and the $p$-adic valuation of the Manin constant $c_f$ is the vanishing order $\phi^\bullet \omega$ along it, as a section of the line bundle $\Omega^1_{\Xcal_{0,m}^\infty/R}$. 
Then one has
$$
v_p(c_f)=v\left(\frac{H^0(\Xcal_{0,m}^\infty,\Omega^1)}{R\cdot \phi^\bullet \omega}\right).
$$
That the expression on the right agrees with the vanishing order of $\phi^\bullet \omega$ on the special fibre of $\Xcal_{0,m}^\infty$ as a section of $\Omega^1_{\Xcal_{0,m}^\infty/R}$, is seen by considerations on $q$-expansions along the section $[i\infty]$.

Unfortunately, the geometry of $\Xcal_{0,m}$ is not convenient (in particular, duality theory is an issue). So we relate the previous expression to $\Xcal_m$ instead, using the forgetful degeneracy map $\alpha:\Xcal_m\to \Xcal_{0,m}$. We have
\begin{equation}\label{EqcAalpha}
v_p(c_f)=v\left(\frac{H^0(\Xcal_{0,m}^\infty,\Omega^1)}{R\cdot \phi^\bullet \omega}\right)\le v\left(\frac{H^0(\Xcal_{m}^\infty,\Omega^1)}{R\cdot (\phi\alpha)^\bullet \omega}\right)
\end{equation}
because $\alpha$ maps the cusp at infinity to the cusp at infinity, so, it restricts to $\Xcal_m^\infty \to \Xcal_{0,m}^\infty$.

However, there is the inconvenience (for later in our argument) that when $\ell\nmid m$, the modular form attached to $A$ is no longer new for the group $U_0(p^nm)\cap U_1(\ell)$, and in that case we are led to also consider the standard second degeneracy map $\beta:\Xcal_m\to \Xcal_{0,m}$ induced on complex points by the map $z\mapsto \ell z$ on $\hfrak$.

 The map $\beta$ sends the cusp $i\infty$ to itself, so it restricts  to $\Xcal_m^\infty \to \Xcal_{0,m}^\infty$, giving
\begin{equation}\label{EqcAbeta}
v_p(c_f)=v\left(\frac{H^0(\Xcal_{0,m}^\infty,\Omega^1)}{R\cdot \phi^\bullet \omega}\right)\le v\left(\frac{H^0(\Xcal_{m}^\infty,\Omega^1)}{R\cdot (\phi\beta)^\bullet \omega}\right).
\end{equation}
When $\ell\nmid m$, it is important to note (say, by looking at $q$-expansions) that $(\phi\alpha)^\bullet \omega$ and $(\phi\beta)^\bullet \omega$ are $R$-linearly independent.

\subsection{The case $\ell \mid m$} Suppose that $\ell | m$. Then the newform $f\in S_2(p^nm)$ attached to $A$ for the group $U_0(p^nm)$ continues to be new for the group $U_0(p^nm)\cap U_1(\ell)$. Then we have an optimal quotient $\theta: J_m\to C$ with  $C$ an elliptic curve over $\Q$ isogenous to $A$. By optimality of $\theta$, there is an isogeny $\pi: C\to A$ over $\Q$ such that the following diagram of morphisms over $\Q$  commutes
\begin{equation}\label{EqDiagramAlpha}
\begin{CD}
 J_m @>{\theta}>> C\\
 @VV{\alpha_*}V @VV{\pi}V\\
J_{0,m} @>{q}>> A.
\end{CD}
\end{equation}
Here, $\alpha_*$ is induced by the degeneracy map $\alpha: X_m\to X_{0,m}$ under Albanese functoriality. 

Let $\Ccal$ be the N\'eron model of $C$ over $R$. We have the commutative diagram of $R$-morphisms
$$
\begin{CD}
\Xcal_m^\infty @>{j_m}>> \Jcal_m @>{\theta}>> \Ccal\\
@V{\alpha}VV @VV{\alpha_*}V @VV{\pi}V\\
\Xcal_{0,m}^\infty @>{j}>>\Jcal_{0,m} @>{q}>> \Acal.
\end{CD}
$$

Dualizing \eqref{EqDiagramAlpha}, we find that $\alpha_*^\vee=\alpha^*$ (induced by Picard functoriality) extends $\pi^\vee$ after composition with the inclusions $q^\vee$ and $\theta^\vee$. Thus, $\deg \pi$ divides $\deg \alpha$, and the latter is 
$$
[U_0(p^nm):U_0(p^nm)\cap U_1(\ell)] \le [U(1) : U_1(\ell)].
$$
Hence
$$
v_p(\deg \pi) \ll_\ell 1. 
$$
Let $a\ge 0$ be an integer such that $\tilde{\omega}:= p^{-a}\pi^\bullet \omega$ is a N\'eron differential on $\Ccal$, and note that $a\le v_p(\deg\pi)\ll_\ell 1$. Then we have
$$
v_p(c_f)\le v\left(\frac{H^0(\Xcal_{m}^\infty,\Omega^1)}{R\cdot (\phi\alpha)^\bullet \omega}\right)= v\left(\frac{H^0(\Xcal_{m}^\infty,\Omega^1)}{R\cdot \phi_m^\bullet \tilde{\omega}}\right) + a
$$
where $\phi_m=\theta j_m$. Note that
$$
v\left(\frac{H^0(\Xcal_{m}^\infty,\Omega^1)}{R\cdot \phi_m^\bullet \tilde{\omega}}\right)\le v(\coker(j_m^\bullet)) + v\left(\frac{H^0(\Jcal_{m},\Omega^1)}{R\cdot \theta^\bullet \tilde{\omega}}\right)
$$
because $j_m^\bullet : H^0(\Jcal_{m},\Omega^1)\to H^0(\Xcal_{m}^\infty,\Omega^1)$ is injective. So by \eqref{EqConrad} we deduce
\begin{equation}\label{EqRed1}
v_p(c_f)\le v\left(\frac{H^0(\Jcal_{m},\Omega^1)}{R\cdot \theta^\bullet \tilde{\omega}}\right)+ O_{p,n,\ell}(1).
\end{equation}
Let $\chi : \T_{U_0(p^nm)\cap U_1(\ell)}\to \Z$ be the system of Hecke eigenvalues attached to the optimal quotient $\theta$. Let $P_{\chi}:H^0(\Jcal_{m},\Omega^1)_\Q\to H^0(\Jcal_{m},\Omega^1)_\Q$ be the orthogonal projection onto the $\chi$-component with respect to the Petersson inner product. Observe that $\theta^\bullet \tilde{\omega}\in H^0(\Jcal_{m},\Omega^1)^\chi$, and from the equation $\theta\theta^\vee = [\deg \phi_m]\in \End (C)$ we deduce that the following diagram commutes:
$$
\begin{CD}
 R\cdot \tilde{\omega} @>{(\deg \phi_m)\cdot}>>  R\cdot \tilde{\omega} \\
 @V{\theta^\bullet}VV @AAA\\
 H^0(\Jcal_{m},\Omega^1) @>>> \frac{H^0(\Jcal_{m},\Omega^1)}{ (H^0(\Jcal_{m},\Omega^1)^\chi)^\perp}
\end{CD}
$$
The rightmost arrow is induced by $(\theta^\vee)^\bullet$ and the fact that $\ker( (\theta^\vee)^\bullet) = (H^0(\Jcal_{m},\Omega^1)^\chi)^\perp$. The bottom-right term in the diagram can be replaced by $P_\chi(H^0(\Jcal_{m},\Omega^1))$, which is an $R$-module of rank $1$, and the bottom arrow can be replaced by $P_\chi$. Chasing the image of $\tilde{\omega}$ we deduce
\begin{equation}\label{EqM1}
v_p(\deg \phi_m)\ge v\left(\frac{H^0(\Jcal_{m},\Omega^1)}{R\cdot \theta^\bullet \tilde{\omega}}\right) + v\left(\frac{P_\chi(H^0(\Jcal_{m},\Omega^1))}{H^0(\Jcal_{m},\Omega^1)^\chi}\right).
\end{equation}
In \cite{ARSdeg} Theorem 3.6 (a), it is shown that the modular exponent (which equals the modular degree in the elliptic curve case) divides the \emph{congruence exponent} (defined in terms of Fourier expansions at $i\infty$) for groups of the form $U_0(N)$ and $U_1(N)$. The same proof works for intermediate subgroups such as $U_0(p^nm)\cap U_1(\ell)$. So, the integer $\deg \phi_m$ divides the congruence exponent, which by definition is the exponent of 
$$
\frac{S_m(\Z)}{(S_m(\Z)^\chi)^\perp + S_m(\Z)^\chi}
$$
where $S_m(\Z)$ is the subgroup of $S_2(U_0(p^nm)\cap U_1(\ell))$ consisting of modular forms with Fourier coefficients at $i\infty$  in $\Z$. By the $q$-expansion principle, we have a canonical isomorphism $S_m(\Z)\otimes R= H^0(\Xcal_m^\infty,\Omega^1)=H^0(\Xcal_m^\infty,\omega_m)$, which together with \eqref{EqDuality} gives
$$
\begin{aligned}
v_p(\deg \phi_m)&\le v\left(\frac{H^0(\Xcal_{m}^\infty,\omega_m)}{(H^0(\Xcal_{m}^\infty,\omega_m)^\chi)^\perp +H^0(\Xcal_{m}^\infty,\omega_m)^\chi}\right)\\
&\le v\left(\frac{H^0(\Xcal_{m}^\infty,\omega_m)}{(H^0(\Xcal_{m},\omega_m)^\chi)^\perp +H^0(\Xcal_{m},\omega_m)^\chi}\right)\\
&\le v\left(\frac{H^0(\Xcal_{m},\omega_m)}{(H^0(\Xcal_{m},\omega_m)^\chi)^\perp +H^0(\Xcal_{m},\omega_m)^\chi}\right) + v\left(\frac{H^0(\Xcal_{m}^\infty,\omega_m)}{H^0(\Xcal_{m},\omega_m)}\right)\\
&= v\left(\frac{P_\chi(H^0(\Jcal_{m},\Omega^1))}{H^0(\Jcal_{m},\Omega^1)^\chi}\right) + v(\coker(j_m^\bullet)).
\end{aligned}
$$
It follows from \eqref{EqConrad} and \eqref{EqM1} that 
$$
v\left(\frac{H^0(\Jcal_{m},\Omega^1)}{R\cdot \theta^\bullet \tilde{\omega}}\right) \ll_{p,n,\ell} 1
$$
which by \eqref{EqRed1} proves
$$
v_p(c_f)\ll_{p,n,\ell}  1.
$$
This concludes the proof of Theorem \ref{ThmReductionManin} in the case $\ell | m$. Note that the bound is independent of $S$, and we actually obtain
$$
v_p(c_f)\le 2 v(\coker (j_m^\bullet)) + v_p(\deg \pi).
$$

\subsection{The case $\ell \nmid m$} Now we assume that $\ell \nmid m$ (and, as always, $\ell\ne p$). Then the newform $f\in S_2(p^nm)$ attached to $A$ for the group $U_0(p^nm)$ is no longer new for the group $U_0(p^nm)\cap U_1(\ell)$. In fact, let $\chi : \T_{U_0(p^nm)\cap U_1(\ell)}\to \Z$ be the system of Hecke eigenvalues attached to this old form, then $\dim_\Q H^0(J_m,\Omega^1)^\chi =2$, which is explained by the two degeneracy maps $\alpha, \beta: X_m\to X_{0,m}$.

Nevertheless, attached to $\chi$ we have an optimal quotient $\theta:J_m\to \Sigma$ over $\Q$, where $\Sigma$ is an abelian surface isogenous to $A\times A$. By optimality of $\theta$ and considering the relevant cotangent spaces, we see that there is an isogeny $\pi: \Sigma\to A\times A$ over $\Q$ making the following diagram commutative:
\begin{equation}\label{EqDiagramAlphaBeta}
\begin{CD}
 J_m @>{\theta}>> \Sigma\\
 @VV{\alpha_*\times \beta_*}V @VV{\pi}V\\
J_{0,m}^2 @>{q\times q}>> A\times A.
\end{CD}
\end{equation}
\begin{lemma} There is an integer $u\ll_{S,p,\ell} 1$ such that $p^u$ annihilates the $p$-primary part of the kernel of $\pi$.
\end{lemma}
\begin{proof} Let $\sigma : X_m\to X_0(p^nm\ell)$ be the forgetful map, and write $\alpha_0,\beta_0:X_0(p^nm\ell)\to X_{0,m}$ for the two degeneracy maps. We have $\alpha=\alpha_0\sigma$ and $\beta=\beta_0\sigma$, hence,  the following diagram commutes:
\begin{equation}\label{EqDiagramShimura}
\begin{CD}
J_0(p^nm\ell)  @>{\sigma^*}>> J_m @<{\theta^\vee}<< \Sigma^\vee\\
 @A{\alpha_{0}^*+ \beta_{0}^*}AA @AA{\alpha^*+ \beta^*}A @AA{\pi^\vee}A\\
J_{0,m}^2 @<{=}<< J_{0,m}^2 @<{q^\vee\times q^\vee}<< A\times A
\end{CD}
\end{equation}
The kernel of $\pi^\vee$ is Cartier-dual to that of $\pi$, so, it suffices to prove the claim for $\ker(\pi^\vee)$ instead.

The maps $q^\vee\times q^\vee$ and $\theta^\vee$ are injective as they are duals of optimal quotients, so, it suffices to bound a power of $p$ that annihilates the $p$-primary part of $\ker(\alpha^*+\beta^*)\cap q^\vee(A)\times q^\vee(A)$. 

Note that $\sigma_*\sigma^*=[\deg \sigma]\in \End(J_0(p^nm))$, so $v(\ker(\sigma^*))\ll_\ell 1$. Hence  
$$
v\left(\ker((\alpha^*+\beta^*)\circ (q^\vee\times q^\vee))\right)= v\left(\ker((\alpha_0^*+\beta_0^*)\circ (q^\vee\times q^\vee)) \right) + O_\ell(1)
$$
and moreover  $\ker((\alpha_0^*+\beta_0^*)\circ (q^\vee\times q^\vee))$ is isomorphic to
$$
Z=\ker(\alpha_0^*+\beta_0^*)\cap (q^\vee(A)\times q^\vee(A)).
$$
By Ihara's lemma in Ribet's formulation \cite{RibetICM, RibetEisenstein}, we have that $\ker(\alpha_0^*+\beta_0^*)$ is Eisenstein (in Mazur's terminology), so that for every prime $r\nmid pm\ell$ one has that $T_r$ acts as $r+1$ on it. On the other hand, when $r\nmid pm\ell$ we have that $T_r$ acts as $\chi(T_r)=a_r(A)$ on $A\times A\subseteq J_{0,m}^2$. Thus, if $Z$ has a (geometric) point of exact order $p^e$, we see that for all primes $r\nmid pm\ell$ we have $a_r(A)\equiv r+1 \mod p^e$. By Lemma \ref{LemmaCongIneff} we obtain $e\ll_{S,p} 1$, which concludes the proof.
\end{proof}

The modular exponent $\tilde{n}_\Sigma$ of the optimal quotient $\theta: J_m\to \Sigma$ is defined as in \cite{ARSdeg}, namely, as the exponent of the group $\ker(\theta\theta^\vee)$. The theory of \cite{ARSdeg} applies to the abelian surface $\Sigma$ as in the first example of Section 3 in \emph{loc. cit.}, with only some minor modifications due to the fact that we are working with the group $U_0(p^nm)\cap U_1(\ell)$ rather than a group of the form $U_0(N)$ or $U_1(N)$.
 
Consider the map 
$$
\tau= \pi\theta: J_m\to A\times A. 
$$
Then $\tau\tau^\vee=\pi\theta\theta^\vee\pi^\vee$, so that $v(\ker(\tau\tau^\vee))= v_p(\tilde{n}_\Sigma) +O_{S,p,\ell}(1)$ by the previous lemma.

On the other hand, from the description $\tau= (q\times q)\circ(\alpha_*\times \beta_*)=(q\alpha_*)\times (q\beta_*)$ we see that 
\begin{equation}\label{EqTauAsMult}
\tau\tau^\vee=[\deg(\phi)\deg(\alpha)]_A\times[\deg(\phi)\deg(\beta)]_A= [\deg(\phi)\deg(\alpha)]_{A\times A}.
\end{equation}
 It follows that
\begin{equation}\label{EqDegCompar}
v_p(\deg(\phi)\deg(\alpha)) = v_p(\tilde{n}_\Sigma) + O_{S,p,\ell}(1).
\end{equation}
Take any $\tilde{\omega}\in H^0(\Acal^2,\Omega^1)$ primitive. We make two observations about $\tilde{\omega}$. 

First, we have
\begin{equation}\label{EqFirstAux}
v\left(\frac{H^0(\Xcal_m^\infty,\Omega^1)}{R\cdot (\tau j_m)^\bullet \tilde{\omega}}\right)\le v(\coker(j_m^\bullet)) + v\left(\frac{H^0(\Jcal_m,\Omega^1)}{R\cdot \tau^\bullet \tilde{\omega}}\right).
\end{equation}
Secondly, we have
\begin{equation}\label{EqSecondAux}
v_p(\deg(\phi)\deg(\alpha))\ge  v\left(\frac{H^0(\Jcal_m,\Omega^1)}{R\cdot \tau^\bullet \tilde{\omega}}\right) + v\left(\frac{P_\chi (H^0(\Jcal_m,\Omega^1))}{H^0(\Jcal_m,\Omega^1)^\chi\cap \Q\cdot \tau^\bullet\tilde{\omega}}\right).
\end{equation}
This last bound is proved  by recalling \eqref{EqTauAsMult}, chasing $\tilde{\omega}$ in the diagram
$$
\begin{CD}
 H^0(\Acal^2,\Omega^1) @>{(\tau\tau^\vee)^\bullet}>>   H^0(\Acal^2,\Omega^1) \\
 @V{\tau^\bullet}VV @AAA\\
 H^0(\Jcal_{m},\Omega^1) @>>> \frac{H^0(\Jcal_{m},\Omega^1)}{ (H^0(\Jcal_{m},\Omega^1)^\chi)^\perp} \quad ,
\end{CD}
$$
noticing that the rightmost arrow is injective, replacing the bottom right term in the diagram by $P_\chi(H^0(\Jcal_m,\Omega^1))$, and noticing that $H^0(\Jcal_m,\Omega^1)\cap \Q\cdot \tau^\bullet\tilde{\omega}=H^0(\Jcal_m,\Omega^1)^\chi \cap \Q\cdot \tau^\bullet\tilde{\omega}$. Here, $P_\chi$ is the orthogonal projection onto the $\chi$-isotypical component.

Furthermore, since $\tau^\bullet$ maps $H^0(\Acal^2,\Omega^1)$ into $H^0(\Jcal_m,\Omega^1)^\chi$ with torsion cokernel, we see that there is some $\tilde{\omega}_0\in H^0(\Acal^2,\Omega^1)$ primitive such that
$$
v\left(\frac{P_\chi (H^0(\Jcal_m,\Omega^1))}{H^0(\Jcal_m,\Omega^1)^\chi\cap \Q\cdot \tau^\bullet\tilde{\omega}_0}\right) = v\left(\frac{P_\chi (H^0(\Jcal_m,\Omega^1))}{H^0(\Jcal_m,\Omega^1)^\chi}\right)
$$
We fix a choice of such an $\tilde{\omega}_0$, and for it we obtain from \eqref{EqDegCompar}, \eqref{EqFirstAux},  \eqref{EqSecondAux} and \eqref{EqConrad} that
\begin{equation}\label{EqM2}
v\left(\frac{P_\chi (H^0(\Jcal_m,\Omega^1))}{H^0(\Jcal_m,\Omega^1)^\chi}\right)  + v\left(\frac{H^0(\Xcal_m^\infty,\Omega^1)}{R\cdot (\tau j_m)^\bullet \tilde{\omega}_0}\right) \le v_p(\tilde{n}_\Sigma) + O_{S,p,n,\ell}(1).
\end{equation}
By \cite{ARSdeg} as in the case $\ell |m$ above, the modular exponent $\tilde{n}_\Sigma$ divides the congruence exponent associated to $\Sigma$, and one deduces
$$
v_p(\tilde{n}_\Sigma)\le v\left(\frac{P_\chi(H^0(\Jcal_{m},\Omega^1))}{H^0(\Jcal_{m},\Omega^1)^\chi}\right) + v(\coker(j_m^\bullet)) 
$$
which together with \eqref{EqM2} gives
\begin{equation}\label{EqFinalBoundManin}
 v\left(\frac{H^0(\Xcal_m^\infty,\Omega^1)}{R\cdot (\tau j_m)^\bullet \tilde{\omega}_0}\right) \ll_{S,p,n,\ell} 1.
\end{equation}
Recall that we had a Neron differential $\omega\in H^0(\Acal,\Omega^1)$. Let $\pi_1,\pi_2$ be the two projections $\Acal^2\to \Acal$ and let $\omega_i=\pi_i^\bullet\omega$. Then $\omega_1,\omega_2$ generate $H^0(\Acal^2,\Omega^1)$ as an $R$-module and we can write $\tilde{\omega}_0=r_1\omega_1 + r_2\omega_2$ with $r_1,r_2\in R$, at least one of them a unit. We observe that
$$
\begin{aligned}
(\tau j_m)^\bullet \tilde{\omega}_0 &=r_1 (\tau j_m)^\bullet(\omega_1)+r_2 (\tau j_m)^\bullet(\omega_2) \\
&= r_1 (\tau j_m)^\bullet\pi_1^\bullet(\omega)+ r_2 (\tau j_m)^\bullet\pi_2^\bullet(\omega)\\
&= r_1 (\pi_1\tau j_m)^\bullet(\omega)+ r_2 (\pi_2\tau j_m)^\bullet(\omega)\\
&= r_1 (q\alpha_* j_m)^\bullet(\omega)+ r_2 (q\beta_* j_m)^\bullet(\omega)\\
&= r_1 (qj\alpha)^\bullet(\omega)+ r_2 (qj\beta)^\bullet(\omega)\\
&= r_1 (\phi\alpha)^\bullet(\omega)+ r_2 (\phi\beta)^\bullet(\omega).
\end{aligned}
$$
We claim that
\begin{equation}\label{EqFinalClaim}
\min\left\{v\left(\frac{H^0(\Xcal_m^\infty,\Omega^1)}{R\cdot (\phi\alpha)^\bullet(\omega)}\right) ,v\left(\frac{H^0(\Xcal_m^\infty,\Omega^1)}{R\cdot (\phi\beta)^\bullet(\omega)}\right) \right\} \le v\left(\frac{H^0(\Xcal_m^\infty,\Omega^1)}{R\cdot (\tau j_m)^\bullet \tilde{\omega}_0}\right). 
\end{equation}
In fact, this follows by interpreting the three terms appearing in the expression as the vanishing orders of the corresponding differential forms along the (irreducible) special fibre of $\Xcal_m^\infty$. 

Finally, from \eqref{EqcAalpha} and \eqref{EqcAbeta} we obtain
$$
v_p(c_f)\le \min\left\{v\left(\frac{H^0(\Xcal_m^\infty,\Omega^1)}{R\cdot (\phi\alpha)^\bullet(\omega)}\right) ,v\left(\frac{H^0(\Xcal_m^\infty,\Omega^1)}{R\cdot (\phi\beta)^\bullet(\omega)}\right) \right\}.
$$
By \eqref{EqFinalBoundManin} and \eqref{EqFinalClaim} we get 
$$
v_p(c_f)\ll_{S,p,n,\ell} 1
$$ 
which concludes the proof of Theorem \ref{ThmReductionManin} in the case $\ell\nmid m$. This completes the proof of Theorem \ref{ThmReductionManin}, hence, of Theorem \ref{ThmManinCt}.

\section{Counting imaginary quadratic extensions of $\Q$} \label{SecCountingFields}

\subsection{The counting result} We will be interested in having a suitable Heegner point in $X_0^D(M)$. This will be achieved by showing the existence of sufficiently many imaginary quadratic extensions $K/\Q$ satisfying certain technical conditions. 

For a Dirichlet character $\chi$, we let $L(s,\chi)$ be its $L$-function.  If $D$ is a fundamental discriminant, we write $K_D$ for the quadratic number field of discriminant $D$, and $\chi_D$ for the non-trivial quadratic Dirichlet character associated to $K_D$. The counting result is the following.

\begin{theorem} \label{ThmCountingFields} There is a uniform constant $\kappa$ such that the following holds:

Let $\theta>0$. For $x>1$ and coprime positive integers $D$ and $M$, let $S_\theta (D,M,x)$ be the set of positive integers $d$ satisfying the following conditions:
\begin{itemize}
 \item[(i)] $x<d\le 2x$
 \item[(ii)] $d\equiv -1 \mod 4$ and $d$ is squarefree (hence, $-d$ is a fundamental discriminant);
 \item[(iii)] $p$ splits in $K_{-d}$ for each prime $p|M$
 \item[(iv)] $p$ is inert in $K_{-d}$ for each prime $p|D$
 \item[(v)] $\#Cl(K_{-d})> d^{0.5-\theta}$ 
 \item[(vi)] $\left|\frac{L'}{L}(1,\chi_{-d})\right|< \kappa \log \log d$.
\end{itemize}
Then, writing $N=DM$, we have that for $x\gg_\theta 1$
$$
\# S_\theta(D,M,x) = \frac{x}{2^{\omega(N)+1}\zeta(2)\prod_{p|2N}(1+\frac{1}{p})} + O(x^{1/2}N^{1/4}(\log N)^{1/2})
$$
where the implicit constant is absolute.
\end{theorem}
\begin{corollary} \label{CoroCounting} Let $\theta>0$ and $\delta>0$. With the notation of Theorem \ref{ThmCountingFields}, for $N\gg_{\theta,\delta} 1$ and $x>  N^{0.5 +\delta}$, we have
$$
\#S_\theta(D,M,x)> x^{1-\delta}.
$$
\end{corollary}
\subsection{Preliminaries on $L$-functions}

We need two analytic results about $L(s,\chi_D)$ where $D$ is a fundamental discriminant.

The next result is Corollary 2.5 in \cite{Lamzouri}.
\begin{proposition}\label{PropLam} Let $\epsilon>0$. There is a number $c_\epsilon>0$ depending only on $\epsilon$ such that the bound
$$
\left| \frac{L'}{L}(1,\chi_D)\right| <c_\epsilon \log \log |D|
$$
holds for all but $O_\epsilon(x^\epsilon)$ fundamental discriminants $D$ with $|D|<x$.
\end{proposition}
We also need Siegel's classical bound for the class number. The following explicit version is due to Tatuzawa \cite{Tatuzawa}.
\begin{proposition}\label{PropSiegel} Let $0<\epsilon<1/12$. For $-d$ a negative fundamental discriminant with $d>e^{1/\epsilon}$, the bound
$$
\#Cl(\Q(\sqrt{-d})) > 0.2\cdot d^{0.5-\epsilon}
$$
holds with at most one exception.
\end{proposition}
\subsection{Lemmas on squarefrees}

\begin{lemma}\label{LemmaSqfChar} Let $m>1$ be a positive integer and let $\chi$ be a non-principal Dirichlet character to the modulus $m$ (not necessarily primitive). For $x\ge  1$ we have
$$
\left|\sum_{d\le x} \mu^2(d)\chi(d)\right| \le 5 x^{1/2}m^{1/4}(\log m)^{1/2}.
$$
\end{lemma}
\begin{proof} When $x\le m^{1/2}\log m$ we see that the required sum is bounded by
$$
x\le x^{1/2}m^{1/4}(\log m)^{1/2}
$$
so we may assume that $x>m^{1/2}\log m$.

Let $y\le x^{1/2}$ be a positive integer. We have
$$
\sum_{d\le x} \mu^2(d)\chi(d) = \sum_{d\le x} \chi(d)\sum_{k^2|d}\mu(k)=\sum_{k\le x^{1/2}} \mu(k)\chi(k^2)\sum_{r\le x/k^2}\chi(r) = S_1 + S_2
$$
where $S_1$ corresponds to terms with $k\le y$, and $S_2$ corresponds to those with $y<k\le x^{1/2}$. By Polya-Vinogradov 
$$
|S_1|\le 2 y m^{1/2}\log m,
$$
while for $S_2$ we have the trivial bound
$$
|S_2|\le x\sum_{k>y}\frac{1}{k^2} < \frac{x}{y}.
$$
Take $y=\lceil x^{1/2}m^{-1/4}(\log m)^{-1/2}\rceil$ and note that
$$
\frac{x^{1/2}}{m^{1/4}(\log m)^{1/2}}\le y \le \frac{2x^{1/2}}{m^{1/4}(\log m)^{1/2}}
$$
because $x> m^{1/2}\log m$. The result follows.
\end{proof}
\begin{lemma}\label{LemmaCountCoprSqf} Let $a$ be and odd residue class modulo $4$, let $m$ be a positive integer and let $x>1$.  Then
$$
\sum_{\substack{d\le x\\ (d,m)=1\\ d\equiv a\,(4)}}\mu^2(d)=\frac{x}{2\zeta(2)\prod_{p|2m}\left(1+\frac{1}{p}\right)} + O(x^{1/2}m^{1/4}(\log m)^{1/2}).
$$
\end{lemma}
\begin{proof} Let $\psi_0$ and $\psi$ be the principal and the non-principal characters modulo $4$ respectively. By the previous lemma we have
$$
\begin{aligned}
2 \sum_{\substack{d\le x\\ (d,m)=1\\ d\equiv a\,(4)}}\mu^2(d) &= \sum_{\substack{d\le x\\ (d,m)=1}}\mu^2(d)\psi_0(d) + \psi(a)\sum_{\substack{d\le x\\ (d,m)=1}}\mu^2(d)\psi(d)\\
&= \sum_{\substack{d\le x\\ (d,2m)=1}}\mu^2(d) + O(x^{1/2}m^{1/4}(\log m)^{1/2}).
\end{aligned}
$$
The number of positive integers coprime to $A$ up to a bound $y$ is
$$
\sum_{n\le y} \sum_{b|(A,n)}\mu(b) = \sum_{b|A}\mu(b)\left(\frac{y}{b}+O(1)\right) = \frac{\phi(A)y}{A} + O(2^{\omega(A)}).
$$
Taking $A=2m$ and writing $\mu^2(d)=\sum_{b^2|d}\mu(b)$ we find
$$
\begin{aligned}
\sum_{\substack{d\le x\\ (d,2m)=1}}\mu^2(d) &= \sum_{\substack{b\le x^{1/2}\\(b,2m)=1}}\mu(b) \sum_{\substack{c\le x/b^2\\(c,2m)=1}}1\\
&= \frac{\phi(2m)x}{2m}\sum_{\substack{b\le x^{1/2}\\(b,2m)=1}}\frac{\mu(b)}{b^2}  + O(2^{\omega(m)}x^{1/2})\\
&= \frac{\phi(2m)x}{2m}\sum_{\substack{b=1\\(b,2m)=1}}^\infty\frac{\mu(b)}{b^2}  + O(x^{1/2} + 2^{\omega(m)}x^{1/2}).
\end{aligned}
$$
The infinite series is no other than $1/\zeta(2)$ with the Euler factors for primes dividing $2m$ removed, hence the result because $2^{\omega(m)}\ll_\epsilon m^\epsilon$.
\end{proof}

\subsection{Proof of the counting result}
We will restrict ourselves to $d>0$ squarefree satisfying $d\equiv 3 \, (4)$, in which case $-d$ is a negative fundamental discriminant. Thus,  $\chi_{-d}$ is the non-principal Dirichlet character associated with the imaginary quadratic field $K_{-d}=\Q(\sqrt{-d})$, and the discriminant of $K_{-d}$ is precisely $-d$. 

Under our assumption on $d$, the character $\chi_{-d}$ is the Kronecker symbol $(\frac{-d}{.})$. It has conductor $d$ and is determined by the following conditions on primes: for $p$ odd, $\chi_{-d}(p)$ is the Legendre symbol $(\frac{-d}{p})$, and for $p=2$ we have
$$
\chi_{-d}(2)=\begin{cases}
1&\mbox{ if }d\equiv 7\,(8)\\
-1&\mbox{ if }d\equiv 3\,(8).
\end{cases}
$$
Equivalently, the values of $\chi_{-d}$ at primes are determined by 
$$
\chi_{-d}(p)=\begin{cases}
0 &\mbox{ if $p$ ramifies in }K_{-d}\\
1 &\mbox{ if $p$ splits in }K_{-d}\\
-1  &\mbox{ if $p$ is inert in }K_{-d}.
\end{cases}
$$
\begin{proof}[Proof of Theorem \ref{ThmCountingFields}] Let $S'(D,M,x)$ be the set of positive integers $d$ satisfying conditions (i), (ii), (iii), and (iv) of the statement of Theorem  \ref{ThmCountingFields}. For each pair $(D,M)$ and each prime $p|N=DM$, define the numbers 
$$
\epsilon_p=\epsilon_p(D,M) =\begin{cases}
1 &\mbox{ if }p|M\\
-1 &\mbox{ if }p|D
\end{cases}
$$
and define $\epsilon_b$ for all divisors $b$ of $N$ as $\epsilon_b=0$ if $b$ is not squarefree, and multiplicatively otherwise (using the numbers $\epsilon_p$). Then we have
$$
\begin{aligned}
\# S'(D,M,x)&=2^{-\omega(N)}\sum_{\substack{x<d\le 2x\\ (d,N)=1\\ d\equiv 3\,(4)}}\mu(d)^2\prod_{p|N}\left(1+\epsilon_p\chi_{-d}(p)\right)\\
&=2^{-\omega(N)}\sum_{\substack{x<d\le 2x\\ (d,N)=1\\ d\equiv 3\,(4)}}\mu(d)^2\sum_{b|N}\epsilon_b\chi_{-d}(b).
\end{aligned}
$$
Writing  $1_N$ for the principal character to the modulus $N$, we find
$$
\# S'(D,M,x)=2^{-\omega(N)}\sum_{b|N}\epsilon_b \sum_{\substack{x<d\le 2x\\ d\equiv 3\,(4)}} \mu^2(d)\left(\frac{-d}{b}\right)\cdot 1_N(d).
$$
By Lemma \ref{LemmaCountCoprSqf}  the contribution to the previous expression coming from  $b=1$ is
$$
2^{-\omega(N)}\sum_{\substack{x<d\le 2x\\ (d,N)=1\\ d\equiv 3\,(4)}} \mu^2(d) = \frac{x}{2^{\omega(N)+1}\zeta(2) \prod_{p|2N}\left(1+\frac{1}{p}\right)} + O(x^{1/2}N^{1/4}(\log N)^{1/2}).
$$
When $b\ne 1$ is a squarefree divisor of $N$, the function $d\mapsto (\frac{-d}{b})\cdot 1_N(d)$ is a Dirichlet character whose modulus divides $4N$, and it is non-principal because  $b$ is squarefree. Thus, by Lemma \ref{LemmaSqfChar} we get
$$
2^{-\omega(N)}\sum_{\substack{b\ne 1\\ b|N}}\epsilon_b \sum_{\substack{x<d\le 2x\\ d\equiv 3\,(4)}} \mu^2(d)\left(\frac{-d}{b}\right)\cdot 1_N(d) \ll x^{1/2}N^{1/4}(\log N)^{1/2}.
$$
Therefore we find
$$
\# S'(D,M,x)=\frac{x}{2^{\omega(N)+1}\zeta(2) \prod_{p|2N}\left(1+\frac{1}{p}\right)} + O\left(x^{1/2}N^{1/4}(\log N)^{1/2}\right).
$$
Finally, by Proposition \ref{PropLam} with $\epsilon=1/2$ and Proposition \ref{PropSiegel} with $\epsilon=\theta/2$ we see that for $x\gg_\theta 1$ we have
$$
\# S'(D,M,x)- \# S_\theta(D,M,x)\ll x^{1/2} 
$$
which concludes the proof, with the constant $\kappa=c_{1/2}$ from Proposition \ref{PropLam}.
\end{proof}
\section{Arakelov degrees }
The purpose of this brief section is to introduce some notation related to Arakelov degrees of metrized line bundles. This will be used at various places later in the paper.

\subsection{Metrized line bundles on arithmetic curves} Let $L$ be a number field and write $S_L=\Spec O_L$.  If $\Mcal$ is an invertible sheaf (also called line bundle) on $S_L$, then $\Mfrak=H^0(S_L,\Mcal)$ is a projective $O_L$-module of rank $1$, and since $S_L$ is affine we have $\Mcal=\Mfrak^\sim$. In this way, invertible sheaves on $S_L$ correspond to projective $O_L$-modules of rank $1$. 

For every embedding $\sigma:L\to\C$ we let $|-|_\sigma$ be the absolute value on $L$ induced by $\sigma$. A metrized line bundle on $S_L$ is an invertible sheaf $\Mcal$ (or equivalently, its associated projective $O_L$-module $\Mfrak$) together with the following data:  for each embedding $\sigma:L\to\C$, a norm $\|-\|_\sigma$  on $M_\sigma:=\Mfrak \otimes_\sigma\C$ compatible with the absolute value $|-|_\sigma$. Thus, a metrized invertible sheaf on $S_L$ is a pair $\widehat{\Mcal}=(\Mcal,\{\|-\|_\sigma\}_\sigma)$ with the notation as before.


\subsection{Arakelov degree} The Arakelov degree $\widehat{\deg}_L\widehat{\Mcal}$ of a metrized line bundle $\widehat{\Mcal}$ on $S_L$ is defined as follows: take any non-zero $\eta\in \Mfrak=H^0(S_L,\Mcal)$, then
$$
\widehat{\deg}_L\widehat{\Mcal} :=\log \#\left(\Mfrak/\langle \eta\rangle\right) - \sum_{\sigma: L \to \C} \log \|\eta\|_\sigma,
$$
which is independent of the choice of $\eta\ne 0$. The Arakelov degree $\widehat{\deg}_L$ is additive on tensor products of metrized line bundles, so, it has an obvious extension to metrized $\Q$-line bundles.

We conclude by making two observations that will be useful in later sections of the paper. First, for any choice of $0\ne \eta\in\Mfrak$ we have
$$
\widehat{\deg}_L\widehat{\Mcal} \ge - \sum_{\sigma: L \to \C} \log \|\eta\|_\sigma.
$$
Secondly, note that if $\Ncal$ is an invertible sub-sheaf of $\Mcal$, and the metrized line bundle $\widehat{\Ncal}$ is defined by restricting the metrics at infinity of the metrized line bundle $\widehat{\Mcal}$, then we have
$$
\widehat{\deg}_L\widehat{\Ncal}\le \widehat{\deg}_L\widehat{\Mcal}.
$$

\section{Arakelov height of Heegner points}

\subsection{Heegner hypothesis for $(D,M)$} Let $D$ and $M$ be coprime positive integers with $D$ squarefree and $\omega(D)$ even. We say that a quadratic number field $K/\Q$ satisfies the \emph{Heegner hypothesis for the pair $(D,M)$} if every prime $p|D$ is inert in $K$ and every prime $p|M$ splits in $K$. In particular, the primes dividing $DM$ are unramified in $K$.

If $K$ satisfies the Heegner hypothesis for $(D,M)$, there is an embedding $\psi: K\to B$ (with $B$ the quaternion algebra of discriminant $D$) which is optimal for the Eichler order $R_0^D(M)\subseteq O_B$ of reduced discriminant $M$, in the sense that $\psi^{-1}(R_0^D(M))=O_K$. 

Fixing a choice of $\psi$ this leads to a point $\tau_K\in\hfrak$ and the corresponding point $P_K\in X(K^{ab})$ (cf. Paragraph \ref{SubSecHeegner}) with $X=\varprojlim X_U$ (cf. Paragraph \ref{SecNotationPro}). Its image $P_{K,D,M}:=P_{K,U_0^D(M)}$ in $X_0^D(M)$ has residue field $H_K$, the Hilbert class field of $K$.  

We call these points \emph{$(D,M)$-Heegner points}. They are a particular case of the $D$-Heegner points discussed in Paragraph \ref{SubSecHeegner}. 


\subsection{Reduction of Heegner points} We write $\Xcal_0^D(M)=\Xcal_0^D(M, U^D_1(1))$ for the normal integral model of $X_0^D(M)$, flat, projective over $\Z$, introduced  in Paragraph \ref{SecFinPart1}, and we let $\Xcal_0^D(M)^0=\Xcal_0^D(M,U^D_0(1))^0$ be the smooth locus of its structure map as in \emph{loc cit}.

\begin{lemma}\label{LemmaRed} Let $K$ be a quadratic imaginary extension of $\Q$ satisfying the $(D,M)$-Heegner hypothesis. Let $U$ be be a compact open subgroup of $O_\B^\times$ with $m_U$ coprime to $DM$. Let $C$ be the closure of $P_{K,U_0^D(M)\cap U}$ in the surface $\Xcal_0^D(M,U)$. Then $C$ is contained in $\Xcal_0^D(M,U)^0$.
\end{lemma}
\begin{proof} We observe that $\Xcal_0^D(M,U)^0$ is the preimage of $\Xcal_0^D(M)^0$ via the forgetful morphism $\Xcal_0^D(M,U)\to \Xcal_0^D(M)$. Thus, we may assume that $U=U^D(1)$. Furthermore, we only need to study the intersection of $C$ with fibres at $p$ with $p|DM$. 

If $p|D$ then $P_{K,U_0^D(M)}$ does not reduce to a supersingular point  because $p$ does not ramify in $K$, and this suffices.  

If $p|M$ then $P_{K,U_0^D(M)}$ does not reduce to a supersingular point  because $p$ splits in $K$ and by same argument as in p. 256 in \cite{GrossZagier}. It remains to show that $C$ does not meet a non-reduced component of the fibre at $p$ (this case only occurs when $M$ is not squarefree).

Since $K$ satisfies the Heegner hypothesis for $(D,M)$, the residue field of $P_{K,U_0^D(M)}$ is $H_K$. As $p$ does not ramify in $K$, it does not ramify in $H_K$ and we can base change to $H_K$ obtaining an \'etale cover near the fibre at $p$. Now the multi-section $C$ is the image of a section $C_{H_K}$ of the structure map to $\Spec O_{H_K}$. Blowing-up the supersingular points of characteristic $p$ we may work on a regular surface, and now Corollary 1.32 in p. 388 of \cite{Liu} shows that $C_{H_K}$ does not meet a non-reduced fibre. 
\end{proof}

We remark that for $D=1$, this is proved in Proposition (3.1) in \cite{GrossZagier}. One could also adapt the argument there to analyze the case $p|M$ in the previous proof.

\subsection{Metrized canonical sheaf} The surface $\Xcal_U=\Xcal_0^D(1,U)$ is semi-stable over $\Z[m_U^{-1}]$, and when $U$ is good enough in the sense of Paragraph \ref{SecFinPart1}, then $\Xcal_U$ is regular. Let $\omega_U$ be the relative dualizing sheaf of $\Xcal_U\to \Spec\Z[m_U^{-1}]$. Then $\omega_U$ is an invertible sheaf.

For each inclusion $U\subseteq V$ of open compact subgroups with $U$ and $V$ good enough, the map $\pi^U_V:\Xcal_U\to\Xcal_V$ is \'etale and 
\begin{equation}\label{EqCompatible}
(\pi^U_V)^*\omega_V= \omega_U
\end{equation}
by Ch. 6 of \cite{Liu}, more precisely, Lemma 4.26 and Theorem 2.32 in \emph{loc. cit}.

For each $U$ good enough, the \emph{metrized canonical sheaf} $\widehat{\omega}_U$ on the surface $\Xcal_U$ (over $\Z[m_U^{-1}]$) is defined as the line sheaf $\omega_U$ endowed with the following metric $\| - \|_U$ at infinity (defined away from cuspidal points, if any):

The base change of $\Xcal_U$ to $\C$ is $X_U^{an}$ and the relative dualizing sheaf becomes $\Omega^1_{X_U^{an}}$. The decomposition \eqref{EqDecomposition} has only one term (because $U$ is good enough), so we can take $g_a=1$ and drop it from the related notation. The uniformization $\xi_{U}:\hfrak \to \tilde{\Gamma}_{U}\backslash\hfrak$ is unramified, and the metric on $\Omega^1_{X_U^{an}}$ is defined by
\begin{equation}\label{EqYZMetric}
\|\alpha_P\|_U = 2|f(\tau)|\Im(\tau)
\end{equation}
for $P\in \tilde{\Gamma}_{U}\backslash\hfrak\subseteq X_U^{an}$, $\alpha$ a regular section of $\Omega^1_{X_U^{an}}$ near $P$, $\tau\in\hfrak$ a pre-image of $P$ under $\xi_{U}$, and $f$ holomorphic near $\tau$ such that  $\xi_{U,g_a}^\bullet\alpha=f(z)dz$ on an appropriate neighborhood of $\tau\in \hfrak$.

When $U$ is not good enough, we do not define a metrized canonical sheaf.  We observe that the metric that we put on $\omega_U$ is not the Arakelov canonical metric, but instead, the hyperbolic metric. When $D=1$ this metric has singularities at the cusps; otherwise, it is smooth.

At least two other alternative approaches to define metrized canonical bundles on integral models of Shimura curves are available in the literature. In \cite{KRY} one works on the stack-theoretical integral model of $X_0^D(1)$, while in \cite{YuZh} one proceeds as above with two technical differences: the integral models for $U$ sufficiently small used there come from the theory of integral models of curves of genus $g\ge 2$, and then quotient maps are used to define a version of $\widehat{\omega}_U$ as metrized $\Q$-line bundles even when $U$ is not sufficiently small. The geometric properties are deduced by relating them to integral models of auxiliary Shimura curves. The  construction of integral models in \cite{YuZh} has the technical advantage of being available even beyond the case of Shimura curves over $\Q$, where a direct modular interpretation is no longer possible.

We will be interested on the Arakelov height of Heegner points with respect to the metrized line bundles $\widehat{\omega}_U$ (suitably defined in the next paragraph).  The three methods for defining a metrized canonical sheaf are in fact equivalent for this purpose, as we will explain.

\subsection{Arakelov height}\label{SubSecHeight}

Let $P$ be an algebraic point in $X=\varprojlim X_U$. For each $U$ we denote by $P_U$ the image of $P$ in $X_U$.

We define the Arakelov-theoretical \emph{height of $P$ with respect to the metrized canonical bundles}, denoted by $h_{Ar}(P)$,  as follows:

Take any finite collection $\Ucal=\{U_j\}_{j=1}^r$ of good enough open compact subgroups such that the numbers $m_j=m_{U_j}$ satisfy $\gcd(m_1,...,m_r)=1$. Let $U=\cap_{j=1}^r U_j$ and note that $m_U$ has the same prime factors as $m_1\cdots m_r$, and it is good enough too. Let $F_{P,U}$ be a field containing the residue field of $P_U$ and let $S_{P,U}=\Spec O_{F_{P,U}}$. The point $P_U$ extends to a map $s_U:S_{P,U}[m_U^{-1}]\to \Xcal_U$. Similarly, the points $P_{U_j}$ extend to maps $s_j: S_{P,U}[m_j^{-1}]\to \Xcal_{U_j}$, which are compatible in the sense that  
$$
s_j|_{S_{P,U}[m_U^{-1}]}=\pi^U_{U_j}\circ s_U.
$$
As the $U_j$ are good enough and $\gcd(m_j)_j=1$, from the previous equation and \eqref{EqCompatible} we deduce that the line sheaves $s_j^*\omega_{U_j}$ (each on $S_{P,U}[m_j^{-1}]$, respectively) glue together, defining a line sheaf on $S_{P,U}$ which we denote by $\omega_{P,U}$. The metrics at infinity induce a metric on $\omega_{P,U}$ for each $\sigma: F_{P,U}\to \C$.  Thus, we obtain a metrized projective $O_{F_{P,U}}$-module $\widehat{\Mfrak}$ of rank $1$, namely, $\Mfrak=H^0(S_{P,U}, \omega_{P,U})$ with the induced metrics at infinity, which we denote by $\|-\|_{\Mfrak,\sigma}$ for each embedding $\sigma:F_{P,U}\to \C$. This allows us to define $h_{Ar}(P)$ as the (normalized) Arakelov degree of  $\widehat{\Mfrak}$:
\begin{equation}\label{EqDefHeight}
h_{Ar}(P)=\frac{\widehat{\deg}_{F_{P,U}} \widehat{\Mfrak}}{[F_{P,U}:\Q]}
\end{equation}
or more explicitly
\begin{equation}\label{EqExplicitArakelovHeight}
h_{Ar}(P)=\frac{1}{[F_{P,U}:\Q]}\left(\log\#\left(\Mfrak/\langle \eta\rangle\right)  - \sum_{\sigma : F_{P,U}\to \C} \log \|\eta\|_{\Mfrak,\sigma} \right)
\end{equation}
for any non-zero $\eta\in \Mfrak$.

Note that we can always choose the required $U_j$, for instance, for any $r\ge 2$, we can take $U_j=U_1^D(\ell_j)$ with $\ell_j\ge 5$ distinct primes not dividing $D$. Furthermore, the number $h_{Ar}(P)$ only depends on $P$; it is independent of the choice of  $\{U_j\}_j$, the choice of field $F_{P,U}$, and the choice of non-zero $ \eta\in \Mfrak$.

For Heegner points, one has the following explicit formula:
\begin{theorem}\label{ThmHeightQ} Let $K$ be a quadratic imaginary field satisfying the Heegner hypothesis for $D$. Let $P_K$ be the associated Heegner point in $X=\varprojlim X_U$. Let $d_K$  be  the absolute value of the discriminant of $K$, and let $\chi_K$ be the non-trivial primitive Dirichlet character attached to $K$. Then
\begin{equation}\label{EqHeightQ}
h_{Ar}(P_K)=-\frac{L'}{L}(0,\chi_K) + \frac{1}{2}\log \left(d_K^{-1} D\right).
\end{equation}
\end{theorem}
Here, $L(s,\chi)=\sum_{n\ge 1} \chi(n)n^{-s}$ (for $\Re(s)>1$).  In the next two paragraphs, we explain how this height formula follows from the existing literature in the case $D>1$. The case $D=1$ will not be used in our work, but we remark that the result is still correct in that case, and it can be deduced directly from the Chowla-Selberg formula.


\subsection{The Chowla-Selberg formula, after Gross, Colmez, Kudla-Rapoport-Yang} By work of Gross \cite{GrossCM} and Colmez \cite{Colmez}, the classical Chowla-Selberg formula can be understood as a formula for the semi-stable Falting's height of a CM elliptic curve. Kudla, Rapoport and Yang \cite{KRY, KRYbook} used this fact to give a height formula for Heegner points on Shimura curves with respect to a suitably defined metrized canonical bundle. Keeping track of normalizations, Theorem \ref{ThmHeightQ} with $D>1$ can be deduced from the results of Kudla-Rapoport-Yang  taking into account that the quantity $c$ in Section 10 of \cite{KRY} in our case is $c=1$, because we only consider Heegner points subject to the Heegner hypothesis for $D$.

For the sake of exposition, let us briefly recall the method of proof in \cite{KRY}. They work on $\Xfrak^D_0(1)$, the moduli stack over $\Z$ associated to $X_0^D(1)$. First they show that the metrized canonical sheaf on $\Xfrak^D_0(1)$  (the relative dualizing sheaf with metrics coming from the complex uniformization) can be identified, up to an explicit factor in the metrics at infinity, with the metrized Hodge bundle coming from the universal family of fake elliptic curves on $\Xfrak^D_0(1)$ (cf. Section 3 \cite{KRY}). The Arakelov height of an algebraic point $P$ on $\Xfrak^D_0(1)$ relative to the metrized Hodge bundle coincides with the Faltings height of the fake elliptic curve $A_P$ associated to $P$, thus, the same holds (up to an explicit factor) for the height relative to the metrized canonical sheaf. On the other hand, when $P$ is a CM point,  $A_P$  is isogenous to $E_P^2$ for certain CM elliptic curve $E_P$, thus the Faltings height of $A_P$ equals $2h(E_P)$ up to a factor coming from the isogeny, which is made explicit in Theorem 10.7 \cite{KRY}. Finally, $h(E_P)$ is expressed in terms of the logarithmic derivative of $L(s,\chi)$ by the classical Chowla-Selberg formula, as mentioned before.

The translation to our setting is possible because, upon adding level structure of a good enough $U$, the corresponding stack is in fact our scheme $\Xcal_U$ over $\Z[m_U^{-1}]$, and the pull-back of the metrized canonical sheaf on $\Xfrak^D_0(1)$ coincides with our $\widehat{\omega}_U$ up to suitable normalization factors on the metrics.


\subsection{The Yuan-Zhang height formula} Another (more direct) way to deduce Theorem \ref{ThmHeightQ} from the existing literature is using a recent general result of Yuan and Zhang, namely, Theorem 1.5 in \cite{YuZh}, along with their theory of integral models. The Yuan-Zhang theorem works for quaternionic Shimura curves over totally real fields in general, not just over $\Q$.

 Using \cite{YuZh}, the translation to our setting with $D>1$ is almost immediate:

The finite parts of our metrized sheaves $\widehat{\omega}_U$ agree with the arithmetic Hodge bundle as defined in \cite{YuZh}  when $U$ is good enough, although in \emph{loc. cit.} the definition is given in more generality (cf. Theorem 4.7 (2) \emph{loc.cit.} and equation \eqref{EqCompatible} above). The metrics at infinity in \emph{loc.cit.} also agree (cf. Equation \eqref{EqYZMetric} above and Theorem 4.7 (3) \emph{loc.cit.}).

The integral model for $X_U$ over $\Z[m_U^{-1}]$ used in \cite{YuZh}, for $U$ small enough, is the minimal regular model. It is unique as $X_U$ has genus $g_U\ge 2$. It can be obtained from our $\Xcal_U$ by repeated blow-up and contraction \emph{away from the smooth locus $\Xcal_U^0$}. By Lemma \ref{LemmaRed} above, if $K$ satisfies the Heegner hypothesis for $D$ (in particular, the discriminant of $K$ is coprime to $D$ as required in \cite{YuZh}) the closure of the associated Heegner point in $\Xcal_0^D(1)$ is contained in $\Xcal_0^D(1)^0$, hence, the closure of $P_{K,U}$ in $\Xcal_U$ is contained in $\Xcal_U^0$. Therefore, this difference on integral models does not affect the height of Heegner points.

At this point, we mention that later, in Section \ref{SecModApproachF}, we will directly use the Yuan-Zhang theory of integral models and height formula over totally real fields. 


\subsection{Functional equation} Given an imaginary quadratic field $K$, the primitive quadratic character $\chi_K$ is odd, has conductor $d_K$, and the functional equation for $L(s,\chi_K)$ is given by $\xi(1-s,\chi_K)=\xi(s,\chi_K)$ where
$$
\xi(s,\chi_K)=\left(\frac{d_K}{\pi}\right)^{(s+1)/2}\Gamma\left(\frac{s+1}{2}\right)L(s,\chi_K).
$$
Thus, 
$$
-\frac{L'}{L}(1-s,\chi_K)= \log\left(d_K/\pi \right) + \frac{1}{2}\left(\frac{\Gamma'}{\Gamma}\left(\frac{s+1}{2}\right)+\frac{\Gamma'}{\Gamma}\left(\frac{2-s}{2}\right)\right) + \frac{L'}{L}(s,\chi_K)
$$
and we see that an equivalent way to state formula \eqref{EqHeightQ} is
\begin{equation}\label{EqHeightQat1}
h_{Ar}(P_K)=\frac{L'}{L}(1,\chi_K) + \frac{1}{2}\log \left(d_K D\right) - (\gamma + \log(2\pi))
\end{equation}
where $\gamma$ is the Euler-Mascheroni constant. This seems more natural from the point of view of analytic number theory because $\frac{L'}{L}(1,\chi_K)$ is expected to be small as $K$ varies. For instance, from \cite{Ihara} one  deduces the following:
\begin{proposition}\label{PropLogDerGRH} As $K$ varies over quadratic imaginary fields satisfying the Heegner hypothesis for $D$, if GRH holds for $L(s,\chi_K)$ then we have
$$
h_{Ar}(P_K)= \frac{1}{2}\log(d_K D) + O(\log \log d_K).
$$
The implicit constant in the error term can be taken as $2+\epsilon$ for any $\epsilon>0$ (for $d_K\gg_\epsilon 1$).
\end{proposition}
See also \cite{IharaMurtyShimura} for further results on the size of $\frac{L'}{L}(1,\chi)$, and see \cite{ColmezAppl} for some related applications of analytic number theory to estimate  the height of CM abelian varieties. In \emph{loc. cit.}, however, analytic lower bounds for the height are found, while we need analytic upper bounds, which seems to be a more difficult problem from the point of view of $L$-functions.

We will not use Proposition \ref{PropLogDerGRH}. For our purposes, we will have some freedom to choose the quadratic imaginary field $K$, and an unconditional estimate for $h_{Ar}(P_K)$ of essentially the same strength can be deduced  from Theorem \ref{ThmCountingFields}.

\section{Integrality and lower bounds for the $L^2$-norm}\label{SecLowerBounds}

\subsection{Notation and result} We are now ready to present a key application of our work in the previous sections. As always, $D$ and $M$ stand for coprime positive integers, with $D$ squarefree with an even number of prime factors. Also, we write $N=DM$. In addition, we will assume $D>1$; in fact, for the results in this section the case $D=1$ is already known by using methods based on $q$-expansions, which are not available in the present case $D>1$.

We will be working with open compact subgroups of the form $U=U_0^D(m)\cap U_1^D(m')$ with $m,m'$ coprime, both coprime to $D$, in  which case $C(U)=(1)$ so that $X_U^{an}$ is connected. Thus, we can choose $g_a=1$ in the decomposition \eqref{EqDecomposition} and write $\tilde{\Gamma}_{U}=\tilde{\Gamma}_{U,g_a}$. Similarly, with this choice we may drop the subscript $g$ in the notation of the $L^2$ and supremum norms of Section \ref{SecNorms}. 

For $U$ as above, we write $\xi_U:\hfrak\to \tilde{\Gamma}_{U}\backslash \hfrak=X_U^{an}$ for the complex uniformization. We have an injective map into the space $S_U$ of  holomorphic of weight $2$ modular forms for $U$ 
$$
\Psi_{U}:  H^0(X_U,\Omega^1)\to S_U
$$ 
given by the condition that the image of a section $\alpha$ is  the modular form $\Psi_U(\alpha)\in S_U$ satisfying that $\xi_U^\bullet\alpha=\Psi_U(\alpha)dz$ with $z$ the complex variable in $\hfrak$ (cf. Paragraph \ref{SecNotationForms}).

Recall the normal integral model $\Xcal_0^D(M)=\Xcal_0^D(M,U^D_1(1))$, flat and projective over $\Z$,  introduced in Paragraph \ref{SecFinPart1}. We let 
$$
\Scal_2^D(M)=\Psi_{U_0^D(M)}(H^0(\Xcal_0^D(M)^0,\Omega^1_{\Xcal_0^D(M)/\Z}))\subseteq S_2^D(M).
$$ 
Thus, $\Scal_2^D(M)$ defines a notion of $\Z$-integrality in $S_2^D(M)$. 

We note that when $D=1$ (which we are not considering here), every element of $\Scal_2^D(M)$ has Fourier expansion (at $i\infty$) with Fourier coefficients in $2\pi i \Z$, but the converse usually fails, see \cite{EdixhovenIntegrality} for details.

The main result in this section is:

\begin{theorem}[Integral forms are not too small]\label{ThmIntegralityBound}  Given $\epsilon>0$, if $N\gg_\epsilon 1$ and if $N=DM$ is an admissible factorization with $D>1$, then for every $f\in \Scal_2^D(M)$ integral non-zero modular form for $U_0^D(M)$ we have
$$
-\log \|f\|_{U_0^D(M),2}\le \left(\frac{5}{6}+\epsilon \right) \log N + \frac{1}{2}\log M.
$$
\end{theorem}
We remark that in this result either $D$ or $M$ can be small (or even remain bounded) as long as the product $N=DM$ is sufficiently large in terms of $\epsilon$.  


\subsection{Lower bound for the height of Heegner points} 
\begin{proposition} \label{PropLiouville} Let $K$ be a quadratic imaginary field satisfying the Heegner hypothesis for $(D,M)$ and let $P_K\in X(K^{ab})$ be a Heegner point associated to $K$. 
Let $H$ be a finite extension of $K$ such that $P_{K,U_0^D(M)}$ is $H$-rational, and for each $\sigma: H\to \C$ choose $\tau_{K,\sigma}\in\hfrak$  such that $\xi_{U_0^D(M)}(\tau_{K,\sigma})=P_{K,U_0^D(M)}^\sigma$.  

Let $\alpha \in H^0(\Xcal_0^D(M)^0,\Omega^1_{\Xcal_0^D(M)/\Z})$ be a non-zero section, write $f=\Psi_{U_0^D(M)}(\alpha)$ and suppose that $f$ does not vanish at the points $\tau_{K,\sigma}$. We have
$$
h_{Ar} (P_K)\ge - \log(2M) - \frac{1}{[H:\Q]}\sum_{\sigma:H\to \C} \log \left(|f(\tau_{K,\sigma})|\Im(\tau_{K,\sigma})\right).
$$
\end{proposition}
\begin{proof} Let $\ell_1,\ell_2\ge 5$ be two distinct primes not dividing $N=DM$. Let $U_j=U_1^D(\ell_j)$ for $j=1,2$, let $U=U_1\cap U_2$, and let $H'$ be a finite extension of $H$ such that $P_{K,U_0^D(M)\cap U}$ is $H'$-rational. Note that $H_K\subseteq H'$ because $P_{K,U_0^D(M)\cap U}$ maps to $P_{K,U_0^D(M)}$, and similarly, $H'$ contains the residue field of $P_{K,U}$. In the notation of Paragraph \ref{SubSecHeight} we let $\Ucal=\{U_1,U_2\}$, $P=P_K$ and $F_{P,U}=H'$, obtaining the metrized $O_{H'}$-module $\Mfrak$ of rank $1$ with
$$
h_{Ar}(P_K) = \frac{1}{[H':\Q]}\widehat{\deg}_{H'} \Mfrak.
$$
We would like to use the section $\alpha$ to construct a non-zero element of $\Mfrak$, so that $h_{Ar}(P_K)$ can be computed as in \eqref{EqExplicitArakelovHeight}. For this we construct a second metrized $O_{H'}$-module $\Nfrak$ where $\alpha$ canonically induces an element, and such that $\Mfrak$ maps to $\Nfrak$ with controlled torsion cokernel (and respecting the metrics).

Let $S_{H'}=\Spec O_{H'}$. Let $s_U:S_{H'}[(\ell_1\ell_2)^{-1}]\to \Xcal_U$ be the morphism associated to $P_{K,U}$, and for $j=1,2$ let $s_j:S_{H'}[\ell_j^{-1}]\to \Xcal_{U_j}$ be the morphism associated to $P_{K,U_j}$. Similarly, we also have morphisms $s'_U:S_{H'}[(\ell_1\ell_2)^{-1}]\to \Xcal_0^D(M,U)$ and  $s'_j:S_{H'}[\ell_j^{-1}]\to \Xcal_0^D(M,U_j)$ induced by the points $P_{K,U_0^D(M)\cap U}$ and $P_{K,U_0^D(M)\cap U_j}$ ($j=1,2$) respectively. These are compatible in the obvious way with the forgetful maps among the six surfaces.

Let $\omega$, $\omega_j$, $\omega'$, and $\omega'_j$ be the sheaves of relative differentials for $\Xcal_U/\Z[(\ell_1\ell_2)^{-1}]$, $\Xcal_{U_j}/\Z[\ell_j^{-1}]$, $\Xcal_0^D(M,U)/\Z[(\ell_1\ell_2)^{-1}]$, and $\Xcal_0^D(M,U_j)/\Z[\ell_j^{-1}]$ (for $j=1,2$) respectively. They are not invertible in general, although they are indeed invertible sheaves on the smooth loci of the corresponding structure maps.

The sheaves ${s'_j}^*\omega'_j$ on $S_{H'}[\ell_j^{-1}]$ glue along $s'^*\omega'$ to define a  sheaf $\omega'_{P_K, U}$ on $S_{H'}$ in the same way that the line sheaves $s_j^*\omega_j$ determine $\omega_{P_K,U}$ on $S_{H'}$ (cf. Paragraph \ref{SubSecHeight}). This is because the forgetful maps $\pi^{U_0^D(M)\cap U}_{U_0^D(M)\cap U_j}: \Xcal_0^D(M,U)\to \Xcal_0^D(M,U_j)$ are \'etale (as $U_j$ is sufficiently small), which can be used instead  of \eqref{EqCompatible} in order to check compatibility. Furthermore,  $\omega_{P_K, U}$ and  $\omega'_{P_K, U}$ are invertible sheaves on $S_{H'}$, by Lemma \ref{LemmaRed}.

We define $\Nfrak=H^0(\Spec O_{H'},\omega'_{P_K,U})$, which is a projective $O_{H'}$-module of rank $1$, endowed with the metrics coming from \eqref{EqYZMetric} and the complex uniformization of $X_0^D(M,U)^{an}$.

The non-zero global section $\alpha\in H^0(\Xcal_0^D(M)^0,\Omega^1_{\Xcal_0^D(M)/\Z})$ induces (via pull-back) compatible sections on 
$$
H^0(\Xcal_0^D(M,U)^0,\Omega^1_{\Xcal_0^D(M,U)/\Z[(\ell_1\ell_2)^{-1}]})=H^0(\Xcal_0^D(M,U)^0,\omega')
$$ 
and 
$$
H^0(\Xcal_0^D(M,U_j)^0,\Omega^1_{\Xcal_0^D(M,U_j)/\Z[\ell_j^{-1}]})=H^0(\Xcal_0^D(M,U_j)^0,\omega'_j)
$$ 
($j=1,2$) which determine an element $\beta\in \Nfrak$.

Since $P_{K,U_0^D(M)}$ is $H$-rational, we have for each $\sigma:H'\to \C$ and each $\tau'_{K,\sigma}\in\hfrak$ mapping to $P_{K,U_0^D(M)\cap U}^\sigma$
\begin{equation}\label{EqMapGal}
\pi^{U_0^D(M)\cap U}_{U_0^D(M)}(\xi_{U_0^D(M)\cap U}(\tau'_{K,\sigma}))=\pi^{U_0^D(M)\cap U}_{U_0^D(M)}(P_{K,U_0^D(M)\cap U})^\sigma = P_{K,U_0^D(M)}^\sigma = \xi_{U_0^D(M)}(\tau_{K,\sigma|_{H}}).
\end{equation}

Since $f$ is in the image of $\Psi_{U_0^D(M)}$, we see that  condition that $f$ does not vanish at the points $\tau_{K,\sigma}$ for $\sigma: H\to \C$ implies that $\beta\ne 0$. 

We observe that the construction of $\Mfrak$ factors through ``adding level $U_0^D(M)$'' in the sense that 
\begin{equation}\label{EqPullBack}
s_j^*\omega_j = {s'_j}^*((\pi^{U_0^D(M)\cap U_j}_{U_j})^*\omega_j)
\end{equation}
and similarly for $s^*\omega$.  Furthermore, on the smooth loci we have exact sequences
$$
0\to (\pi^{U_0^D(M)\cap U}_{U})^*\omega|_{\Xcal_0^D(M,U)^0}\to \omega'|_{\Xcal_0^D(M,U)^0}\to \Ccal\to 0
$$
and
$$
0\to (\pi^{U_0^D(M)\cap U_j}_{U_j})^*\omega_j|_{\Xcal_0^D(M,U_j)^0}\to \omega'_j|_{\Xcal_0^D(M,U_j)^0}\to \Ccal_j\to 0
$$
with $\Ccal$, $\Ccal_j$ annihilated by the integer $M$, thanks to Corollary \ref{CoroRelDual}. 

By \eqref{EqPullBack} and  Lemma \ref{LemmaRed}, these exact sequences on smooth loci induce a map of $O_{H'}$-modules $\iota:\Mfrak\to \Nfrak$ which is still injective because the maps $\pi^{U_0^D(M)\cap U_j}_{U_j}$ and $\pi^{U_0^D(M)\cap U}_{U}$ are \'etale, and also by the non-vanishing assumption on $f$. 

Thus, the cokernel of $\iota$ is annihilated by the integer $M$. Furthermore, the map $\iota$ is locally described by pull-back, and the forgetful maps are compatible with the complex uniformization of the relevant curves, so, the map $\iota$ respects the metrics at infinity. Thus, $\iota$ induces an inclusion $\widehat{\Mfrak}\subseteq \widehat{\Nfrak}$ of metrized $O_{H'}$-modules, whose cokernel is annihilated by $M$.

So, $M\beta\in \Mfrak$ is a non-zero element, and from \eqref{EqExplicitArakelovHeight} we obtain
$$
\begin{aligned}
h_{Ar}(P_K)&= \frac{1}{[H':\Q]}\left(\log\#\left(\Mfrak/\langle M\beta \rangle\right)  - \sum_{\sigma : H'\to \C} \log \| M\beta\|_{\Mfrak,\sigma} \right)\\
&\ge \frac{-1}{[H':\Q]} \sum_{\sigma : H'\to \C} \log \|M\beta \|_{\Mfrak,\sigma}\\
&= -\log(M)-\frac{1}{[H':\Q]} \sum_{\sigma : H'\to \C} \log \|\beta \|_{\Nfrak,\sigma}.
\end{aligned}
$$
For each $\sigma: H'\to \C$ choose $\tau'_{K,\sigma}\in\hfrak$ such that $\xi_{U_0^D(M)\cap U}(\tau'_{K,\sigma})=P_{K,U_0^D(M)\cap U}^\sigma$, then we have that the last expression is
\begin{equation}\label{EqIntermediate}
-\log(M)-\frac{1}{[H':\Q]} \sum_{\sigma : H'\to \C} \log \left(2\cdot |f(\tau'_{K,\sigma})|\Im(\tau'_{K,\sigma})\right).
\end{equation}

Also from \eqref{EqMapGal} and the fact that $f$ is in the image of $\Psi_{U_0^D(M)}$, it follows that \eqref{EqIntermediate} is equal to
$$
-\log(2M) - \frac{1}{[H:\Q]} \sum_{\sigma : H\to \C} \log \left(|f(\tau_{K,\sigma})|\Im(\tau_{K,\sigma})\right)
$$
which proves the result.
\end{proof}
\subsection{$L^2$-norm of integral modular forms}
\begin{lemma}\label{LemmaZeros} Let $P$ be an algebraic point of $X_0^D(M)$ (non-cuspidal, as $D>1$), and let $P_j$ for $j=1,\ldots,r$ be the Galois conjugates of $P$, with $P_1=P$, say. Let $\tau_j\in\hfrak$ be such that $\xi_{U_0^D(M)}(\tau_j)=P_j$ for each $1\le j\le r$. Let $\alpha\in H^0(X_0^D(M),\Omega^1_{X_0^D(M)/\Q})$ and let $f=\Psi_{U_0^D(M)}(\alpha)$ be the associated modular form. Then $f$ vanishes at one of the $\tau_j$ if and only if it vanishes at all of them. 

Furthermore, given $\eta>0$, if $N\gg_\eta 1$ and $N=DM$, then $f$ has at most $N^{1+\eta}$ zeros on $\hfrak$ up to $\tilde{\Gamma}_{U_0^D(M)}$-equivalence.
\end{lemma}
\begin{proof} The zeros of $f$ on $\hfrak$ map via $\xi_{U_0^D(M)}$ to two types of points on $X_0^D(M)$: points in the support of the divisor of $\alpha$, and the elliptic points on $X_0^D(M)$. The first is Galois-stable, while the second can be seen as the branch locus of 
$$
\pi^{U_0^D(M)\cap U_1^D(\ell)}_{U_0^D(M)} : X_0^D(M,U_1^D(\ell))\to X_0^D(M)
$$
for any prime $\ell>3$ coprime to $N=DM$, and this branch locus is also Galois-stable. Hence the first claim.

For the second claim, an upper bound is given by the number of zeros of $f$ up to $\tilde{\Gamma}_{U_0^D(M)\cap U_1^D(\ell)}$-equivalence, and this number is at most the degree of the zero divisor of $\beta=(\pi^{U_0^D(M)\cap U_1^D(\ell)}_{U_0^D(M)})^*\alpha$ on $X_0^D(M,U_1^D(\ell))$ since the complex uniformization $\xi_{U_0^D(M)\cap U_1^D(\ell)}$ is unramified. The degree of the divisor of $\beta$ is $2g-2$, where $g$ the genus of $X_0^D(M,U_1^D(\ell))$. By standard dimension formulas, the genus satisfies
$$
\log g= \log N + O(\log\log\log  N) + O(\log \ell).
$$
Since it is possible to take $\ell \ll \log N$, we get the result.
\end{proof}


\begin{proof}[Proof of Theorem \ref{ThmIntegralityBound}] Fix $\epsilon>0$. Let $\theta,\delta>0$ be defined by $\delta = \min\{1/4,\epsilon\}$ and $\theta=\delta/5$. 

Take  $N\gg_{\theta,\delta} 1$ with the same implicit constant as in Corollary \ref{CoroCounting}, and take $N=DM$ an admissible factorization with $D>1$. We also require $N\gg_\delta 1$ so that the implicit constant is admissible for Lemma \ref{LemmaZeros} with $\eta=\delta/5$. After these two conditions, note that we are only requiring that $N\gg_{\epsilon} 1$.

By  Corollary \ref{CoroCounting} with $x=N^{\frac{2}{3}+\delta}>N^{\frac{1}{2}+\delta}$, we have
$$
\# S_\theta(D,M,N^{\frac{2}{3}+\delta})> N^{(\frac{2}{3}+\delta)(1-\delta)}.
$$
For each fundamental discriminant $-d$ with $d\in S_{\theta}(D,M,N^{\frac{2}{3}+\delta})$ note that $K_{-d}$ satisfies the Heegner hypothesis for $(D,M)$. Let $P_{-d}=P_{K_{-d},U_0^D(M)}$ be the associated Heegner point in $X_0^D(M)$ and note that the number of Galois conjugates of $P_{-d}$ is
$$
[H_{K_{-d}}:\Q]=2[H_{K_{-d}}:K_{-d}]>2d^{\frac{1}{2}-\theta}>x^{\frac{1}{2}-\theta}=N^{(\frac{2}{3}+\delta)(\frac{1}{2}-\theta)}
$$
because the residue field of $P_{-d}$ over $K$ is the Hilbert class field $H_{K_{-d}}$, and by item (v) in the definition of $S_\theta(D,M,x)$ in Theorem \ref{ThmCountingFields}.
Hence, the number of points in the set
$$
\{P_{-d}^\sigma : d\in S_\theta(D,M,N^{\frac{2}{3}+\delta}), \sigma: H_{K_d}\to \C\}
$$
is at least 
$$
N^{(\frac{2}{3}+\delta)(\frac{1}{2}-\theta)+(\frac{2}{3}+\delta)(1-\delta)}=N^{1 + (\frac{5}{6} -\delta)\delta - (\delta + \frac{2}{3})\theta}> N^{1 + \frac{1}{2}\delta - \theta}> N^{1 + \frac{1}{4}\delta}
$$
because $\delta<1/3$ and $\theta < \delta/4$. 

By Lemma \ref{LemmaZeros} with $\eta=\delta/5$, we deduce from the previous estimate that there is $d_0\in S_\theta(D,M,N^{\frac{2}{3}+\delta})$ such that if we let $K=K_{-d_0}$, the Heegner point $P_{K,U_0^D(M)}$ satisfies the non-vanishing condition with respect to $f$ required in Proposition \ref{PropLiouville}. Namely, letting $\tau_{K,\sigma}\in \hfrak$ be a point mapping to $P_{K,U_0^D(M)}^\sigma$ for each $\sigma:H_K\to \C$, we have that $f$ does not vanish at these points.

Thus, Proposition \ref{PropLiouville} gives
$$
\begin{aligned}
h_{Ar} (P_K)&\ge - \log(2M) - \frac{1}{[H_K:\Q]}\sum_{\sigma:H_K\to \C} \log \left(|f(\tau_{K,\sigma})|\Im(\tau_{K,\sigma})\right)\\
& \ge -\log(2M) - \log \|f\|_{U_0^D(M),\infty}.
\end{aligned}
$$

By Theorem \ref{ThmNorms} (as $D>1$) applied to $f$  and the compact open subgroup $U_0^D(M)\cap U_1^D(\ell)$ for some prime $5\le \ell\le 10 \log N$ coprime to $N$ (cf. Lemma \ref{LemmanmidGh}), we deduce that 
$$
\begin{aligned}
\log \|f\|_{U_0^D(M),\infty} & = \log \|f\|_{U_0^D(M)\cap U_1^D(\ell),\infty} \\
& \le \log \|f\|_{U_0^D(M)\cap U_1^D(\ell), 2} + O(1)\\
& = \log \|f\|_{U_0^D(M), 2} + O(\log \log N)
\end{aligned}
$$
with absolute implicit constants. Hence,
$$
-\log \|f\|_{U_0^D(M),2}\le \log M + h_{Ar}(P_K) + O(\log \log N)
$$
with an absolute implicit constant.

Recall that $K$ has discriminant $-d_0$ with $d_0\in S_{\theta}(D,M,N^{\frac{2}{3}+\delta})$. By Theorem \ref{ThmHeightQ} in its formulation \eqref{EqHeightQat1}, we obtain that for some uniform constant $\kappa$ as in Theorem \ref{ThmCountingFields} (where $S_\theta(D,M,x)$ was defined)
$$
\begin{aligned}
-\log \|f\|_{U_0^D(M),2}&\le \log M + \frac{L'}{L}(1,\chi_{K}) + \frac{1}{2}\log(d_0D) + O(\log\log N)\\
&\le \log M + \kappa \log\log d_0 +\frac{1}{2}\log(d_0D) + O(\log \log N).
\end{aligned}
$$
As $d_0\in S_{\theta}(D,M,N^{\frac{2}{3}+\delta})$ we have $d_0\le 2N^{\frac{2}{3}+ \delta}$, which gives
$$
\begin{aligned}
-\log \|f\|_{U_0^D(M),2}&\le \log M + \frac{1}{2}\log D + \left(\frac{1}{3} + \frac{\delta}{2}\right)\log N + O(\log\log N)\\
&= \frac{1}{2}\log M + \left(\frac{5}{6} + \frac{\delta}{2}\right)\log N + O(\log\log N)
\end{aligned}
$$
with an absolute implicit constant, provided that $N\gg_\epsilon 1$. Since $\delta\le \epsilon$, the result follows.
\end{proof}
\section{Linear forms in logarithms}

\subsection{An application of linear forms in logarithms}

As a preparation for the proof of Theorem \ref{ThmMainABC}, we prove the following simple consequence of the theory of linear forms in $p$-adic logarithms. 
\begin{proposition}\label{PropLFL} Let $\epsilon>0$. There is a constant $C_\epsilon>1$ such that for all triples $a,b,c$ of coprime positive integers with $a+b=c$, we have
$$
\frac{d(abc)}{(\log d(abc))^\nu}< C_\epsilon^\nu \nu^{2\nu^2} \rad(abc)^{1+\epsilon\nu}
$$
where $\nu=\omega(abc)$ is the number of distinct primer divisors of $abc$. In particular, if we consider $\nu$ as fixed, then for those triples $a,b,c$  we have
$$
d(abc)\ll_{\epsilon,\nu} \rad(abc)^{1+\epsilon}.
$$

\end{proposition}

The reader will note that we are actually aiming for a bound of the following sort for $abc$-triples:
\begin{equation}\label{EqAim}
d(abc)\ll \rad(abc)^\kappa
\end{equation}
for some fixed $\kappa$. Proposition \ref{PropLFL} falls short of proving such an estimate because it only applies for fixed (or bounded) value of $\omega(abc)$, but nevertheless, it will be useful to get an improved value of $\kappa$. Namely, Proposition \ref{PropLFL} will be used for $abc$-triples with  at most $\nu$ prime factors (for some suitable choice of $\nu$), while the theory developed in previous sections of this paper will be used on the general case of $abc$ triples with more than $\nu$ prime factors. We remark that for our purposes, the previous proposition is not really necessary and one can show a version of Theorem \ref{ThmMainABC} (i.e. \eqref{EqAim} with a slightly worse exponent $\kappa$) without it.

\begin{proof} Consider coprime positive integers $a,b,c$ with $a+b=c$. Let $\nu_a=\omega(a)$ be the number of distinct prime factors of $a$, and let $L_a=\prod_{p|a}(1+\log p)$. We make analogous definitions for $b$ and $c$.  If $p$ is a prime dividing $a$, we note that 
$$
v_p(abc)=v_p(a)=v_p\left(\frac{b}{c}-1\right).
$$
The latter quantity is well-suited for the theory of linear forms in $p$-adic logarithms. For instance, the corollary in p. 245 in \cite{YuOld}  gives
\begin{equation}\label{EqLinFormLogs}
v_p(abc) \ll (\nu_b+\nu_c)^{3(\nu_b+\nu_c)}\cdot p\cdot L_bL_c  \log(L_bL_c)\log d(bc)
\end{equation}
where we used the inequality $d(m)\ge \max_{q|m}v_q(m)+1$ ($q$ varying over primes). Multiplying the analogous bounds for all $p|abc$  we get
$$
\begin{aligned}
d(abc)&=\prod_{p|abc}(v_p(abc)+1)\\
 &\ll \nu^{6(\nu_a\nu_b+\nu_a\nu_c+\nu_b\nu_c)}\rad(abc) (L_aL_bL_c)^{2\nu}(\log d(abc))^\nu\\
& \le \nu^{2\nu^2}\rad(abc)\cdot \prod_{p|abc}(1+\log p)^{2\nu} (\log d(abc))^\nu.
\end{aligned}
$$
The result now follows from 
$$
\prod_{p|m}(1+\log p)\ll_\epsilon \rad(m)^\epsilon 
$$
which has an implicit constant that only depends on $\epsilon$. Note that we will take $\nu$-th power of the previous estimate, which explains the term $C_\epsilon^\nu$ in the final result.
\end{proof}

We also the following estimate for $v_2(abc)$, which also relies on the theory of linear forms in logarithms.

\begin{lemma}\label{Lemmav2} Let $\epsilon>0$. There is a constant $C'_\epsilon>1$ such that for all triples $a,b,c$ of coprime positive integers with $a+b=c$ we have
$$
v_2(abc)< C'_\epsilon \cdot \rad(abc)^{\epsilon}.
$$
\end{lemma}
\begin{proof} For $abc$ triples with $\omega(abc)=2$ the result is a consequence of Mihailescu's solution to Catalan's conjecture \cite{Mihailescu} (this particular case could be addressed by  linear forms in logarithms too, cf. \cite{Tijdeman}). 

So we may assume $\omega(abc)\ge 3$. Then the prime $p=2$ satisfies condition (16) in \cite{ABC3}, namely, $p=2< \exp((\log \omega(abc))^2)$. Noticing that for $abc$ triples we know \cite{ABC1}
$$
\log\log(abc) \le \log (\rad(abc)^{15}) + O(1)\ll \log \rad(abc),
$$
our claim follows from (21), (22) and (23) in \cite{ABC3}.
\end{proof}

\subsection{Heuristics on the applicability of linear forms in logarithms}\label{SecHeuristic}

The bound \eqref{EqLinFormLogs} is not the sharpest available result in the literature, and it was used because it is simple to state and enough for our purposes. Can one get \eqref{EqAim} by using instead the best available bounds on linear forms in $p$-adic logarithms? We feel skeptical about this, although strictly speaking we don't have a proof that this is not possible. Nevertheless, here is a heuristic justification:

To the best of the author's knowledge, the sharpest improvements to \eqref{EqLinFormLogs} are due to Yu \cite{YuNew}, see in particular the quantities $C_1$ and $C_2$ given in p.189 \emph{loc. cit.} For coprime positive integers $a,b,c$ with $a+b=c$, let us write $\nu_a=\omega(a)$, $\Lambda_a=\prod_{p|a}\log p$ and similarly for $b$ and $c$. Combining the best aspects of these two quantities $C_1$ and $C_2$, and optimistically ignoring a few factors (which can possibly be problematic!), the main theorem in p. 190 \cite{YuNew} points in the direction that the methods can at best give a bound of the form
\begin{equation}\label{EqMaybe}
v_p(abc)=v_p(b/c-1) <_{(?)} e^{\alpha (\nu_b+\nu_c)}\cdot \frac{p}{(\log p)^{\nu_b+\nu_c+2}}\cdot \Lambda_b\cdot \Lambda_c
\end{equation}
for $p|a$, with $\alpha>0$ some absolute constant. 

For the sake of clarity, note that the factor $e^{\alpha n}$ (as opposed to $n^{\alpha n}$) is the best aspect suggested by $C_2$, while the denominator on the right hand side of \eqref{EqMaybe} is the best aspect suggested by $C_1$, but as far as we know, these two improvements are not available simultaneously.

Nevertheless, let us examine the strength of the hypothetical bound \eqref{EqMaybe}. Multiplying as $p$ varies  and using the analogous bounds for $b$ and $c$, one would get in this optimistic scenario that
$$
\begin{aligned}
\prod_{p|abc}v_p(abc) &< e^{2\alpha(\nu_a\nu_b+\nu_a\nu_c+\nu_b\nu_c)} \frac{\rad(abc)}{\Lambda^{\nu_b+\nu_c+2}_a\Lambda^{\nu_a+\nu_c+2}_b\Lambda^{\nu_b+\nu_b+2}_c}\\
&\qquad \times (\Lambda_a\Lambda_b)^{\nu_c}(\Lambda_a\Lambda_c)^{\nu_b}(\Lambda_b\Lambda_c)^{\nu_a}\\
&= e^{2\alpha(\nu_a\nu_b+\nu_a\nu_c+\nu_b\nu_c)} \frac{\rad(abc)}{(\Lambda_a\Lambda_b\Lambda_c)^2}.
\end{aligned}
$$
This bound is not better than
$$
\prod_{p|abc}v_p(abc) < e^{2\alpha(n_an_b+n_an_c+n_bn_c)} \rad(abc)^{1/2}.
$$
When $a,b,c$ have a comparable number of prime factors (which \emph{a priori} is a possible scenario), this would only give
\begin{equation}\label{EqOptimistic}
\prod_{p|abc}v_p(abc) < e^{\beta\cdot \omega(abc)^2} \rad(abc)^{1/2}
\end{equation}
for some constant $\beta$. From here it is not clear to the author how to get \eqref{EqAim}, because 
$$
\omega(N) =\Omega\left(\frac{\log \rad N}{\log \log \rad N}\right)
$$
and in fact, it would already be a problem if for some fixed $\delta>0$ one has
$$
\omega(abc)>(\log \rad(abc))^{\frac{1}{2}+\delta}.
$$
In order to make this heuristic more precise, let us describe a \emph{hypothetical} type of $abc$ triples, whose existence would be consistent with  \eqref{EqOptimistic} and with any bound of the type
\begin{equation}\label{EqSTY}
\log c \ll \rad(abc)^\kappa
\end{equation}
with $\kappa$ fixed (as the ones obtained by Stewart, Tijdeman,  and Yu \cite{ABC1,ABC2,ABC3}), yet it would contradict any bound of the form \eqref{EqAim}. 

To simplify notation, write $\nu=\omega(abc)$, $M=abc$ and $R=\rad(abc)$. We require the following:
\begin{itemize}
\item[(i)] $\nu$ is large,
\item[(ii)] the prime factors of $M$ are among the first $2\nu$ primes,  
\item[(iii)] all the primes $p|M$ satisfy
$$
v_p(M)< \exp\left(\frac{\beta \log R}{2\log \log R}\right)
$$
with $\beta>0$ as in \eqref{EqOptimistic}, and
\item[(iv)] a positive proportion of the primes $p|M$, say at least half of them, satisfy
$$
v_p(M)> \exp\left((\log\log R)^2\right).
$$
\end{itemize}
Note that by (i) and (ii)
\begin{equation}\label{EqRn}
\frac{\log R}{\log \log R} < 2\nu.
\end{equation}
Then by (iii) and \eqref{EqRn} we have
$$
\log c < \log M\le \left(\max_{p|M}v_p(M)\right)\log R\le \exp\left(\frac{\beta \log R}{\log \log R}\right)\ll_\epsilon R^\epsilon
$$
which is consistent with \eqref{EqSTY}. Also, observe that by (iii) and \eqref{EqRn} we have
$$
\prod_{p|M}v_p(M)\le \left(\max_{p|M}v_p(M)\right)^\nu\le \exp\left(\frac{\beta \nu\log R}{2\log \log R}\right)\le e^{\beta \nu^2}
$$
which is consistent with \eqref{EqOptimistic}, even without the factor $\rad(abc)^{1/2}$. Finally, note that by (iv) and \eqref{EqRn} we have
$$
\prod_{p|M}v_p(M)\ge \exp\left(\frac{n}{2}(\log_2 R)^2\right) \ge \exp\left(\frac{1}{4}(\log R )\log_2 R\right)=R^{(\log_2 R)/4}
$$
(here, $\log_2 X = \log\log X$) which is not consistent with \eqref{EqAim}, for any value of $\kappa$.

Of course, such hypothetical $abc$-triples satisfying (i), (ii), (iii), and (iv) \emph{do not exist} because our Theorem \ref{ThmMainABC} (cf. Theorem \ref{ThmValABCIntro}) actually proves a version of \eqref{EqAim}. Nevertheless, this heuristic analysis suggests (to the author) that the approach of $p$-adic linear forms in logarithms has a limitation in this direction.  In any case, regardless of the applicability of the theory of linear forms in logarithms in the context of \eqref{EqAim}, it is worth noticing that the theory in this paper provides a completely new approach to establishing $abc$-type bounds.

\section{Bounds for products of valuations}\label{SecProdVal}

\subsection{A general estimate for elliptic curves}
\begin{theorem}\label{ThmGeneralVal1} Let $S$ be a finite set of primes and let $\epsilon>0$. For all but finitely many elliptic curves $E/\Q$ semi-stable away from $S$, the following holds:

Let $N=N_E$ be the conductor of $E$ and let $\Delta_E$ be the minimal discriminant of $E$. Consider an admissible factorization $N=DM$ (in particular, $D$ is supported away from $S$). Then
$$
\prod_{p|D} v_p(\Delta_E) < N^{\frac{11}{3} + \epsilon}.
$$
\end{theorem}
\begin{proof} We may assume $D>1$. From Theorem \ref{ThmRT}, Item (a), we have
\begin{equation}\label{EqUsingRT}
\log \left(\prod_{p|D} v_p(\Delta_E)\right) \le \log \delta_{1,N}-\log \delta_{D,M} + \log D + O_S\left(\frac{\log D}{\log \log D}\right).
\end{equation}
Here, we recall that $\delta_{D,M}$ is the modular degree of the optimal quotient $q_{D,M}:J_0^D(M)\to A_{D,M}$ with $A_{D,M}$ an elliptic curve isogenous to $E$ over $\Q$. 

Using the cusp $i\infty$ we have an embedding $j_{N}:X_0(N)\to J_0(N)$ so that the composite map $\phi_N:X_0(N)\to A_{1,N}$ is a classical optimal modular parameterization and has degree exactly $\delta_{1,N}$. By Frey's formula \eqref{EqFrey} we get
\begin{equation}\label{EqFreyClassical}
\log \delta_{1,N} = 2\log(2\pi |c_f|) +2\log( \|f\|_{U_0^1(N),2}) +2h(A_{1,N})
\end{equation}
where $c_f$ is the (positive) Manin constant of the  modular parameterization $\phi_N$ and $f\in S_2(N)$ is the normalized cuspidal Hecke newform associated to $E$.

From Corollary \ref{CoroDegApproxQ}, we recall that there is a non-constant morphism  $\phi_{D,M} : X_0^D(M)\to A_{D,M}$  degree satisfying 
\begin{equation}\label{Eqdnfje}
\delta_{D,M}\le \deg \phi_{D,M}\le (9\log N)^2\delta_{D,M}.
\end{equation} 
Let $\Acal_{D,M}$ be the N\'eron model of $A_{D,M}$ over $\Z$, then by the N\'eron mapping property $\phi_{D,M}$ extends to a $\Z$-morphism of integral models on the smooth locus
$$
\phi_{D,M}: \Xcal_0^D(M)^0\to \Acal_{D,M}.
$$
Let $\omega_{A_{D,M}}$ be a N\'eron differential of $A_{D,M}$ (unique up to sign) and let 
$$
\alpha_{D,M}=\phi_{D,M}^\bullet \omega_{A_{D,M}}\in H^0(\Xcal_0^D(M)^0,\Omega^1_{\Xcal_0^D(M)/\Z})
$$ 
i.e. the image of $\phi_{D,M}^*\omega_{A_{D,M}}$ under $\phi_{D,M}^*\Omega^1_{\Acal_{D,M}/\Z}\to \Omega^1_{\Xcal_0^D(M)/\Z}$. Then 
$$
f_{D,M}=\Psi_{U_0^D(M)}(\alpha_{D,M})\in \Scal_2^D(M)
$$ 
is integral and by the same argument as the proof of \eqref{EqFrey} we get
\begin{equation}\label{EqFreyQuaternionic}
 \log\deg  \phi_{D,M} = 2\log( \|f_{D,M}\|_{U_0^D(M),2}) +2h(A_{D,M}).
\end{equation}
By \eqref{EqUsingRT}, \eqref{EqFreyClassical}, \eqref{Eqdnfje}, and \eqref{EqFreyQuaternionic} we deduce
$$
\begin{aligned}
\log \left(\prod_{p|D} v_p(\Delta_E)\right) & = 2\log |c_f| + 2\log( \|f\|_{U_0^1(N),2})  - 2\log( \|f_{D,M}\|_{U_0^D(M),2})\\
&\quad  +2h(A_{1,N})- 2h(A_{D,M}) + O(\log \log N) \\
&\quad  +\log D + O_S\left(\frac{\log D}{\log \log D}\right).
\end{aligned}
$$
Since $A_{1,N}$ and $A_{D,M}$ are both isogenous to $E$, they are connected by an isogeny of degree $\le 163$, so that $2|h(A_{1,N})- h(A_{D,M})|\le \log 163$. It follows that
\begin{equation}\label{EqRTF}
\begin{aligned}
\log \left(\prod_{p|D} v_p(\Delta_E)\right) &= 2\log |c_f|+2\log( \|f\|_{U_0^1(N),2})  - 2\log( \|f_{D,M}\|_{U_0^D(M),2}) \\
&\quad +\log D + O_S\left(\log \log N+ \frac{\log D}{\log \log D}\right).
\end{aligned}
\end{equation}
Since $f$ is a normalized newform for $\Gamma_0(N)$, from \cite{MaiMurty, MurtyBounds} we get $\|f\|_{U_0^1(N),2}^2\ll N\log N$ which gives
$$
2\log( \|f\|_{U_0^1(N),2}) \le \log N + O(\log\log N).
$$
On the other hand, since $f_{D,M}\in \Scal_2^D(M)$ is integral and non-zero, by Theorem \ref{ThmIntegralityBound} we obtain
$$
-2\log( \|f_{D,M}\|_{U_0^D(M),2})\le \left(\frac{5}{3} + \frac{\epsilon}{2} \right)\log N + \log M
$$
provided that $N\gg_\epsilon 1$. Therefore, for $N\gg_{\epsilon, S} 1$, we obtain
\begin{equation}\label{EqEndPf}
\log \left(\prod_{p|D} v_p(\Delta_E)\right)\le 2\log |c_f|+\left(\frac{8}{3} + \epsilon \right)\log N + \log M + \log D.
\end{equation}
Finally, by Corollary \ref{CoroManinCt} we see that $\log |c_f|\ll_S 1$, hence the result.
\end{proof}

Let us observe that if one only works in the semi-stable case (as in Theorem \ref{ThmProdSS} below) then we don't need Corollary \ref{CoroManinCt} in the previous proof. The existing literature on the Manin constant in the semi-stable case would suffice.

We also remark that a computation in Section 2.2.1 of \cite{Prasanna} also combines the classical Ribet-Takahashi formula with Frey's equation \eqref{EqFrey}, in a way somewhat similar to what we did to derive equation \eqref{EqRTF} in the previous argument. However, the computation in \cite{Prasanna} occurs in a different context ($p$-integrality of ratios of Petersson norms) and in a less precise form, omitting the contribution of Eisenstein primes and requiring semi-stability. Of course, for our purposes it is crucial to have control on \emph{each} prime, since we aim to prove global estimates. Frey-Hellegouarch curves always have $p=2$ as an Eisenstein prime, so we cannot ignore this issue.

For later reference, we record here a simple consequence.
\begin{corollary}\label{CoroValGen} Let $S$ be a finite set of primes and let $\epsilon>0$. For all but finitely many elliptic curves $E/\Q$ semi-stable away from $S$ and having at least two primes of multiplicative reduction, we have
$$
\prod_{p|N_E^*}v_p(\Delta_E)< N_E^{\frac{11}{2} +\epsilon}
$$
where $N_E^*$ is the product of all the primes of multiplicative reduction of $E$.
\end{corollary}
\begin{proof}
When $E$ has an even number of primes of multiplicative reduction the result follows from Theorem \ref{ThmGeneralVal1} with $D=N_E^*$.

When $E$ has an odd number $n$ of primes of multiplicative reduction, necessarily $n\ge 3$ by our assumptions. Call these primes $p_1,\ldots,p_n$, then
$$
\prod_{p|N_E^*} v_p(\Delta_E) = \left(\prod_{i=1}^n \prod_{p|(N_E^*/p_i)} v_p(\Delta_E)\right)^{1/(n-1)}.
$$
The result follows from Theorem \ref{ThmGeneralVal1} for the various $D=N_E^*/p_i$, since $\frac{11}{3}\cdot \frac{n}{n-1}\le 11/2$.
\end{proof}

In particular, we obtain the following application.

\begin{corollary} \label{CoroTamSa1} Let $S$ be a finite set of primes and let $\epsilon>0$. For all but finitely many elliptic curves $E/\Q$ semi-stable away from $S$ and having at least two primes of multiplicative reduction, we have
$$
\mathrm{Tam}(E)< N_E^{\frac{11}{2} +\epsilon}.
$$
\end{corollary}
\begin{proof}
This follows from Corollary \ref{CoroValGen}. In fact, given an elliptic curve $E$ over $\Q$, if $E$ has additive or non-split multiplicative reduction at a prime $p$ then $\mathrm{Tam}_p(E)\le 4$, so
$$
\mathrm{Tam}(E)\le 4^{\omega(N_E)} \prod_{p|N_E^*} v_p(\Delta_E)\ll_\epsilon N_E^{\epsilon} \prod_{p|N_E^*} v_p(\Delta_E).
$$
\end{proof}


\subsection{A sharpening of the general estimate} For the cases on which we are primarily interested, the estimate in Theorem \ref{ThmGeneralVal1} can be improved.

\begin{theorem} \label{ThmValSharper} Let  $\epsilon>0$. Let $E$ be an elliptic curve over $\Q$ of conductor  $N\gg_\epsilon 1$. Let $N=DM$ be an admissible factorization and suppose that either
\begin{itemize} 
\item[(i)] $E$ is semi-stable and $M$ is not a prime number, or
\item[(ii)] $E$ is a Frey-Hellegouarch elliptic curve and $M$ is divisible by at least two odd primes.
\end{itemize}
Then we have
$$
\prod_{p|D} v_p(\Delta) < N^{\frac{8}{3}+\epsilon} M.
$$
\end{theorem}
\begin{proof} The proof is essentially the same as for Theorem \ref{ThmGeneralVal1} with the only difference that we use Item (b) instead of Item (a) from Theorem \ref{ThmRT} (or alternatively, we can directly use Corollary \ref{CoroRT}). This allows us to replace \eqref{EqUsingRT} by 
$$
\log \prod_{p|D} v_p(\Delta)\le \log \delta_{1,N} - \log \delta_{D,M} + O\left(\frac{\log D}{\log \log D}\right)
$$
where the implicit constant is absolute. The rest of the proof continues in the same way, and at the end we can replace \eqref{EqEndPf} by
$$
\log \left(\prod_{p|D} v_p(\Delta_E)\right)< 2\log |c_f|+\left(\frac{8}{3} + \epsilon \right)\log N + \log M
$$
provided that $N\gg_\epsilon 1$. The necessary bound for the Manin constant is classical in the semi-stable case, and follows from Corollary \ref{CoroManinCt} in the case of Frey-Hellegouarch elliptic curves.
\end{proof}

\subsection{The semi-stable case}
\begin{theorem}\label{ThmProdSS} Let $\epsilon>0$. There is a number $K_\epsilon>0$ depending only on $\epsilon$ such that the following holds:

For every semi-stable elliptic curve $E$ over $\Q$ we have
$$
\prod_{p|N_E} v_p(\Delta_E) < K_\epsilon\cdot N_E^{\frac{11}{2}+\epsilon}.
$$
If moreover $\epsilon>0$ and $E$ has at least $3+11/\epsilon$ places of bad reduction, then we have the stronger estimate
$$
\prod_{p|N_E} v_p(\Delta_E) < K_\epsilon\cdot N_E^{\frac{8}{3}+\epsilon}.
$$
Furthermore, similar estimates hold for $\mathrm{Tam}(E)$ instead of $\prod_{p|N_E} v_p(\Delta_E)$.
\end{theorem}
\begin{proof}
If $N_E=p$ is prime then  $v_p(\Delta_E)\le 5$ (cf. \cite{MestreOesterle}). The first part of the result now follows from Corollary \ref{CoroValGen} with $S=\emptyset$.

For the second part, write $N=p_1\cdots p_n$ with $p_j$ different primes and note that $n\ge 4$. If $n$ is even then the bound follows from case (i) of Theorem \ref{ThmValSharper} with $D=N$ and $M=1$. If $n$ is odd (thus, $n\ge 5$) we apply the same result with $M=(p_{j}p_{j+1}p_{j+2})$ and $D=N/M$ for each $j=1,\ldots, n$ (taking the indices of the $p_j$ modulo $n$).  After multiplying the resulting estimates and taking $(n-3)$-rd roots we get 
$$
\prod_{p|N} v_p(\Delta_E) < N^{\left(\frac{8}{3}+\epsilon'\right)\frac{n}{n-3}}\cdot N^{\frac{3}{n-3}} 
$$
for any given $\epsilon'>0$ and for  $N\gg_{\epsilon'} 1$. For $\epsilon >0$ given, there is a sufficiently small $\epsilon'>0$ (only depending on $\epsilon$) such that for all integers $n>11/\epsilon +3$ we have 
$$
\left(\frac{8}{3}+\epsilon'\right)\frac{n}{n-3} + \frac{3}{n-3} < \frac{8}{3}+\epsilon,
$$
hence the result. The final claim about $\mathrm{Tam}(E)$ follows.
\end{proof}

An immediate consequence is the following upper bound for the number of level-lowering primes (in the sense of Ribet's theory \cite{RibetInv100}) that a semi-stable elliptic curve can have. For this, we recall from the introduction that for an elliptic curve $E$ over $\Q$ we write 
$$
L(E):=\{\ell \mbox{ prime }: \exists p\mbox{ prime such that }p|N_E\mbox{ and }\ell | v_p(\Delta_E)\}.
$$
\begin{corollary}[Counting level-lowering primes] \label{CoroLL} Let $\epsilon>0$. 
Then for all semi-stable elliptic curves $E$ over $\Q$ we have
$$
\# L(E) < \frac{(11/2 +\epsilon)\log N_E}{\log \log N_E} +O_\epsilon(1).
$$
\end{corollary}
\begin{proof} Let $\Lambda(E)$ be the product of the primes in $L(E)$. Then $\Lambda(E)$ divides  $\prod_{p|N_E}v_p(\Delta_E)$ and the previous theorem gives
$$
\Lambda(E)\ll_\epsilon N_E^{\frac{11}{2}+\epsilon}.
$$
Since $\log \Lambda(E)$ is bigger than or equal to the sum of $\log p$ where $p$ runs over  the first $\# L(E)$ prime numbers, the result now follows from the prime number theorem.
\end{proof}

\subsection{Products of valuations of $abc$ triples}
\begin{theorem} \label{ThmMainABC} Given $\epsilon>0$, there is a constant $K_\epsilon>0$ such that the following holds:

For all coprime positive integers $a,b,c$ with $a+b=c$, we have
$$
d(abc)< K_\epsilon \cdot  \rad(abc)^{\frac{8}{3}+\epsilon}.
$$
In particular, we have
$$
\prod_{p|abc}v_p(abc) < K_\epsilon \cdot  \rad(abc)^{\frac{8}{3}+\epsilon}.
$$
\end{theorem}

\begin{proof} Let $a,b,c$ be as in the statement, and recall that for a positive integer $n$ the number of positive divisors of $n$ is $d(n)=\prod_{p|n}(v_p(n)+1)$. 

Let $\nu\ge 4$ be arbitrary (we will later take $\nu$ large in terms of $\epsilon>0$). If $\omega(abc)\le \nu$ we obtain $d(abc)\ll_\nu \rad(abc)^2$ (say) by Proposition \ref{PropLFL}. So we can assume that $n:=\omega(abc)> \nu$.

Let $E$ be the  Frey-Hellegouarch curve with affine equation $y^2=x(x-a)(x+b)$ associated to the triple $a,b,c$. Let $\Delta$ be its minimal discriminant and $N$ its conductor. We note that for $p>2$, the conditions $p|abc$ and $p|N$ are equivalent, and moreover for such a prime $p$ we have
$$
v_p(N)=1 \mbox{ and } v_p(\Delta)=2v_p(abc).
$$
For $p=2$, however, we always have $2|abc$, while $2$ does not need to divide $N$, and when it does, its exponent can be larger than $1$. Nevertheless, we still have $N\asymp \rad(abc)$, and by Lemma \ref{Lemmav2} we know
\begin{equation}\label{Eqat2}
v_2(abc)\ll_\epsilon \rad(abc)^\epsilon. 
\end{equation}

Write $N=2^rp_1\cdot p_m$ where $0\le r\le 8$, $p_i>2$ are distinct primes, and $m=n$ if $r=0$, or $m=n-1$ if $1\le r\le 8$. In either case, $m\ge \nu \ge 4$.

If $m$ is even, we let $D=p_1\cdots p_m$ and $M=2^r$. With these choices, Theorem \ref{ThmValSharper} gives 
$$
\prod_{\substack{p|abc \\ p\ne 2}}(v_{p}(abc)+1)\le \prod_{j=1}^m(2v_{p_j}(abc))= \prod_{j=1}^m v_{p_j}(\Delta)\ll_\epsilon N^{\frac{8}{3}+\epsilon}.
$$
This, together with \eqref{Eqat2} gives
$$
d(abc)\ll_\epsilon N^{\frac{8}{3}+\epsilon}
$$
proving the result in the case that $m$ is even.

If $m$ is odd then $m\ge 5$. For any choice of $j_1,j_2,j_3\in\{1,\ldots, m\}$ of three distinct indices, we let $D=p_1\cdots p_m/(p_{j_1} p_{j_2} p_{j_3})$ and $M=2^rp_{j_1}p_{j_2}p_{j_3}$. Similarly we obtain
$$
\prod_{\substack{1\le j\le m \\ j\ne j_1, j_2,j_3}} (2v_{p_j}(abc))\ll_\epsilon N^{\frac{8}{3}+\epsilon}p_{j_1}p_{j_2}p_{j_3}.
$$
 Let us take the following cyclic choices: $j_1$ arbitrary, $j_2\equiv j_1+1\mod m$, and $j_3\equiv j_1+2\mod m$. We multiply these $m$ inequalities and then we take $(m-3)$-rd roots to obtain
$$
\prod_{\substack{p|abc \\ p\ne 2}}(v_{p}(abc)+1)\le \prod_{j=1}^m(2v_{p_j}(abc))\ll_\epsilon N^{(\frac{8}{3}+\epsilon)\frac{m}{m-3}}N^{\frac{3}{m-3}}\le N^{(\frac{8}{3}+\epsilon)\frac{\nu}{\nu-3} + \frac{3}{\nu-3}}.
$$
The result now follows from \eqref{Eqat2} and the fact that  $\nu$ can be taken in advance as large as needed.
\end{proof}

\section{Counting quadratic extensions of a totally real number field}

In this section we give suitable analogues of the results in Section \ref{SecCountingFields} in the setting of totally real number fields. 

In this section $n$ will denote the degree of the number field under consideration,  and we will write $c_1(n),c_2(n),...$ for certain \emph{strictly positive} quantities that only depend on the integer $n$; some $c_i(n)$ will need to be large, while others will be needed small. 

For a number field $F$ and a non-zero integral ideal $\afrak$ of $O_F$, the absolute norm of $\afrak$ is denoted by $\Norm \afrak=[O_F:\afrak]$.

\subsection{Auxiliary quadratic extensions}

\begin{lemma}\label{LemmaOneField} 
Let $F$ be a totally real number field with $[F:\Q]=n$ and discriminant $d_F$. Let $\Ifrak$ and $\Sfrak$ be  non-zero coprime ideals of $O_F$. There is a totally real quadratic extension $L/F$ satisfying that each prime $\pfrak$ dividing $\Ifrak$ is inert in $O_L$, each prime $\pfrak$ dividing $\Sfrak$ is split in $L$, and with
$$
\Norm(\Disc(L/F))< c_1(n)\cdot (d_F\Norm(\Ifrak\Sfrak))^{c_2(n)}.
$$
\end{lemma}
\begin{proof} The construction of $L$ without the condition on the size of $\Norm(\Disc(L/K))$ is classical: One chooses a monic quadratic polynomial $f(x)= x^2+\alpha x+ \beta \in O_F[x]$ which is irreducible modulo $\pfrak$ for each $\pfrak|\Ifrak$, is the product of two distinct linear factors modulo $\pfrak$ for each $\pfrak|\Sfrak$, and with positive discriminant under each real embedding. (These are simply congruence conditions on the coefficients $\alpha,\beta$.) Then any root of $f(x)$ generates a field $L$ over $F$ with the desired properties. 

Using geometry of numbers (or more elementary arguments) one can control the size under each real embedding of the coefficients $\alpha,\beta$ of $f$ in the previous construction. The discriminant of $f(x)$ is divisible by the relative discriminant $L/F$, and the result follows.
 \end{proof}
Given a number field $F$ and quadratic extensions $M_1,M_2$ of $F$, we define the quadratic extension $M_3=M_1*M_2$ of $F$ as follows: If $M_1=M_2$ we let $M_3=M_1$. Otherwise, the compositum $M_1M_2$ has degree $4$ over $F$ and contains three quadratic extensions of $F$, namely, $M_1, M_2$ and the third one is defined to be $M_3=M_1*M_2$.
\begin{lemma}\label{LemmaMixing} Let $F$ be a totally real number field, let $\Ifrak$ and $\Sfrak$ be non-zero coprime ideals in $O_F$ and let $L/F$ be a totally real quadratic extension such that every prime dividing $\Ifrak$ is inert in $O_L$ and every prime dividing $\Sfrak$ splits in $O_L$. Let $K/\Q$ be an imaginary quadratic extension such that every rational prime below $\Ifrak\Sfrak$ splits in $K$. Then the field $M=L*(FK)$ satisfies the following:
\begin{itemize}
\item $M/F$ is quadratic
\item $M$ is totally imaginary
\item every prime of $O_F$ dividing $\Ifrak$ is inert in $O_M$
\item every prime of $O_F$ dividing $\Sfrak$ splits in $O_M$.
\end{itemize}
\end{lemma}
\begin{proof} This follows by observing that any two of $L, FK, M$ generate the compositum $LK$.
\end{proof}


\subsection{The counting result}

\begin{theorem} \label{ThmCountingF}
Let $F$ be a totally real number field with $[F:\Q]=n$ and discriminant $d_F$. Let $\Ifrak$ and $\Sfrak$ be  non-zero coprime ideals of $O_F$ and let $x> c_3(n)\cdot (d_F\Norm(\Ifrak\Sfrak))^{c_4(n)}$.  The number of (pairwise non-isomorphic) quadratic extensions $M/F$ satisfying
\begin{itemize} 
\item[(i)] $M$ is totally imaginary,
\item[(ii)]  each prime dividing $\Ifrak$ is inert in $O_M$,
\item[(iii)] each prime dividing $\Sfrak$ splits in $O_M$,
\item[(iv)] $x <\Norm(\Disc(M/F))< c_5(n) x$, and
\item[(v)] $\left| \frac{L'}{L}(1,\chi_M)\right| \le c_6(n)(\log d_F +  \log\log x)$ where $\chi_M$ is the quadratic  Hecke character over $F$ associated to the extension $M/F$
\end{itemize}
is at least $x^{c_7(n)}$.
\end{theorem}
\begin{proof} We construct the desired fields $M$ as $M=L*K$ with $L$ as in Lemma \ref{LemmaOneField} (whose discriminant has controlled size) and $K/\Q$ as in Lemma \ref{LemmaMixing}. Then conditions (i), (ii) and (iii) are satisfied.

Using Corollary \ref{CoroCounting} we can produce enough quadratic extensions $K/\Q$ such that (iv) is also satisfied. Thus, it only remains to show that upon discarding a negligible number of quadratic extensions $K/\Q$, we can also achieve (v).

Given $L/F$ and $K/\Q$ quadratic extensions as above, let $\xi$ and $\psi$ be the corresponding quadratic Hecke characters over $F$ and $\Q$ respectively. Then for $M=L*(FK)$, the quadratic Hecke character over $F$ is 
$$
\chi_M=\xi\cdot \psi_F
$$
where $\psi_F=\psi\circ \Norm_{F/\Q}$ is the base change of $\psi$ from $\Q$ to $F$. In particular,
\begin{equation}\label{EqLfns}
L(s, \chi_M)=L(s, \xi\otimes \psi_F).
\end{equation}
The bound (v) is proved for all but a negligible number of choices of $K$ exactly as in \cite{Lamzouri}. That is, one first proves a version of Proposition 2.3 \emph{loc. cit.} with $k=1$ for $L(s,\chi_M)$, and then one proceeds as in Corollary 2.5 \emph{loc. cit.}, using the zero density estimate in Proposition \ref{prop:ZDE} as $\psi$ (hence $\psi_F$) varies, instead of using Heath-Brown's zero-density from \cite{HeathBrown} which works over $\Q$. The additional term $\log d_F$ in (v) comes from \cite{Weiss} Lemma 1.11(b). 

More precisely, we apply Proposition \ref{prop:ZDE} over the number field $F$ choosing $d=1$ and taking $\pi$ the automorphic representation of $GL_1(\A_F)$ associated to $\xi$.  This produces zero density estimates for $L$-functions of the form $L(s,\pi\otimes \chi)=L(s,\xi\otimes \chi)$ as $\chi$ varies, which is precisely what we need in view of \eqref{EqLfns}. Although Proposition \ref{prop:ZDE} gives a zero density estimate as $\chi$ varies over all finite order Hecke characters (not just for quadratic characters such as $\psi_F$), this is enough for us since the key feature that we need from the result is the exponent $c\cdot (1-\sigma)+\epsilon$ for a constant $c$ that in our application only depends on $n=[F:\Q]$. 
\end{proof}

\begin{remark} If one assumes that $F/\Q$ is solvable then it admits automorphic induction (cf. \cite{ArtClo, ClozelICM}). In this setting, it is not necessary to use the results in the Appendix, because Corollary 1.4 of \cite{OliverThorner} (by the same authors of the Appendix) suffices for producing the necessary zero-density estimates ---in fact, this is how we proceeded in an earlier version of this work. Briefly, the idea is the following: If $\pi_\xi$ is the automorphic representation of $GL_n(\A_\Q)$ induced from $\xi$, then 
$$
L(s, \chi_M)=L(s, \xi\otimes \psi_F)=L(s, \pi_\xi\otimes \psi)
$$
(cf. \cite{Jacquet}). Then one can check that, upon prescribing a few more local conditions on $\xi$ (which can be achieved by adding some additional factors to the ideals $\Ifrak$ and $\Sfrak$) the induced representation $\pi_\xi$ over $\Q$ is cuspidal, and Corollary 1.4 of \cite{OliverThorner} can be applied.
\end{remark}

\section{A modular approach to Szpiro's conjecture over number fields} \label{SecModApproachF}


In this section we introduce some tools that make it possible to approach Szpiro's conjecture over totally real fields using modular parameterizations coming from Shimura curves.

As in the previous section, $n$ will denote the degree of a number field,  and we will continue to write $c_i(n)$ for strictly positive quantities that only depend on $n$. The numeration will be consistent with that of the previous section, although this is only intended to stress the fact that these ``constants'' (depending only on $n$) can change from line to line.

Several of the necessary ideas have been already presented in detail over $\Q$ in previous sections of this paper, so, here the discussion will be more concise.

\subsection{Faltings height of elliptic curves over number fields}\label{SecFaltingsHtTR} In this paragraph we discuss the Faltings height of elliptic curves in more generality than in Section \ref{SecClassical}.

Let $L$ be a number field and let $E$ be an elliptic curve over $L$. Let $\Ecal$ be the N\'eron model of $E$ over $O_L$. For each embedding $\sigma:L\to \C$ we let $E_\sigma=E\otimes_\sigma\C$ be the complexification of $E$ under the embedding $\sigma$. We define the norm $\|-\|_\sigma$ on $H^0(E_\sigma,\Omega^1_{E_\sigma/\C})$ by
$$
\|\alpha \|^2_\sigma:=\frac{i}{2}\int_{E_\sigma}\alpha \wedge \overline{\alpha}.
$$
The \emph{Faltings height} of $E$ over $L$, denoted by $h(E)$, is defined as the normalized  Arakelov degree of the metrized rank $1$ projective module $H^0(\Ecal,\Omega^1_{\Ecal/O_L})$ with the previous metrics at infinity. Namely, taking any non-zero $\beta\in H^0(\Ecal,\Omega^1_{\Ecal/O_L})$ one defines
$$
 h(E)=\frac{1}{[L:\Q]}\left(\log \#(H^0(\Ecal,\Omega^1_{\Ecal/O_L})/O_L\beta) - \sum_{\sigma: L\to \C} \log \|\beta\|_\sigma\right).
$$
Note that this is not the stable Faltings height, and, in general, it can change after enlarging $L$. However, when $A$ is semi-stable over $L$ then $h(A)$ is invariant under base change to a finite extension of $L$.

Silverman \cite{Silverman} proved an alternative formula for $h(E)$ that we now recall. The modular $j$-function and the Ramanujan cusp form $\Delta$ are normalized so that
$$
\begin{aligned}
j(z)&=q^{-1} + 744+...\\
\Delta(z)&= q\prod_{n\ge 1}(1-q^n)^{24}, \quad q=e^{2\pi i z}.
\end{aligned}
$$
For $E$ an elliptic curve over $L$ and $\sigma:L\to \C$ and embedding, we choose $\tau_{E,\sigma}\in\hfrak$ satisfying that $j(\tau_{E,\sigma})$ is the $j$-invariant of $E\otimes_\sigma\C$. Then Silverman's formula is
\begin{equation}\label{EqSilvermanH}
h(E)=\frac{1}{12[L:\Q]}\left(\log  \Delta_E - \sum_{\sigma : L\to \C} \log \left|\Delta(\tau_{E,\sigma})\cdot \Im (\tau_{E,\sigma})^6\right|\right)-\log(2\pi)
\end{equation}
where  $\Delta_E$ is the norm of the minimal discriminant ideal of $E$ over $L$. (Note that in \cite{Silverman} there is a minor typo in the definition of $\Delta(\tau)$ that gives $+\log(2\pi)$ instead of $-\log(2\pi)$; this has been corrected in a number of places). One deduces:

\begin{lemma} \label{LemmaSilverman} With the previous notation, 
$$
\frac{1}{[L:\Q]}\log \Delta_E <  12h(E) + 16.
$$
\end{lemma}
We will need another description of $h(E)$. Suppose that $E$ is semi-stable over $F$ and let $\bar{\Ecal}$ be the corresponding semi-stable model over  $O_L$, such that the smooth locus $\Ecal$ of its structure map  is the N\'eron model of $E$. Let $\omega$ be the relative dualizing sheaf of $\bar{\Ecal}$; in particular $\omega$ is an invertible sheaf ($E$ is semi-stable) and its restriction to $\Ecal$ is $\Omega^1_{\Ecal/O_L}$.

We make $\omega$ into a metrized line bundle $\widehat{\omega}$ with the following metric $\| - \|_\sigma$ associated to an embedding $\sigma: L\to\C$:

 Let $x\in E_\sigma$ and let $\lambda\in \Omega^1_{E_\sigma/\C}|_x$. Let $\alpha\in H^0(E_\sigma,\Omega^1_{E_\sigma/\C})$ be the unique invariant differential with $\alpha_x=\lambda$. Then we define
$$
\|\lambda\|^2_{\sigma,x}=\frac{i}{2}\int_{E_\sigma} \alpha\wedge\overline{\alpha}.
$$
We remark that, in general, the metrized line bundle $\widehat{\omega}$ is \emph{not} the Arakelov canonical metrized line bundle, which one might denote by $\widehat{\omega}^{Ar}$. The metrics of $\widehat{\omega}$ differ from the metrics of  $\widehat{\omega}^{Ar}$ by a scalar multiple.

\begin{theorem}\label{ThmFaltingsHF} Let $E$ be a semi-stable elliptic curve over $L$ and keep the previous notation. Let $P$ be an algebraic point of $E$ satisfying that its closure in $\bar{\Ecal}$ is contained in $\Ecal$. Let $L'/L$ be a finite extension over which $P$ is defined and let $S=\Spec O_L$ and $S'=\Spec O_{L'}$. Let $s:S'\to \Ecal$ be the $S$-map attached to $P$. Then
$$
\widehat{\deg}_{L'} s^*\widehat{\omega} = [L':\Q]\cdot h(E).
$$
\end{theorem}

\begin{proof} This is proved in the same way as Proposition 7.2 in \cite{Jong}, which assumes that $P$ is rational and uses $\widehat{\omega}^{Ar}$ instead of $\widehat{\omega}$. The passage from rational to algebraic points of the type considered here is straightforward, provided that in \emph{loc. cit.} one starts with $(P.\widehat{\omega}^{Ar})$  (instead of starting with $(P.P)$ and using adjunction). On the other hand, replacing $\widehat{\omega}^{Ar}$ by $\widehat{\omega}$ modifies the metrics in a way made explicit by Definition 4.1 and Proposition 4.6 in \emph{loc. cit.} Taking this contribution at infinity into account, one obtains the Faltings height of $E$ in the form \eqref{EqSilvermanH}.
\end{proof}


\subsection{Shimura curves} Let $F$ be a totally real number field of degree $n$ over $\Q$ with $\tau_1,\ldots,\tau_n$ its real embeddings. Let $\Nfrak$ be a non-zero  ideal in $O_F$ which is the product of $\nu$ distinct prime ideals. Suppose that $n+\nu$ is odd and let $B$ be a quaternion $F$-algebra of discriminant $\Nfrak$ and having $\tau_1$ as its only split place at infinity. Furthermore, when $F=\Q$ we assume $\nu>0$. Write $\B=B\otimes \hat{\Z}$ and fix a choice of maximal order $O_\B$ in $\B$. Let $B^\times_+$ be the set of units in $B$ with totally positive reduced norm.

Associated to this data, for each open compact $U\subseteq \B^\times$ there is a compact Shimura curve $X_U$ defined over $F$, whose complex points (via $\tau_1$)  are given by
$$
X^{an}_U= B^\times_+\backslash \hfrak\times\B^\times/U.
$$

The curve $X_U$ is irreducible over $F$, although $X_U^{an}$ is not necessarily irreducible and its connected components are parametrized by the class group $C_U=F^\times_+\backslash \A_F^\infty/\rn (U)$ where $\rn$ denotes the reduced norm and $F^\times_+$ is the multiplicative group of totally positive elements of $F$. We also keep the notation from Paragraph \ref{SecNormsQuat1}.

The following expression for the hyperbolic volume of the Shimura curve with $U=O_\B^\times$ is a special case of a result of Shimizu \cite{Shimizu}.
\begin{proposition}\label{PropVolF} The number of connected components of $X_{O_\B^\times}^{an}$ is $h^+_F$, the narrow class number of $F$. Each component has the form $\tilde{\Gamma}_{O_\B^\times, g}\backslash \hfrak$ for suitable $g\in \B^\times$, and they have hyperbolic area
$$
Vol\left(\tilde{\Gamma}_{O_\B^\times, g}\backslash \hfrak\right)\asymp_n d_F^{3/2}\prod_{\pfrak | \Nfrak}(\Norm(\pfrak)-1).
$$ 
\end{proposition}

The Jacobian of $X_U$ over $F$ is denoted by $J_U$. We denote by $\T_U^c$ the ring of Hecke correspondences on $X_U$ and by $\T_U$ the ring of Hecke operators acting on $J_U$. They are generated by the Hecke correspondences $T^c_\nfrak$ (resp. Hecke operators $T_\nfrak$) for all $\nfrak$ integral ideals of $O_F$ coprime to $\Nfrak$ and coprime to the places where $U$ is not maximal. Then $\T_U$ acts also on holomorphic differentials of $J_U$ and of $X_U$. See \cite{Gross} for details. Note that the action of $\T^c_\nfrak$ on divisors permutes the components of $X_U^{an}$ according to the action of $\nfrak$ on $C(U)$.

The Shimura construction associated to a system of Hecke eigenvalues $\chi:\T_U\to\C$ has been carried out in \cite{ZhangAnn} and the theory is similar to that of Shimura curves over $\Q$. In particular, when $\chi$ is $\Z$-valued and new, the associated optimal quotient $q:J_U\to A$ satisfies that $A$ is an elliptic curve. We define the modular degree of $A$ arising in this way, as the integer 
$$
\delta_\chi=qq^\vee\in\End(A).
$$
(We remark that $\delta_\chi$ is not \emph{a priori} the degree of any particular morphism from a Shimura curve to $A$.) Furthermore, in the special case $U=O_\B^\times$ (for instance), results of Carayol \cite{Carayol} give that the conductor ideal of $A$ is precisely $\Nfrak$.

In addition, we define Eisenstein correspondences (cf. Paragraph \ref{SecMapsQ}) of the form
$$
E^c_{U,\nfrak}=T^c_\nfrak-\sigma_1(\nfrak)\Delta_U
$$
where $\sigma_1(\nfrak)=\sum_{\afrak|\nfrak}\Norm(\afrak)$, $\Delta_U$ is the diagonal correspondence on $X_U$, and $\nfrak$ is required to act as the identity on $C(U)$. In the particular case $U=O_\B^\times$ we have that $C(U)$ is the narrow class group of $F$ and we can use $\nfrak$ coprime to $\Nfrak$, principal and with a totally positive generator, e.g. $\nfrak=(p)$ with $p$ a rational prime not dividing $\Norm(\Nfrak)$. Such a choice of $\nfrak$ will ensure that the analogue of condition (*) of Paragraph \ref{SecMapsQ} holds for $E^c_{O_\B^\times , \nfrak}$ in our current setting.

Using Eisenstein correspondences as in Section \ref{SecModDeg} (in the special case $U=O_\B^\times$ for simplicity) we obtain:

\begin{theorem} \label{ThmDegreeTR} Let $X_{O_\B^\times}$ be the Shimura curve over $F$ (for $U=O_\B^\times$) associated to an indefinite quaternion $F$-algebra $B$ of discriminant $\Nfrak$ as above. Let $\chi$ be a $\Z$-valued new system of Hecke eigenvalues on $\T_{O_\B^\times}$ and let $q:J_{O_\B^\times}\to A$ be the associated elliptic curve optimal quotient with modular degree $\delta_\chi$. There is a map $j:X_{O_\B^\times}\to J_{O_\B^\times}$ over $F$ such that the composition
$$
\phi=qj:X_{O_\B^\times}\to A
$$
has degree satisfying
$$
h_F^+\delta_\chi \le \deg\phi \le c_{8}(n) (\log \Norm(\Nfrak))^{c_{9}(n)} h_F^+ \delta_\chi
$$
where $h_F^+$ is the narrow class number of $F$.
\end{theorem}


\subsection{Modular elliptic curves} We say that an elliptic curve $E$ over a totally real field $F$ is (geometrically) modular if it is isogenous to a factor of the Jacobian $J_U$ of a Shimura curve $X_U$ for some choice of open compact subgroup $U$ in $\B^\times$, and for a suitable choice of $\Nfrak$ (with the same notation as in the previous paragraph). 

For technical simplifications, it turns out to be more convenient to require a more restrictive condition. Let us say that $E$ is \emph{modular of Shimura level $1$} if in addition one can choose $\Nfrak$ as the conductor ideal of $E$ and $U=O_\B^\times$. In this case $E$ is semi-stable (so, its conductor is squarefree) and the parity of the number of primes of bad reduction of $E$ (i.e. of the prime factors of its conductor) is opposite to the parity of $n=[F:\Q]$. The usual modularity conjecture for elliptic curves over totally real number fields along with the Jacquet-Langlands correspondence, imply that in fact each elliptic curve over $F$ whose conductor satisfies these two conditions, is modular of Shimura level $1$ in the sense that we just defined.

 If $E$ is modular of Shimura level $1$ we simply write $X_1=X_{O_\B^\times}$ for the Shimura curve affording the modular parameterization $X_1\to E$ of the definition. We observe that in this case $E$ is isogenous to $A$ for an optimal elliptic curve quotient $q:J_1\to A$, where $J_1$ is the Jacobian of $X_1$ over $F$. The associated system of Hecke eigenvalues is denoted by $\chi_E$ and the modular degree $\delta_{\chi_E}$ is simply written $\delta_E$. 
 
 One would like to control the minimal degree of an isogeny $A\to E$ and to combine this bound with Theorem \ref{ThmDegreeTR}. However, a suitable version of the ``isogeny theorem'' for elliptic curves over number fields of bounded degree has not yet been proved, while we would like some uniformity as the number field varies. It turns out that the  bounds for the degree of a minimal  isogeny (with controlled field of definition) proved by Gaudron and R\'emond  \cite{GauRem} are  enough for our purposes, and one obtains:

\begin{theorem} \label{ThmDegreeCompF} Let $E$ be an elliptic curve over a totally real number field $F$ and assume that $E$ is modular of Shimura level $1$. Then, with the previous notation, there is a non-constant map
$$
\phi_E : X_1\to E
$$
defined over $F$, whose degree satisfies
$$
h_F^+\delta_E\le \deg \phi_E \le c_{10}(n) \max\{1,h(E)\}^{c_{11}(n)}h_F^+\delta_E
$$
where $h(E)$ denotes the (logarithmic) Faltings height of $E$.
\end{theorem}

(Note that the factor $(\log \Norm(\Nfrak))^{c_{9}(n)}$ from Theorem \ref{ThmDegreeTR} has been absorbed by the factor $\max\{1,h(E)\}^{c_{11}(n)}$, in view of Lemma \ref{LemmaSilverman}.)


\subsection{Arakelov height of Heegner points after Yuan and Zhang} \label{SecYuZhTR} In this paragraph we briefly recall some of the main points of the theory developed by Yuan and Zhang in \cite{YuZh}.

As $U$ varies over open compact subgroups of $O_\B^\times$ (notation as above) one obtains a projective system $\{X_U\}_U$ of curves over $F$ mapping to $X_{O_\B^\times}$. 

For a positive integer $m$, write $U(m)=(1+ m O_\B)^\times$. The next is Proposition 4.1 in \cite{YuZh}.
\begin{proposition}
If $m\ge 3$ and $U\subseteq U(m)$, then for every $g\in\B^\times$ the group $\tilde{\Gamma}_{U,g}$ acts freely on $\hfrak$, and the genus of every geometric component of $X_U$ is at least $2$.
\end{proposition} 
 Write $X_m$ for the curve $X_{U(m)}$, in particular $X_1=X_{O_\B^\times}$ agrees with our previous notation. Note that when $m_1|m_2$ we have a natural map $X_{m_2}\to X_{m_1}$. The next result follows from the first part of Section 4.2 in \cite{YuZh}.
\begin{theorem} For $m\ge 3$ coprime to $\Norm(\Nfrak)$, there is a canonical regular integral model $\Xcal_m$ for $X_m$, which is flat and projective over $O_F$. Furthermore, $\Xcal_m[1/m]$ is semi-stable over $O_F[1/m]$. More precisely, for a prime ideal $\pfrak$ of $O_F$ not dividing $m$ one has that the special fibre at $\pfrak$ is
\begin{itemize}
\item[(i)] relative Mumford, if $\pfrak |\Nfrak$;
\item[(ii)] smooth, if $\pfrak\nmid \Nfrak$.
\end{itemize}
For any $m_0$ coprime to $\Norm(\Nfrak)$ an integral model $\Xcal_{m_0}$ for $X_{m_0}$ is obtained as follows: Choose any  $m\ge 3$ coprime to $\Norm(\Nfrak)$ with $m_0|m$, then $\Xcal_{m_0}$ is defined as the quotient of $\Xcal_{m}$ by the natural action of the finite group $U(m_0)/U(m)$. The scheme $\Xcal_{m_0}$ is independent of the choice of $m$ (up to isomorphism), it is normal, and it is flat and projective over $O_F$. In particular, this construction applies  to $X_1$.

For $m_1|m_2$ positive integers coprime to $\Norm(\Nfrak)$, the map $X_{m_2}\to X_{m_1}$ extends to an  $O_F$-morphism $\Xcal_{m_2}\to \Xcal_{m_1}$.
\end{theorem}

In \cite{YuZh} integral models $\Xcal_U$ for $X_U$ are constructed in more generality, but for our purposes the case of $X_m$ with $m$ a positive integer coprime to $\Norm(\Nfrak)$ suffices. In what follows, such an integer $m$ will be called \emph{admissible}.

For each admissible $m$ let  $\widehat{\Lcal}_m$ be the metrized Hodge bundle on $\Xcal_m$, as constructed in \cite{YuZh} Section 4.2. This is a metrized $\Q$-line bundle on $\Xcal_m$  in general, and when $m\ge 3$ it is a metrized line bundle on $\Xcal_m$. Its finite part (i.e. forgetting the metrics) is denoted by $\Lcal_m$. The next result follows from Theorem 4.7 \cite{YuZh}.
\begin{theorem} \label{ThmHodge} The metrized Hodge bundles satisfy the following properties:
\begin{itemize}
\item[(i)] for $m_1|m_2$ admisible integers, the pull-back of $\widehat{\Lcal}_{m_1}$ by $\Xcal_{m_2}\to \Xcal_{m_1}$ is $\widehat{\Lcal}_{m_2}$ ;
\item[(ii)] for $m\ge 3$ admissible, we have that $\Lcal_m[1/m]$ is the relative dualizing sheaf of 
$$
\Xcal_m[1/m]\to \Spec O_F[1/m] ;
$$
\item[(iii)] For each embedding $\sigma : F\to \R\subseteq  \C$, the metric on the restriction of $\Lcal$ to  $X_m\otimes_\sigma \C$ is induced via complex uniformization by the metric $|dz|=2\Im(z)$ on differential forms on $\hfrak$.
\end{itemize}
\end{theorem}

A totally imaginary quadratic extension $K/F$ is said to satisfy the (generalized) Heegner hypothesis for $\Nfrak$ if every prime $\pfrak|\Nfrak$ is inert in $K$. In particular, such an extension $K/F$ has relative discriminant $\dfrak_{K/F}$ coprime to $\Nfrak$.

If $K$ satisfies the Heegner hypothesis for $\Nfrak$,  for each open compact $U$ we have a Heegner point $P_{K,U}$ in $X_U$. These can be chosen to form a compatible system of algebraic points in the projective system $\{X_U\}_U$. The point $P_K=P_{K,O_\B^\times}$ is defined over the Hilbert class field of $K$, and in general, all the $P_{K,U}$ are defined over abelian extensions of $K$.

Theorem 1.5 in \cite{YuZh} gives a formula for the Arakelov-theoretical height 
$$
h_{Ar} (P_K):=\frac{1}{[H_K:F]} \widehat{\deg}_{O_{H_K}} s^*\widehat{\Lcal}_1
$$
of $P_K$ on $\Xcal_1$ with respect to $\widehat{\Lcal}_1$, where $s:\Spec O_{H_K}\to \Xcal_1$ is the multi-section attached to $P_K$. The result is:
\begin{theorem} \label{ThmYuZhF} Suppose that $F$ has at least two ramified places in $B$. Let $K/F$ be a totally imaginary quadratic  extension satisfying the Heegner hypothesis for $\Nfrak$. Let $\chi_K$ be the quadratic Hecke character of $F$ associated to the extension $K/F$. Then we have
$$
h_{Ar}(P_K)=-\frac{L'}{L}(0,\chi_K)  + \frac{1}{2}\log \frac{\Norm(\Nfrak)}{\Norm(\dfrak_{K/F})}.
$$
This can be rewritten as
$$
h_{Ar}(P_K)= \frac{L'}{L}(1,\chi_K) + \frac{1}{2}\log \Norm(\dfrak_{K/F}\Nfrak) + \log d_F - n\cdot (\gamma + \log(2\pi)).
$$
\end{theorem}
The first formula is Theorem 1.5 in \cite{YuZh}, while the second is simply a consequence of the functional equation of $L(s,\chi_K)$.



\subsection{An approach to Szpiro's conjecture over totally real number fields}
\begin{theorem} \label{ThmSzpiroF} Let $F$ be a totally real number field of degree $n$ over $\Q$. Let $E$ be an elliptic curve over $F$ and suppose that $E$ is modular of Shimura level $1$. Let $\Nfrak$ be the conductor ideal of $E$, let $B$ be a  quaternion $F$-algebra of discriminant $\Nfrak$ with exactly one split place at infinity, and let $X_1$ be the associated Shimura curve. 

Let $\psi:X_1\to E$ be a non-constant morphism over $F$ (which exists, as $E$ is modular of Shimura level $1$). Then we have
$$
h(E) \le   \frac{1}{2}\log \deg\psi + c_{12}(n)\log (d_F\Norm(\Nfrak))
$$
and
$$
\log \Delta_E \le  6n\log \deg\psi + c_{13}(n)\log (d_F\Norm(\Nfrak)).
$$
\end{theorem}
\begin{proof}
For each admissible $m$, let $\psi_m: X_m\to E$ be the composition of $X_m\to X_1$ with $\psi$. For $m\ge 3$ admissible, the map $\psi_m$ extends to an $O_F[1/m]$-morphism
$$
\psi_m: \Xcal_m[1/m]^\circ \to \Ecal[1/m]
$$
by the N\'eron mapping property, where $\Ecal$ is the N\'eron model of $E$ and $\Xcal_m[1/m]^\circ$ is the smooth locus of $\Xcal_m[1/m]\to \Spec O_F[1/m]$, obtained from $\Xcal_m[1/m]$ by removing the singular points of the special fibres at primes dividing $\Nfrak$.

Choose two different admissible prime numbers $p,q$, different from $2$ and small in the sense that
$$
p,q < 100\log \Norm\Nfrak
$$
(say), which is possible by a variation of Lemma \ref{LemmanmidGh}. Let $\Ycal_{pq}$ be the open set of $\Xcal_{pq}$ obtained as the union of the preimages of $\Xcal_p[1/p]^\circ\subseteq \Xcal_p$ and $\Xcal_q[1/q]^\circ\subseteq \Xcal_q$. Then by glueing we obtain an $O_F$-morphism
$$
\psi_{pq}: \Ycal_{pq}\to \Ecal.
$$
Let $K/F$ be a totally imaginary quadratic extension of $F$ satisfying the Heegner hypothesis. Furthermore, assume that the associated Heegner point $P_{K,U(pq)}$ is not a ramification point of $\psi_{pq}$. Let $H$ be a number field containing $K$ over which $P_{K,U(pq)}$ is defined. Write $S_H=\Spec O_H$ and let $s:S_H\to \Xcal_{pq}$ be the $O_F$-map obtained from the algebraic point $P_{K, U(pq)}$. Observe that $s$ exists because $\Xcal_{pq}$ is projective over $O_F$, and it has image contained in $\Ycal_{pq}$ because all primes dividing $\Nfrak$ are inert in $K$ (cf. p. 242 in \cite{YZZbook}).

By a patching argument using $\psi_p$ and $\psi_q$, we get the canonical injective sheaf morphism (cf. the notation  in Paragraphs \ref{SecFaltingsHtTR} and \ref{SecYuZhTR})
$$
u:\psi_{pq}^*(\omega|_{\Ecal}) \to \Lcal_{pq}|_{\Ycal_{pq}}
$$
which allows us to see the first as a sub-sheaf of the second, on $\Ycal_{pq}$. The sheaf $\psi_{pq}^*(\omega|_{\Ecal})$ has the norms $\|-\|'_{\sigma}$ induced from the metrized relative dualizing sheaf $\widehat{\omega}$ of the semi-stable elliptic curve $E$, while the Hodge bundle has its own norms $\|-\|_{\sigma}$. We now compare these norms. 

Let $\sigma: F\to \C$ be any embedding and take a complex point $x\in X_{pq}\otimes_\sigma \C$. Let $\mu \in \psi_{pq}^*(\Omega^1_{E_\sigma/\C})|_x$ and let $\lambda \in \Omega^1_{E_\sigma/\C}|_{\psi_{pq}(x)}$ be such that $\psi_{pq}^*\lambda=\mu$. Let $\alpha\in H^0(E_\sigma, \Omega^1_{E_\sigma/\C})$ be the invariant differential with $\alpha_{\psi_{pq}(x)}=\lambda$. Then by definition of the pull-back norm on $\psi_{pq}^*(\omega|_{\Ecal})$ we have
\begin{equation}\label{EqNormTR1}
\left(\|\mu\|'_{ \sigma, x}\right)^2 = \|\lambda\|^2_{ \sigma, \psi_{pq}(x)} = \frac{i}{2}\int_{E_\sigma} \alpha\wedge\overline{\alpha}.
\end{equation}
We now estimate $\|u(\mu)\|_{ \sigma, x}$ according to the metrics of the Hodge bundle. Note that
$$
u(\mu)=u(\psi_{pq}^*\lambda)= u(\psi_{pq}^*(\alpha_{\psi_{pq}(x)}))=u(\psi_{pq}^*\alpha)_x= (\psi_{pq}^\bullet \alpha)_x 
$$
and therefore for suitable $g\in \B^\times$ according to the connected component in which $x$ is, we have 
$$
\|u(\mu)\|_{ \sigma, x}\le 2 \|\Psi_{U(pq),g}(\psi_{pq}^\bullet \alpha)\|_{U(pq),g,\infty}\le c_{14}(n)  \|\Psi_{U(pq),g}(\psi_{pq}^\bullet \alpha)\|_{U(pq),g,2}
$$
where $\Psi_{U,g}$ is defined analogously to the definition in Paragraph \ref{SecNotationForms} (adapting from $\Q$ to $F$), and the second inequality is by Theorem \ref{ThmNorms}. 

By pulling back the $(1,1)$-form $\alpha\wedge\overline{\alpha}$ from \eqref{EqNormTR1}, we now observe that
$$
\|\Psi_{U(pq),g}(\psi_{pq}^\bullet \alpha)\|^2_{U(pq),g,2} \le \deg(\psi_{pq})\|\mu\|^2_{ \sigma, x}
$$
(this is only a bound rather than an equality because the left hand side just considers one  geometric component of $X_{U(pq)}$) which gives the norm comparison
\begin{equation}\label{EqNormTR2}
\|u(\mu)\|_{ \sigma, x}\le c_{15}(n) (\deg \psi_{pq})^{1/2} \|\mu\|_{ \sigma, x}.
\end{equation}
Pulling back $u$ by $s$ we get a sheaf morphism on $S_H$
$$
u': (\psi_{pq}s)^*\omega=s^*\psi_{pq}^*(\omega|_{\Ecal})\to s^*\Lcal_{pq}|_{\Ycal_{pq}}=s^*\Lcal_{pq}
$$
which still is injective because $P_{K,U(pq)}$ is not in the ramification locus of $\psi_{pq}$, by assumption. Injectivity, together with the norm comparison \eqref{EqNormTR2} gives
$$
\widehat{\deg}_{S_H} (\psi_{pq}s)^*\widehat{\omega}\le  \widehat{\deg}_{S_H} s^*\widehat{\Lcal}_{pq} + \frac{[H:\Q]}{2}\log \deg\psi_{pq} + c_{16}(n) [H:\Q].
$$
Dividing by $[H:\Q]=[H:F]\cdot n$, recalling Theorem \ref{ThmFaltingsHF}, the definition of $h_{Ar}(P_K)$, and the fact that the metrized Hodge bundles are compatible with pull-back (Theorem \ref{ThmHodge}), one obtains
$$
h(E) \le  \frac{1}{n} h_{Ar}(P_K) + \frac{1}{2}\log \deg\psi_{pq} + c_{17}(n).
$$
By Theorem \ref{ThmYuZhF} we get
\begin{equation}\label{EqnKEYF}
n\cdot h(E) \le \frac{n}{2}\log \deg\psi_{pq} + \frac{L'}{L}(1,\chi_K) + \frac{1}{2}\log \Norm(\dfrak_{K/F}\Nfrak) + \log d_F + c_{18}(n).
\end{equation}
The number of complex ramification points of $\psi_{m}$ is at most $2g_{m}-2$ where $g_{m}$ is the sum of the genera the geometric components of $X_{m}$. This is at most the total (hyperbolic) volume of $X_{m}$ divided by $2\pi$, which is at most
$$
\frac{1}{2\pi}[U(1):U(m)]\cdot Vol(X_1^{an})\le  c_{19}(n)(md_F)^{c_{20}(n)}  \Norm(\Nfrak)
$$
since $X_1^{an}$ has $h_F^+$ components, and using the volume formula from Proposition \ref{PropVolF}. 

Hence, by Theorem \ref{ThmCountingF} and recalling that $p$ and $q$ are small, there is some totally imaginary quadratic  extension $K/F$ satisfying the Heegner hypothesis for $\Nfrak$, such that $P_{K,U(pq)}$ is not a ramification point of $\psi_{pq}$, and satisfying the estimates
$$
\Norm (\dfrak_{K/F}) < c_{21}(n)(d_F\Norm(\Nfrak))^{c_{22}(n)}
$$
and
$$
\left|\frac{L'}{L}(1,\chi_K) \right| < c_{23}(n)\left(\log d_F + \log \log \Norm(\Nfrak)\right).
$$
From \eqref{EqnKEYF} and using the fact that the degree of $X_m\to X_1$ is at most $m^{c_{24}(n)}$ (and that $p$ and $q$ are small), we finally deduce
$$
h(E)\le \frac{1}{2}\log \deg \psi + c_{25}(n) \log (d_F\Norm(\Nfrak)).
$$
The estimate for  $\Delta_E$ follows by Lemma \ref{LemmaSilverman}
\end{proof}

Theorem \ref{ThmSzpiroF} motivates the following conjecture:

\begin{conjecture} \label{ConjModDegTR} Suppose that $F$ is a totally real number field of degree $n$ over $\Q$, and that $E/F$ is an elliptic curve which is modular of Shimura level $1$. Then
$$
\log \delta_E < \kappa\cdot \log (d_F\Norm(\Nfrak))
$$
for some constant $\kappa$ that only depends on the degree $n$.
\end{conjecture}

Note that for a given field $F$, the conjecture only concerns modular elliptic curves of Shimura level $1$ and it involves the canonically defined quantity $\delta_E$ instead of using the degree of some arbitrary choice of modular parametrization $X_1\to E$. Conjecture \ref{ConjModDegTR} is of interest because it implies Szpiro's conjecture for such elliptic curves, as we prove in the next paragraph. 
 

\subsection{Bounds for $\delta_E$ and Szpiro's conjecture}

The purpose of this paragraph is to spell-out the precise relation between Szpiro's conjecture and our proposed conjectural bound for the quantity $\delta_E$, namely, Conjecture \ref{ConjModDegTR}. 

First we do this in a particularly convenient case where several technical assumptions can be removed: the case when $F$ is real quadratic. In particular, here one does not need to assume modularity as a hypothesis, because  it is proved in the relevant cases.
\begin{theorem} \label{ThmQuadraticCond} Assume Conjecture \ref{ConjModDegTR} for real quadratic fields and their totally real quadratic extensions. Then there is an absolute constant $c>0$ such that the following holds:

For every real quadratic field $F$ and every semi-stable elliptic curve $E$ over $F$ which does not have everywhere good reduction, we have
$$
h(E)\le c \cdot \log (d_FN_E)\quad \mbox{ and }\quad \Delta_{E}\le (d_FN_E)^{c}.
$$
\end{theorem}
\begin{proof} If $E$ has an odd number of places of bad reduction, then it is modular of Shimura level $1$ (according to our definition) by the modularity results from \cite{FLS} (see also \cite{YZZbook, ZhangAnn} to deduce geometric modularity from automorphy). In this case we have, for the corresponding Shimura curve $X_1$, a modular parameterization 
$$
\phi: X_1\to E
$$
afforded by Theorem \ref{ThmDegreeCompF}, whose degree satisfies
$$
\log \deg(\phi)\le  \log \delta_E +\log h_F^+ + O_n\left(1+ \log\max\{1, h(E)\} \right).
$$
By Theorem \ref{ThmSzpiroF} and bounding $h_F^+$ in terms of $d_F$, we get (unconditionally)
$$
h(E)  \le \frac{1}{2}\log \delta_E + O_n\left( \log\max\{1, h(E)\} + \log (d_FN_E)\right).
$$
At this point we apply Conjecture \ref{ConjModDegTR} to get the result. 

When $E$ has a non-zero even number of primes of bad reduction over $F$, we base-change to a suitable totally real quadratic extension $F'/F$, such that exactly one prime of bad reduction of $E$ is inert and all the other are split. Since $E$ is semi-stable over $F$, now it has an odd number of bad places over $F'$ and by \cite{DieFre} it is modular of Shimura level $1$. As $E/F$ is semi-stable, its Faltings height is the same after base change. Furthermore, $F'$ can be chosen with $\log d_{F'}\ll \log (d_FN_E)$ by Lemma \ref{LemmaOneField}. Now the same argument applies.
\end{proof}

Finally, here is a version of the previous result for totally real number fields beyond the quadratic case, under the assumption of modularity. 
\begin{theorem} \label{ThmModApproach} Assume Conjecture \ref{ConjModDegTR}. Assume that elliptic curves over totally real number fields are modular (in the sense of \cite{Gelbart}). Then the following holds:

Let $F$ be a totally real number field of degree $n$ over $\Q$.  Then for every semi-stable elliptic curve $E$ over $F$ which does not have everywhere good reduction, we have
$$
h(E)\le \kappa\cdot  \log (d_FN_E)\quad\mbox{ and }\quad \Delta_{E}\le (d_FN_E)^{\kappa}
$$
for some constant $\kappa$ that only depends on the degree $n$.
\end{theorem}
The proof is similar to that of Theorem \ref{ThmQuadraticCond}. Of course, our methods also show that partial modularity results together with partial progress towards Conjecture \ref{ConjModDegTR} are enough to give consequences for Szpiro's conjecture. We discuss such an application in the next paragraph. 

\subsection{Application: Unconditional exponential bounds} \label{SecApplF} Finally, we observe that unconditional exponential bounds for $\delta_E$ over a  totally real field $F$ (and thanks to our results, for Szpiro's conjecture) can be proved by the same methods of Sections \ref{SecModDeg} and \ref{SecBounds}. Namely, it is straightforward to extend any of the two proofs of Theorem \ref{ThmSpectral} to the totally real setting, and then such an extension is used to prove an analogue of Theorem \ref{ThmBoundUncond}. To do the latter, one should use results on effective multiplicity one for $GL_2$ ---such as those in \cite{Brumley, YWang, LiWa}--- in order to distinguish systems of Hecke eigenvalues using only few Hecke operators. To count the number of systems of Hecke eigenvalues used in the proof, a crude bound is the genus of $X_{O_\B^\times}$, which can be estimated by Shimizu's volume formula.

For instance, an unconditional consequence is the following:

\begin{theorem} \label{ThmFinalApp}
Let $F$ be a totally real number field and let $\epsilon>0$.  For all semi-stable elliptic curves $E$ over $F$ satisfying that the number of places of bad reduction of $E$ has opposite parity to $[F:\Q]$, we have $\log \delta_E\ll_{F,\epsilon} N_E^{1+\epsilon}$, hence
$$
h(E) \ll_{F,\epsilon} N_E^{1+\epsilon} \quad \mbox{ and }\quad \log \Delta_E \ll_{F,\epsilon} N_E^{1+\epsilon}.
$$
\end{theorem}

Indeed, Theorem 5 in  \cite{FLS} gives a finite list of possible $j$-invariants for non-modular elliptic curves over $F$. As we are assuming semi-stability, this translates into finitely many $F$-isomorphism classes of elliptic curves that can fail to be modular in our setting, and they only contribute to a possibly different implicit constant. In all other cases our bounds apply, hence the result.

We note that the list of exceptional $j$-invariants comes from Faltings theorem for curves and it is in general ineffective. In the case that $F$ is real quadratic the list of exceptional $j$-invariants is empty \cite{FLS} hence, the only way in which a real quadratic $F$ contributes to the implicit constants of Theorem \ref{ThmFinalApp} is by means of the application of effective multiplicity one for $GL_2$ over varying quadratic number fields. 

The problem of proving good bounds for effective multiplicity one on $GL_2$ over number fields of bounded degree with explicit dependence on the number field (and without the assumption of GRH) seems to be an open problem in analytic number theory.


\section{Acknowledgments}

Part of this project was carried out while I was a member at the Institute for Advanced Study in 2015-2016, and then continued at Harvard. I greatly benefited from conversations with Enrico Bombieri, Noam Elkies, Nicholas Katz, Barry Mazur, Peter Sarnak, Richard Taylor, and Shou-Wu Zhang, and I sincerely thank them for generously sharing their ideas and knowledge. In particular, the argument in Paragraph \ref{SecSecondProofDeg} originates in an idea of R. Taylor in the case of classical modular curves, and the connection between injectivity radius and Lehmer's conjecture used in Section \ref{SecNorms} was pointed out to me by P. Sarnak. Feedback from B. Mazur on the topic of the Manin constant was  of great help, and I also thank Kestutis Cesnavicius and Bas Edixhoven for some additional observations on this subject. I learned some of the material relevant for this project by attending the joint IAS-Princeton algebraic number theory seminar during my stay at the IAS, and I thank  Christopher Skinner and Richard Taylor for organizing it.  

Comments from Natalia Garcia-Fritz and Ricardo Menares regarding the presentation of this article are gratefully acknowledged. I also thank Robert Lemke Oliver and Jesse Thorner for the appendix. Finally, I thank Enrico Bombieri and Natalia Garcia-Fritz for their encouragement, which allowed me to push this project further than I initially had in mind.


\appendix

\section{Zeros of twisted  $L$-functions near $\Re(s)=1$ (by Robert Lemke Oliver and Jesse Thorner)}

Let $K$ be a number field with absolute discriminant discriminant $d_K$, and let $\chi$ be a primitive ray class character defined over $K$.  Our goal is to prove:
\begin{proposition}
\label{prop:MVT}
	Let $Q,T\geq 1$ and $\epsilon>0$.  Suppose that $y=T^{[K:\Q]/2+2+2\epsilon} d_K^{1/2} Q^{5}$.  For any function $b(\pfrak)$ supported on the prime ideals of $K$ with norm greater than $y$,
	\[
	\sum_{\mathrm{N}\qfrak\leq Q}~\sideset{}{'}\sum_{\chi\bmod\qfrak}\int_{-T}^{T}\Big|\sum_{\mathrm{N}\pfrak>y}\chi(\pfrak)b(\pfrak)\mathrm{N}\pfrak^{-it}\Big|^2 dt\ll_{\epsilon,[K:\Q]}(d_K Q)^{\epsilon}\sum_{\mathrm{N}\pfrak>y}|b(\pfrak)|^2\mathrm{N}\pfrak.
	\]
Here, $\sideset{}{'}\sum_{\chi\bmod\qfrak}$ denotes a sum over primitive ray class characters $\chi$ of modulus $\qfrak$ over $K$.
\end{proposition}

With Proposition \ref{prop:MVT} in hand, one can use the proofs in \cite[Sections 3 and 4]{OliverThorner} to prove a zero density estimate for twisted $L$-functions.  Let $\mathbb{A}_{K}$ denote the ring of adeles of $K$, and let $\pi$ be a cuspidal automorphic representation of $\mathrm{GL}_d(\mathbb{A}_{K})$ with unitary central character.  We make the implicit assumption that the central character of $\pi$ is trivial on the product of positive reals when embedded diagonally into the (archimedean places of the) ideles.  The standard $L$-function associated to $\pi$ is of the form
\[
L(s,\pi)=\sum_{\afrak}\frac{\lambda_{\pi}(\afrak)}{\mathrm{N}\afrak^s}=\prod_{\pfrak}\prod_{j=1}^d (1-\alpha_{\pi}(j,\pfrak)\mathrm{N}\pfrak^{-s})^{-1},
\]
where the sum runs over the nonzero integral ideals of $K$ and the product runs over the prime ideals of $K$.  Assume that $\pi$ satisfies the generalized Ramanujan conjecture, in which case $|\alpha_{\pi}(j,\pfrak)|\leq 1$ for every $j$ and $\pfrak$.  We now consider the zeros of the twisted $L$-functions $L(s,\pi\otimes\chi)$. Let $N_{\pi\otimes\chi}(\sigma,T):=\#\{\rho=\beta+i\gamma\colon L(\rho,\pi\otimes\chi)=0,~\beta\geq\sigma,~|\gamma|\leq T\}$.

\begin{proposition}
	\label{prop:ZDE}
	Let $\epsilon>0$, let $Q,T\geq 1$, and let $1/2\leq\sigma\leq 1$.  If $\pi$ satisfies the generalized Ramanujan conjecture, then there exists a constant $c>0$ (depending only on $d$ an $[K:\Q]$) such that
	\[
	\sum_{\mathrm{N}\qfrak\leq Q}~\sideset{}{'}\sum_{\chi\bmod{\qfrak}}N_{\pi\otimes\chi}(\sigma,T)\ll_{\epsilon,[K:\Q]}(C(\pi)Q T)^{c(1-\sigma)+\epsilon},
	\]
	where $C(\pi)$ is the analytic conductor of $\pi$.
\end{proposition}

The proof of Proposition \ref{prop:ZDE} proceeds exactly the same as \cite[Corollary 1.4]{OliverThorner}.  While the density estimate here is not as strong as \cite[Corollary 1.4]{OliverThorner}, the point is that now one can take $\pi$ and the Hecke characters $\chi$ to be defined over number fields other than $\Q$.  Since we aim for brevity, Propositions \ref{prop:MVT} and \ref{prop:ZDE} are not the strongest results that the method can produce.  A more careful treatment will appear in forthcoming work by R. Lemke Oliver, B. Linowitz, and J. Thorner.

Let $\mathfrak{q}$ be an integral ideal of $K$, let $I(\mathfrak{q})$ be the group of fractional ideals which are relatively prime to $\mathfrak{q}$, and let $P_{\mathfrak{q}}$ be the subgroup of $I(\mathfrak{q})$ consisting of principal fractional ideals $(\alpha)$ with $\alpha$ totally positive and $\alpha\equiv 1~\textup{mod}^{*}~\mathfrak{q}$.  Thus $I(\mathfrak{q})/P_{\mathfrak{q}}$ is the group of ideal classes modulo $\mathfrak{q}$ of $K$ (in the ``narrow'' sense, which involves no essential loss of generality).

We write $\chi\pmod{\mathfrak{q}}$ to denote a Hecke character of the finite abelian group $I(\mathfrak{q})/P_{\mathfrak{q}}$.  If $\mathfrak{q}$ divides $\mathfrak{n}$, then the inclusion map $I(\mathfrak{n})\to I(\mathfrak{q})$ induces a surjective homomorphism $I(\mathfrak{n})/P_{\mathfrak{n}}\to I(\mathfrak{q})/P_{\mathfrak{q}}$.  Composing a character $\chi\pmod{\mathfrak{q}}$ with this map induces a character $\chi^{\prime}\pmod{\mathfrak{n}}$.  We recall that $\chi$ is a \emph{primitive character} if the only character which induces $\chi$ is $\chi$ itself.  Given a character $\chi\pmod{\mathfrak{q}}$, the equivalence class of $\chi$ contains a unique primitive character $\tilde{\chi}\bmod\mathfrak{f}_{\chi}$ so that the equivalence class contains precisely those characters which are induced by $\tilde{\chi}$.  The integral ideal $\mathfrak{f}_{\chi}$ is called the \emph{conductor} of $\chi$, and it depends only on the equivalence class of $\chi$.

While the above definition of a Hecke character $\chi\pmod{\mathfrak{q}}$ holds only for those integral ideals which are coprime to $\mathfrak{q}$, we may extend $\chi$ to be a completely multiplicative function on the integral ideals of $K$ by setting $\chi(\mathfrak{a})=0$ if $\gcd(\mathfrak{a},\mathfrak{q})\neq (1)$.  Thus we may associate an $L$-function $L(s,\chi)$ to $\chi$ whose Dirichlet series and Euler product are given by
\begin{equation}
L(s,\chi) = \sum_{\mathfrak{a}}\chi(\mathfrak{a})\mathrm{N}\mathfrak{a}^{-s}=\prod_{\mathfrak{p}}(1-\chi(\mathfrak{p})\mathrm{N}\mathfrak{p}^{-s})^{-1}.\notag
\end{equation}
The series and product converge absolutely for $\mathrm{Re}(s)>1$ and can be analytically continued to $\mathbb{C}$ with a functional equation relating $s$ to $1-s$.  (See \cite{Weiss} for further discussion and pertinent references.)

Fix a smooth function $\phi$ whose support is a compact subset of $(-2,2)$.  Let
\[
\hat{\phi}(s)=\int_{\R}\phi(t)e^{st}dt.
\]
Thus $\hat{\phi}(s)$ is entire.  We repeatedly integrate by parts to see that $\hat{\phi}(\sigma+it)\ll_{\phi,k}e^{2|\sigma|}|\sigma+it|^{-k}$ for any integer $k\geq0$.  Let $T\geq1$.  By Fourier inversion, we find that for any $c,x>0$
\[
\phi(T\log x)=\frac{1}{2\pi i T}\int_{c-i\infty}^{c+i\infty}\hat{\phi}(s/T)x^{-s}ds.
\]

\begin{lemma}
\label{lem:sums}
	Let $\epsilon>0$ be sufficiently small (with respect to $[K:\Q]$).  If $\chi\pmod{\qfrak}$ is a (possibly non-primitive) Hecke character, then for any $x>0$ and $T\geq 1$,
	\[
	\Big|\sum_{\afrak}\chi(\afrak)\phi\Big(T\log\frac{\mathrm{N}\afrak}{x}\Big)-\delta(\chi)\kappa\frac{\varphi(\qfrak)}{\mathrm{N}\qfrak} x \frac{\hat{\phi}(1/T)}{T}\Big|\ll_{\epsilon,[K:\Q], \phi}(d_K\mathrm{N}\ffrak_{\chi})^{1/4+\epsilon}(\Norm \qfrak)^\epsilon \sqrt{x}T^{[K:\Q]/4+\epsilon},
	\]
	where $\delta(\chi)=1$ if $\chi$ is trivial and $\delta(\chi)=0$ otherwise, and $\kappa := \mathrm{Res}_{s=1}\zeta_K(s)$.
\end{lemma}
\begin{proof}
	The quantity whose absolute value we wish to bound equals
	\[
	\frac{1}{2\pi i}\int_{1/2-i\infty}^{1/2+i\infty}L(s,\tilde{\chi})\hat{\phi}(s/T)x^s\prod_{\substack{\mathfrak{p}\mid \mathfrak{q} \\ \mathfrak{p}\nmid \mathfrak{f}_{\chi}}}(1-\tilde{\chi}(\mathfrak{p})\mathrm{N}\mathfrak{p}^{-s})ds.
	\]
	Rademacher \cite{Rademacher} proved that if $\chi$ is a primitive character and $\epsilon>0$, then
\[
\Big|\Big(\frac{s-1}{s+1}\Big)^{\delta(\chi)}L(s,\chi)\Big|\ll_{\epsilon,[K:\Q]}(d_K\mathrm{N}\mathfrak{f}_{\chi}(1+|t|)^{[K:\Q]})^{\frac{1-\sigma+\epsilon}{2}}
\]
uniformly in the region $-\epsilon\leq\sigma\leq 1+\epsilon$.  In particular,
\begin{equation}
\label{eqn:residue}
\kappa = \mathrm{Res}_{s=1}\zeta_K(s)\ll_{\epsilon,[K:\Q]}d_K^{\epsilon}.
\end{equation}
Thus the above integral is
\[
\ll_{\epsilon,[K:\Q]}(d_K\mathrm{N}\ffrak_{\chi})^{1/4+\epsilon/2}2^{\omega(\qfrak)}\sqrt{x}\int_{-\infty}^{\infty}\hat{\phi}\Big(\frac{1/2+it}{T}\Big)(1+|t|)^{(1/4+\epsilon/2)[K:\Q]}dt
\]
where $\omega$ denotes the number of different prime ideals that divide a non-zero integral ideal. Using our bound on $\hat{\phi}(s)$, the above expression becomes
\[
\ll_{\epsilon,[K:\Q], \phi}(d_K\mathrm{N}\ffrak_{\chi})^{1/4+\epsilon/2}2^{\omega(\qfrak)}\sqrt{x}\int_{-\infty}^{\infty}\min\Big\{1,\frac{T^{[K:\Q]/4+2}}{(1+|t|)^{[K:\Q]/4+2}}\Big\}(1+|t|)^{(1/4+\epsilon/2)[K:\Q]}dt.
\]

One has $2^{\omega(\qfrak)}\ll_{\epsilon, [K:\Q]} (d_K\Norm \qfrak)^\epsilon$ (cf. Lemma 1.13 in \cite{Weiss}). If $\epsilon$ is sufficiently small with respect to $[K:\Q]$, then the claimed bound follows.
\end{proof}

This simple bound is sufficient to establish a weak form of the large sieve inequality for Hecke characters over $K$, generalizing the classical case over $\Q$.

\begin{proposition}
\label{thm:large_sieve}
Let $\epsilon,x>0$ and $Q,T\geq 1$.  For any complex-valued function $b(\mathfrak{a})$ supported on the integral ideals of $K$, we have that
\begin{align*}
\sum_{\mathrm{N}\mathfrak{q}\leq Q}~\sideset{}{^{\prime}}\sum_{\chi\bmod\mathfrak{q}}\Big|\sum_{x<\mathrm{N}\afrak\leq xe^{1/T}}b(\mathfrak{a})\chi(\mathfrak{a})\Big|^2\ll_{\epsilon,[K:\Q]}(d_K Q)^{\epsilon}\Big(\frac{x}{T}+d_K^{\frac{1}{4}}Q^{\frac{5}{2}}\sqrt{x}T^{\frac{[K:\Q]}{4}+\epsilon}\Big)\sum_{x<\mathrm{N}\afrak\leq xe^{1/T}}|b(\mathfrak{a})|^2.
\end{align*}
\end{proposition}
\begin{proof}
Let $\epsilon,x>0$ and $Q,T\geq 1$.  By a standard application of duality, it suffices to prove that for any sequence of complex numbers $b_{\chi}$ indexed by the primitive Hecke characters $\chi$,
\begin{align}
\label{eqn:large_sieve}
&\sum_{x<\mathrm{N}\afrak\leq xe^{1/T}}\Big|\sum_{\mathrm{N}\mathfrak{q}\leq Q}~\sideset{}{^{\prime}}\sum_{\chi\bmod\mathfrak{q}}b_{\chi}\chi(\mathfrak{a})\Big|^2\notag\\
&\ll_{\epsilon,[K:\Q]}(d_K Q)^{3\epsilon}\Big(\frac{x}{T}+d_K^{1/4}Q^{5/2}\sqrt{x}T^{[K:\Q]/4+\epsilon}\Big)\sum_{\mathrm{N}\mathfrak{q}\leq Q}~\sideset{}{^{\prime}}\sum_{\chi\bmod\mathfrak{q}}|b_{\chi}|^2.
\end{align}

We will prove \eqref{eqn:large_sieve}.  First, we observe that for an appropriate uniform choice of $\phi$, the indicator function $\mathbf{1}_{(x,xe^{1/T}]}(t)$ for the interval $(x,xe^{1/T}]$ is bounded above by $\phi(T\log\frac{t}{x})$. In fact, any smooth $\phi$ dominating the indicator function of $[0,1]$ with compact support in $(-2,2)$ works, and we fix it once and for all. Thus the left hand side of \eqref{eqn:large_sieve} is bounded by
\begin{equation}
\label{eqn:large_sieve_2}
\sum_{\mathfrak{a}}\Big|\sum_{\mathrm{N}\mathfrak{q}\leq Q}~\sideset{}{^{\prime}}\sum_{\chi\bmod\mathfrak{q}}b_{\chi}\chi(\mathfrak{a})\Big|^2 \phi\Big(T\log\frac{\mathrm{N}\afrak}{x}\Big).
\end{equation}
Expanding the square and changing the order of summation, we see that \eqref{eqn:large_sieve_2} equals
\begin{align}
\label{eqn:large_sieve_3}
\sum_{\mathrm{N}\mathfrak{q}_1,\mathrm{N}\mathfrak{q}_2\leq Q}~\sideset{}{^{\prime}}\sum_{\substack{\chi_1\bmod{\mathfrak{q}}_1 \\ \chi_2\bmod{\mathfrak{q}}_2}}&b_{\chi_1}\overline{b_{\chi_2}}\sum_{\mathfrak{a}}\chi_1(\mathfrak{a})\bar{\chi}_2(\mathfrak{a})\phi\Big(T\log\frac{\mathrm{N}\afrak}{x}\Big)\notag\\
&\leq \Big(\max_{\substack{\mathrm{N}\mathfrak{q}_1\leq Q \\ \chi_1\bmod{\mathfrak{q}}_1}}\sum_{\mathrm{N}\mathfrak{q}_2\leq Q}~\sideset{}{^{\prime}}\sum_{\chi_2\bmod{\mathfrak{q}}_2}\Big|\sum_{\mathfrak{a}}\chi_1(\mathfrak{a})\bar{\chi}_2(\mathfrak{a})\phi\Big(T\log\frac{\mathrm{N}\afrak}{x}\Big)\Big|\Big)\sum_{\mathrm{N}\mathfrak{q}\leq Q}~\sideset{}{^{\prime}}\sum_{\chi\bmod\mathfrak{q}}|b_{\chi}|^2.
\end{align}
We now apply Lemma \ref{lem:sums} to \eqref{eqn:large_sieve_3}.  Since 
$$
\varphi(\qfrak):=\Norm \qfrak \prod_{\pfrak|\qfrak}\left(1-\frac{1}{\Norm \pfrak}\right)\leq\mathrm{N}\qfrak
$$ 
for any non-zero integral ideal $\qfrak$ of $K$, we see that \eqref{eqn:large_sieve_3} is
\begin{align}
	\label{eqn:large_sieve_4}
	\ll_{\epsilon,[K:\Q]}\Big(\kappa x\frac{\hat{\phi}(1/T)}{T}+(d_K Q^2)^{1/4+\epsilon}\sqrt{x}T^{[K:\Q]/4+\epsilon}\sum_{\mathrm{N}\qfrak_2\leq Q}\sideset{}{'}\sum_{\chi_2\bmod\qfrak_2}1\Big)\sum_{\mathrm{N}\mathfrak{q}\leq Q}~\sideset{}{^{\prime}}\sum_{\chi\bmod\mathfrak{q}}|b_{\chi}|^2.
\end{align}
For every $Q\geq2$ and every $\epsilon>0$ we have $\sum_{\mathrm{N}\mathfrak{q}_2\leq Q}1\ll_{[K:\Q]}Q^{1+\epsilon}$ by \cite[Lemma 1.12]{Weiss}.  Thus we have the trivial bound
\[
\sum_{\mathrm{N}\mathfrak{q}_2\leq Q}~\sideset{}{'}\sum_{\chi_2\bmod\qfrak_2}1\ll_{\epsilon,[K:\Q]} Q^{2+\epsilon}.
\]
With this bound as well as \eqref{eqn:residue}, we find that  \eqref{eqn:large_sieve_4} is
\[
\ll_{\epsilon,[K:\Q]}(d_K Q)^{3\epsilon}\Big(x\frac{\hat{\phi}(1/T)}{T}+d_K^{1/4}Q^{5/2}\sqrt{x}T^{[K:\Q]/4+\epsilon}\Big)\sum_{\mathrm{N}\mathfrak{q}\leq Q}~\sideset{}{^{\prime}}\sum_{\chi\bmod\mathfrak{q}}|b_{\chi}|^2.
\]
Since $\hat{\phi}(1/T)\asymp_{\phi} 1$ for all $T\geq 1$ and $\phi$ is fixed, we have proved \eqref{eqn:large_sieve} once we re-scale $\epsilon$.
\end{proof}

\begin{proof}[Proof of Proposition \ref{prop:MVT}]
	A result of Gallagher \cite[Theorem 1]{Gallagher} states that for any sequence of complex numbers $a_n$ and any $T\geq 1$, we have
\[
\int_{-T}^{T}\Big|\sum_{n\geq1}a_n n^{-it}\Big|^2 dt\ll T^2\int_0^{\infty}\Big|\sum_{x<n\leq xe^{1/T}}a_n\Big|^2\frac{dx}{x}.
\]
Thus for any function $b(\afrak)$ supported on the ideals of $K$,
\[
\sum_{\mathrm{N}\qfrak\leq Q}~\sideset{}{'}\sum_{\chi\bmod\qfrak}\int_{-T}^{T}\Big|\sum_{\afrak}b(\afrak)\chi(\afrak)\mathrm{N}\afrak^{-it}\Big|^2 dt\ll T^2\int_0^{\infty}\sum_{\mathrm{N}\qfrak\leq Q}~\sideset{}{'}\sum_{\chi\bmod\qfrak}\Big|\sum_{x<\mathrm{N}\afrak\leq xe^{1/T}}b(\afrak)\chi(\afrak)\Big|^2\frac{dx}{x}
\]
By Proposition \ref{thm:large_sieve}, we see that the right hand side of the above display is
\begin{align*}
&\ll_{\epsilon,[K:\Q]}(d_K Q)^{\epsilon}T^2\int_0^{\infty}\Big(\frac{x}{T}+d_K^{1/4}Q^{5/2}\sqrt{x}T^{[K:\Q]/4+\epsilon}\Big)\sum_{x<\mathrm{N}\afrak\leq xe^{1/T}}|b(\afrak)|^2\frac{dx}{x}\\
&=(d_K Q)^{\epsilon}\sum_{\afrak}|b(\afrak)|^2\Big(T\int_{e^{-1/T}\mathrm{N}\afrak}^{\mathrm{N}\afrak} dx+T^{[K:\Q]/4+2+\epsilon}d_K^{1/4}Q^{5/2}\int_{e^{-1/T}\mathrm{N}\afrak}^{\mathrm{N}\afrak}x^{-1/2}dx\Big)\\
&\ll_{\epsilon,[K:\Q]}(d_K Q)^{\epsilon}\sum_{\afrak}|b(\afrak)|^2 \mathrm{N}\afrak\Big(1+T^{[K:\Q]/4+1+\epsilon}d_K^{1/4}Q^{5/2}\mathrm{N}\afrak^{-1/2}\Big).
\end{align*}
Choose $b(\afrak)$ to be supported on the prime ideals $\pfrak$ with $\mathrm{N}\pfrak>y$.  Then the above display is
\begin{align*}
&\ll_{\epsilon,[K:\Q]}(1+T^{[K:\Q]/4+1+\epsilon}d_K^{1/4}Q^{5/2}y^{-1/2}\Big)(d_K Q)^{\epsilon}\sum_{\mathrm{N}\pfrak>y}|b(\pfrak)|^2 \mathrm{N}\pfrak \\
&\ll_{\epsilon,[K:\Q]}(d_K Q)^{\epsilon}\sum_{\mathrm{N}\pfrak>y}|b(\pfrak)|^2 \mathrm{N}\pfrak,
\end{align*}
as desired.
\end{proof}


\end{document}